\documentclass[11pt]{amsart}
\usepackage{bm,amssymb,amscd,url,mathrsfs}
\usepackage{graphics,epsfig,color}

\usepackage[letterpaper,asymmetric,left=1.0in,right=1.0in,top=1.0in,bottom=1.
0in,bindingoffset=0.0in]{geometry}

\usepackage{mathrsfs}

\newcommand{\MATRIXGP}{{\Gamma}}
\newcommand{\AUTOGROUP}{{\mathcal G}}
\newcommand{\GRID}{{\mathbb{G}}}

\newcommand{\goodparams}{\mathbb{C}^4_{\rm good}}
\newcommand{\BIGFATOU}{{\mathcal{F}}}

\renewcommand{\tilde}{\widetilde}

\newcommand{\C}{\mathbb{C}}
\newcommand{\N}{\mathbb{N}}
\newcommand{\Z}{\mathbb{Z}}

\newcommand{\fol}{\mathcal{D}}
\newcommand{\bbCP}{\mathbb{CP}}

\newcommand{\discrete}{\mathcal{D}}
\newcommand{\nondiscrete}{\mathcal{N}}
\newcommand{\parameterneighborhood}{\mathcal{P}}
\newcommand{\infinityneighborhood}{\mathcal{U}}

\newcommand{\ourboundary}{\mathscr{B}}

\newtheorem{theorem}{\bf Theorem}[section]
\newtheorem{proposition}[theorem]{\bf Proposition}
\newtheorem{definitionproposition}[theorem]{\bf Definition/Proposition}

\newtheorem*{theorem*}{\bf Theorem}
\newtheorem*{theoremA}{\bf Theorem A}
\newtheorem*{theoremB}{\bf Theorem B}
\newtheorem*{theoremC}{\bf Theorem C}
\newtheorem*{theoremD}{\bf Theorem D}
\newtheorem*{theoremE}{\bf Theorem E}
\newtheorem*{corollaryE}{\bf Corollary to Theorem E}
\newtheorem*{theoremF}{\bf Theorem F}
\newtheorem*{corollaryF}{\bf Corollary to Theorem F}
\newtheorem*{theoremG}{\bf Theorem G}
\newtheorem*{corollaryG}{\bf Corollary to Theorem G}
\newtheorem*{theoremH}{\bf Theorem H}
\newtheorem*{theoremK}{\bf Theorem K}

\newtheorem{lemma}[theorem]{\bf Lemma}
\newtheorem{corollary}[theorem]{\bf Corollary}
\newtheorem{remark}[theorem]{\bf Remark}

\newtheorem*{2GeneralRemark}{\bf Remark}

\begin{document}

\title[Dynamics on Character Varieties, Fatou/Julia, and Painlev\'e 6]{Dynamics of groups of automorphisms of character varieties and Fatou/Julia decomposition for Painlev\'e~6}

\author{Julio Rebelo and Roland Roeder}
\date{\today}

\thanks{}


\maketitle

\begin{abstract}
We study the dynamics of the group of holomorphic automorphisms of the affine
cubic surfaces
\begin{align*}
S_{A,B,C,D} = \{(x,y,z) \in \C^3 \, : \, x^2 + y^2 + z^2 +xyz = Ax + By+Cz+D\},
\end{align*}
where $A,B,C,$ and $D$ are complex parameters.  We focus on a finite index
subgroup $\AUTOGROUP_{A,B,C,D} < {\rm Aut}(S_{A,B,C,D})$ whose action not only
describes the dynamics of Painlev\'e 6 differential equations but also arises
naturally in the context of character varieties. We define the Julia and Fatou
sets of this group action and prove that there is a dense orbit in the Julia
set.  In order to show that the Julia set is ``large'' we consider a second
dichotomy, between locally discrete and locally non-discrete dynamics.  For an
open set in parameter space, $\parameterneighborhood \subset \C^4$, we show
that there simultaneously exist an open set in $S_{A,B,C,D}$ on which
$\AUTOGROUP_{A,B,C,D}$ acts locally discretely and a second open set in
$S_{A,B,C,D}$ on which $\AUTOGROUP_{A,B,C,D}$ acts locally non-discretely.  Their
common boundary contains an invariant set $\ourboundary_{A,B,C,D}$ of
topological dimension $3$.  After removing a countable union of real-algebraic
hypersurfaces from $\parameterneighborhood$ we show that $\AUTOGROUP_{A,B,C,D}$
simultaneously exhibits a non-empty Fatou set and also a Julia set having
non-trivial interior.  The open set $\parameterneighborhood$ contains a natural
family of parameters previously studied by Dubrovin-Mazzocco.

The interplay between the Fatou/Julia dichotomy and the locally
discrete/non-discrete dichotomy plays a major theme in this paper and seems
bound to play an important role in further dynamical studies of holomorphic
automorphism groups.
\end{abstract}

\section{Introduction}

\subsection{Setting}
Let $A$, $B$, $C$, and~$D$ be fixed complex parameters and consider the affine cubic surface
\begin{align*}
S_{A,B,C,D} = \{(x,y,z) \in \C^3 \, : \, x^2 + y^2 + z^2 +xyz = Ax + By+Cz+D\}.
\end{align*}
It is known that the corresponding family of projective surfaces contains all smooth cubic surfaces (see \cite{toledo}
or Section~\ref{cubics} for further detail). Note that every
line parallel to the $x$-axis
intersects the surface $S_{A,B,C,D}$ at two points (counted with multiplicity) and one can therefore
define an involution $s_x: S_{A,B,C,D} \rightarrow S_{A,B,C,D}$ that switches them:
\begin{align}\label{EQN:SX}
s_x\left(x, y, z\right) = \left(-x-yz+A,  y, z \right).
\end{align}
Two further involutions $s_y: S_{A,B,C,D} \rightarrow S_{A,B,C,D}$ and $s_z: S_{A,B,C,D} \rightarrow S_{A,B,C,D}$
are defined analogously by means of lines parallel to the $y$-axis and $z$-axis, respectively.

Consider the group
\begin{align}
\AUTOGROUP^\pm \equiv \AUTOGROUP^\pm_{A,B,C,D} = \langle s_x, s_y, s_z \rangle\leq {\rm Aut}(S_{A,B,C,D}),
\end{align}
where ${\rm Aut}(S_{A,B,C,D})$ denotes the group of all (algebraic) holomorphic diffeomorphisms of $S_{A,B,C,D}$.
For a generic choice of parameters one has $\AUTOGROUP^\pm_{A,B,C,D} = {\rm Aut}(S_{A,B,C,D})$ and, in general,
$\AUTOGROUP^\pm_{A,B,C,D}$ is a subgroup of ${\rm Aut}(S_{A,B,C,D})$ of index at most $24$; See \cite[Theorem 3.1]{cantat-2} and also \cite{huti}.

Consider also the index two subgroup
\begin{align*}
\AUTOGROUP \equiv \AUTOGROUP_{A,B,C,D}= \langle g_x,  g_y, g_z \rangle  < \AUTOGROUP^\pm.
\end{align*}
where $g_x = s_z \circ s_y$, $g_y = s_x \circ s_z$, and $g_z = s_y \circ s_x$.

The dynamics of the action of groups $\AUTOGROUP^\pm_{A,B,C,D}$ and
$\AUTOGROUP_{A,B,C,D}$ on $S_{A,B,C,D}$ and their individual elements have several
deep connections, including to the dynamics of mapping class groups on character varieties, to the
monodromy of the Painlev\'e~6 differential equation, and to the aperiodic
Schr\"odinger equation; see, for example, \cite{cantat-1} for a nice description of these
connections.
For this reason there is an extensive body of previous works on this dynamical system. In Section \ref{SEC:RELATED_WORKS} we will provide
a sample, focusing on those that we consider most closely related to the present paper.
We will also provide further details about the connections with dynamics on character varieties and with the 
Painlev\'e~6 differential equation in Section \ref{SEC:MOTIVATIONS}.

\vspace{0.1in}
\noindent
{\bf Main goal of our work:}
Our paper is devoted to questions about the ``pointwise complex dynamics
of the whole group'', i.e.\  to the orbits of individual points, their closures,
and more generally to the nature of subsets of the complex surface $S_{A,B,C,D}$
that are invariant under $\AUTOGROUP^\pm$ and $\AUTOGROUP$.   

\vspace{0.1in}
\noindent
We focus on the ``chaotic'' part of the dynamics {\rm (}i.e.\ the 
dynamics on the Julia set $\mathcal{J}_{A,B,C,D}$, as defined in Section \ref{SUBSEC:FATOU_JULIA}{\rm )} which is therefore complementary
to the extensive work previously done by Bowditch, Tan and his collaborators,
and others about the domains on which the dynamics is properly discontinuous.   Indeed, our work can be seen 
as an effort to answer the questions posed by Bowditch  immediately below Corollary 5.6  on p. 728 of \cite{bowditch}.
Let us point out that the complex dynamics on the $SL(2,\mathbb{C})$ character
varieties  (i.e. of $\AUTOGROUP_{A,B,C,D}$ on $S_{A,B,C,D}$) is described as
being ``extremely mysterious and non-trivial'' by Goldman \cite[p.
461]{goldman-1} and by Tan, Wong, and Zhang   \cite[p. 762]{TWZ}.

\vspace{0.1in}
It is important to remark that we are considering the action
of $\AUTOGROUP_{A,B,C,D}$ on the (non-compact) affine surface $S_{A,B,C,D} \subset
\mathbb{C}^3$.  Indeed the elements of $\AUTOGROUP_{A,B,C,D}$ extend only as
birational mappings of the compactification $\overline{S}_{A,B,C,D} \subset
\mathbb{CP}^3$, with both indeterminate and super-attracting/collapsing behavior
at infinity (see Section \ref{dynamics_near_infinty}).   This lack of a good compact domain 
for $\AUTOGROUP_{A,B,C,D}$ makes several aspects of this dynamical system rather
challenging.   It also seems to rule out  ergodic-theoretic
methods like those recently used by Cantat and Dujardin \cite{cantat-dujardin}
to study automorphism groups of compact surfaces.  This will be discussed
further in Section \ref{SUBSEC:CANTAT_DUJARDIN}.

\subsection{Some ``preferred'' parameters}
Throughout the paper we will refer to the following specific parameters and parametric families.

\vspace{0.05in}
\noindent
{\bf Markoff Parameters:} $(A,B,C,D) = (0,0,0,0)$, as discussed in \cite{bowditch} and \cite{series}.

\vspace{0.05in}
\noindent
{\bf Picard Parameters:} $(A,B,C,D) = (0,0,0,4)$. For these parameters,
Picard proved that the
Painlev\'e equation has explicit first integrals and countably many algebraic
solutions.  This is related to the curious fact that the action of
$\AUTOGROUP_{0,0,0,4}$ is semi-conjugate to an action on $(\mathbb{C} \setminus
\{0\})^2$ by monomial mappings.  In particular, for these parameters everything
can be computed rather explicitly.

\vspace{0.05in}
\noindent
{\bf Punctured Torus Parameters:} $(A,B,C,D) = (0,0,0,D)$ for any $D \in \mathbb{C}$. These
parameters correspond to dynamics on the character variety of the once
punctured torus; see, e.g.\  \cite[Sec. 1.1]{cantat-1}. For real $D$ and the
corresponding real surfaces, this is the family studied by Goldman~\cite{goldman-1}.

\vspace{0.05in}
\noindent
{\bf Dubrovin-Mazzocco Parameters:}  This is a
real $1$-parameter family studied by Dubrovin and  Mazzocco \cite{dubrovinmazzocco} which seems
to play a significant role in several problems related to Mathematical-Physics and, in
particular, on the study of Frobenius manifolds.
In our notations, the Dubrovin-Mazzocco parameters correspond to
\begin{align}\label{EQN:DM_PARAMETERS}
A(a) = B(a) = C(a) = 2a+4, \quad \mbox{and} \quad D(a) = -(a^2 + 8a +8)
\end{align}
for $a \in (-2,2)$.

Notice that both the Markoff and Picard parameters are included within the Punctured Torus Parameters.
Meanwhile, the Picard parameters are in the closure of the Dubrovin-Mazzocco parameters, corresponding to
$a=-2$, but the Markoff parameters are not.

\subsection{Two relevant dynamical dichotomies}\label{SUBSEC:2DICHOTOMIES}
With the goal of producing interesting invariant sets and finding points with complicated orbit closures
we introduce two dynamically invariant dichotomies.

A consequence of our results is that, in general, the action of
$\AUTOGROUP$ (or of $\AUTOGROUP^\pm$) is genuinely
non-linear, for example, it does not preserve any rigid geometric structure in the sense of Gromov~\cite{gromov};
see Remark~\ref{genuinelynonlinear}. Whereas this
was mostly expected, this action
still appears to share some
basic properties/issues with actions of countable subgroups of finite
dimensional Lie groups. This typically happens on some (proper) open
subsets of the surface $S_{A,B,C,D}$ and, for this reason, the
notions of {\em locally discrete} vs.\ {\em locally non-discrete dynamics} of $\AUTOGROUP$ will come in handy.
Meanwhile, to deal with the non-linear nature of the global dynamics we will adapt the {\em Fatou/Julia theory}
to the group $\AUTOGROUP$.  The core of this paper lies in the interplay between these two dichotomies.

\subsection{Locally non-discrete/discrete dichotomy}\label{SUBSEC:LOCALLY_DISCR}
Let $M$ be a (possibly open) connected complex manifold and consider a group $G$ of holomorphic diffeomorphisms of $M$.
The group $G$ is said to be {\it locally non-discrete} on an open set $U \subset M$
if there is a sequence of maps $\{ f_n \}_{n=0}^{\infty} \in G$ satisfying the following conditions (see for example \cite{REBELO_REIS}):
\begin{enumerate}
 \item For every $n$, $f_n$ is different from the identity.
  \item The sequence of maps $f_n$ converges uniformly to the identity on compact subsets of $U$.
\end{enumerate}
If there is no such sequence $f_n$ on $U$ we say that $G$ is {\it locally discrete} on $U$.

Remark that for an action by a finite dimensional Lie group, local
non-discreteness on some open set implies that the corresponding sequence of
elements converges to the identity on all of $M$, i.e. that the action is
globally non-discrete.  However, in our context the non-linearity of the
mappings allow for local non-discreteness to occur on a proper open subset $U
\subset M$ in such a way that it does not extend beyond $U$.

For any choice of parameters $(A,B,C,D)$ let
\begin{flalign*}
\nondiscrete_{A,B,C,D} = \{p \in S_{A,B,C,D} \, : \, \AUTOGROUP_{A,B,C,D} \, \mbox{is locally non-discrete on an open neighborhood $U$ of $p$}\}, &&
\end{flalign*}
and let 
$\discrete_{A,B,C,D} =   S_{A,B,C,D} \setminus \nondiscrete_{A,B,C,D}$.
We will refer to $\nondiscrete_{A,B,C,D}$ as the ``locally non-discrete locus'' and to $\discrete_{A,B,C,D}$ as the ``locally discrete locus''.
By definition, $\nondiscrete_{A,B,C,D}$ is open, $\discrete_{A,B,C,D}$ is closed, and each is invariant under $\AUTOGROUP_{A,B,C,D}$.

Note that $\nondiscrete_{A,B,C,D}$ can be empty for certain parameter values;
indeed it is for the Picard Parameters (Theorem D(ii), below).  We do not
know if $\discrete_{A,B,C,D}$ can be empty for any parameter values.

\subsection{Fatou/Julia dichotomy}\label{SUBSEC:FATOU_JULIA}

For any point $p \in S_{A,B,C,D}$ we 
denote the orbit of $p$ under $\AUTOGROUP$ by
\begin{align*}
\AUTOGROUP(p) = \{ \, \gamma(p) \; : \; \, \gamma \in \AUTOGROUP \, \}.
\end{align*}
The {\em Fatou set} of the group action $\AUTOGROUP$ is defined as
\begin{align*}
{\mathcal F}_{A,B,C,D} = \{p \in S_{A,B,C,D} \, : \, \mbox{$\AUTOGROUP$ forms a normal family in some open neighborhood of $p$}\}.
\end{align*}
Naturally the condition of being a normal family means that every sequence of maps as indicated must have a
convergent subsequence (for the topology of uniform convergence on compact subsets -- compact-open topology).
However, since $S_{A,B,C,D}$ is {\it open}, sequences of maps avoiding compact sets are expected to
arise as well. It is then convenient to make the notion of converging subsequence accurate
by means of the following definition: a sequence of maps (diffeomorphisms onto their images) $f_i$ is said to
{\it converge to infinity}\, on an open set $U \subset S_{A,B,C,D}$ if for every compact set
$\overline{V} \subset U$ and every
compact set $K \subset S_{A,B,C,D}$, there are only finitely many maps $f_i$ such that
$f_i (\overline{V}) \cap K \neq \emptyset$. Sequences converging to infinity are to be included
in the definition of normal family used above. In particular, if the sequence formed by all elements
of $\AUTOGROUP$ converge to infinity on some open set $U \subset S_{A,B,C,D}$, then $U$ is contained
in the Fatou set of $\AUTOGROUP$.

\begin{remark}
	{\rm According to Proposition~\ref{PROP:FATOU_COMP_ARE_HYPERBOLIC}, any component of the Fatou set ${\mathcal F}_{A,B,C,D}$
		is Kobayashi hyperbolic and, by exploiting this condition, it can be shown that our definition amounts to requiring
		the family to be normal if it is viewed as (continuous) maps from $S_{A,B,C,D}$ with values in its one-point compactification.}
\end{remark}

The {\em Julia set} of the group action $\AUTOGROUP$ is defined as
\begin{align*}
{\mathcal J}_{A,B,C,D} = S_{A,B,C,D} \setminus {\mathcal F}_{A,B,C,D}.
\end{align*}
It follows from the definitions that ${\mathcal
F}_{A,B,C,D}$ is open while ${\mathcal J}_{A,B,C,D}$ is closed.  Furthermore
both sets are invariant under $\AUTOGROUP$.  
(For some parameters $S_{A,B,C,D}$ may be singular, but this is not an issue: we will see in
Remark~\ref{REM:SINGULAR_POINTS_JULIA} that such singular points are always in 
${\mathcal J}_{A,B,C,D}$.)

It is worth emphasizing  that the Julia set is non-empty for every choice of parameters; see,
for example, Lemma~\ref{COR:GX_JULIA}.  
With more effort one can show that for every choice of parameters $(A,B,C,D)$
the group $\AUTOGROUP_{A,B,C,D}$ contains elements $f$ having positive entropy and exhibiting dynamics quite
similar to that of H\'enon maps; see \cite{IU_ERGODIC,cantat-1,cantat-2}.   The Julia set associated with iteration of each
such individual mapping is a subset of ${\mathcal J}_{A,B,C,D}$, including the collection of all saddle-type periodic points
of each such mapping.
However, the Fatou set 
can be empty for some parameter values; indeed this happens for the Picard Parameters (Theorem D(i), below).

\subsection{Main results}

Classical results from the holomorphic dynamics of rational maps of the Riemann sphere assert that there is a dense orbit in the Julia set (topological transitivity) and that repelling periodic points are dense in the Julia set.
We search for analogous statements for the action of $\AUTOGROUP_{A,B,C,D}$ on $S_{A,B,C,D}$.

\begin{theoremA}
For any parameters $(A,B,C,D)$ there is a point $p \in {\mathcal J}_{A,B,C,D}$ such that
\begin{align*}
\overline{\AUTOGROUP(p)} = {\mathcal J}_{A,B,C,D},
\end{align*}
i.e., there is a dense orbit of $\AUTOGROUP$ in ${\mathcal J}_{A,B,C,D}$.
\end{theoremA}

In the setting of group actions, the natural analog of having a dense set of
repelling periodic points consists of looking for a dense set ${\mathcal
J}^*_{A,B,C,D} \subset {\mathcal J}_{A,B,C,D}$ such that each point $p \in {\mathcal
J}^*_{A,B,C,D}$ has 
an element of its stabilizer whose derivative at $p$ is hyperbolic.  In Theorem~D, below, we will see that
this does not hold for the Picard Parameters $(0,0,0,4)$.  We leave it as an
open question to characterize for which parameters $(A,B,C,D)$, if any, the set
${\mathcal J}^*_{A,B,C,D}$ is dense in ${\mathcal J}_{A,B,C,D}$.

A more modest statement is obtained by 
replacing ``hyperbolic derivative'' by derivative
conjugate to a ``shear map''. This is the content of
Theorem~B below.

\begin{theoremB}
For any choice of parameters $(A,B,C,D)$
there is a dense set ${\mathcal J}^\#_{A,B,C,D} \subset {\mathcal J}_{A,B,C,D}$ such that for every
$p \in {\mathcal J}^\#_{A,B,C,D}$ there exists
$\gamma \in \AUTOGROUP$ such that $\gamma(p) = p$ and
\begin{align*}
D\gamma(p) \qquad \mbox{is conjugate to} \qquad \left[\begin{array}{cc} 1 & 1 \\ 0 & 1\end{array}\right].
\end{align*}
\end{theoremB}

Often parametric families of rational maps of the Riemann sphere have some mappings with connected Julia set and other mappings with disconnected Julia set. In our context we have 
a slightly surprising general topological property of Julia sets, namely:

\begin{theoremC}
For any parameters $A,B,C,D$ the Julia set ${\mathcal J}_{A,B,C,D}$ is connected.
\end{theoremC}

Let us now transition from results that hold for all parameters to results that only hold for certain parameters.  

\begin{theoremD} For the Picard Parameters $(A,B,C,D) = (0,0,0,4)$ we have:
\begin{itemize}
\item[(i)] $\mathcal{J}_{0,0,0,4} = S_{0,0,0,4}$ and consequently $\mathcal{F}_{0,0,0,4} = \emptyset$,
\item[(ii)] The action of $\AUTOGROUP_{0,0,0,4}$ is locally discrete on any open subset of $S_{0,0,0,4}$, and
\item[(iii)] The closure of the set of points $\mathcal{J}^*_{0,0,0,4}$ that have hyperbolic stabilizers 
is contained in $S_{0,0,0,4} \cap [-2,2]^3$ and hence is a proper subset of $\mathcal{J}_{0,0,0,4} = S_{0,0,0,4}$.
\end{itemize}
\end{theoremD}


In contrast to the Picard parameters, for which the corresponding Fatou set is
empty, Teichm\"uller theory can be applied to show that the Fatou set is non-empty
for certain values of the parameters, including the Markov parameters $(A,B,C,D) = (0,0,0,0)$.
A more powerful approach stems from 
issues related to the Bowditch Conjecture and the Bowditch BQ
Conditions, as introduced by Bowditch \cite{bowditch} and Tan, Wong, and Zhang
\cite{TWZ} and later studied by Maloni, Palesi, and Tan \cite{MPT},  Hu, Tan, and Zhang
\cite{Hu}, and several others.  See Section~\ref{SUBSEC:TEICH} for more details.
Their methods can be adapted to prove the following theorem:

\begin{theoremE}  The Fatou set $\mathcal{F}_{A,B,C,D}$ is non-empty for any $(A,B,C,D)$ in an open neighborhood in $\mathbb{C}^4$ of
\begin{itemize}
\item[(1)] any punctured torus parameters $(0,0,0,D)$ for $D \neq 4$, and 
\item[(2)] any Dubrovin-Mazzocco Parameter $(A(a),B(a),C(a),D(a))$, where $a \in (-2,2)$.
\end{itemize}
\end{theoremE}
\noindent

\noindent
We will denote the subset of the Fatou set obtained in the proof of Theorem E
by $V_{\rm BQ} \equiv V_{\rm BQ}(A,B,C,D)$ and call it the {\em Bowditch set} because any point $p \in V_{\rm BQ}$
satisfies the BQ Conditions.  (See Section \ref{SUBSEC:TEICH} for the precise statement of these conditions.)   As a matter of fact, the action of $\AUTOGROUP$ on
$V_{\rm BQ}$ is properly discontinuous, as was shown in \cite{TWZ} and
\cite{MPT}.  (We will also show this in the proof of Theorem F, below.)  
The Bowditch set seems to be of considerable interest and, in
particular, the following rephrasing of Statement (1) in Theorem E seems to be
new:

\begin{corollaryE}
Consider the punctured torus parameters $(0,0,0,D)$ where $D \in \mathbb{C}$. The Bowditch set $V_{\rm BQ}(0,0,0,D)$
is non-empty if and only if $D \neq 4$.
\end{corollaryE}

However, let us note that we recently discovered that Statement (2) of Theorem E is an immediate consequence
of a stronger statement \cite[Theorem 5.3]{MPT}.   
In general, a complete classification of all $(A,B,C,D) \in \mathbb{C}^4$ for which
$V_{\rm BQ}(A,B,C,D)$ is non-empty seems to be a delicate question.  

\vspace{0.1in}

Because the Fatou set can be non-empty it is interesting to determine how
``large'' the Julia set is, especially if one wishes to apply Theorems A and B.
For this reason we study the interplay of the Fatou/Julia dichotomy with the
locally non-discrete/discrete dichotomy.  We will first discuss the locally
non-discrete/discrete dichotomy and then relate it to the Fatou/Julia
dichotomy. 

In Section~\ref{SEC:LOCALLY_NONDISCRETE},
we will show the existence of
a very large set of parameters $(A,B,C,D)$ for which
$\AUTOGROUP_{A,B,C,D}$ is locally non-discrete on some open subset of the surface
$S_{A,B,C,D}$.
Later in Section \ref{SEC:PROOF_OF_THEOREM_G} we obtain as a consequence 
the following theorem about the coexistence of both locally discrete and
locally non-discrete dynamics:


\begin{theoremF} There is an open neighborhood $\parameterneighborhood \subset
\mathbb{C}^4$ of the Markoff Parameters $(0,0,0,0)$ and of each of the
Dubrovin-Mazzocco Parameters $(A(a),B(a),C(a),D(a))$, where $a \in (-2,2)$, with
the following property.  

For any $(A,B,C,D) \in \parameterneighborhood$ there are disjoint non-empty opens sets $U, V_{\rm BQ} \subset S_{A,B,C,D}$ 
such that:
\begin{enumerate}
\item The action of $\AUTOGROUP_{A,B,C,D}$ is locally non-discrete on $U$; i.e.\ $U \subset \nondiscrete_{A,B,C,D}$.
\item The action of $\AUTOGROUP_{A,B,C,D}$ is locally discrete on any open neighborhood of any point from $V_{\rm BQ}$, i.e.\ $V_{\rm BQ} \subset \discrete_{A,B,C,D}$.
Indeed, the action of $\AUTOGROUP_{A,B,C,D}$ on $V_{\rm BQ}$ is properly discontinuous.
\end{enumerate}
There are non-commuting pairs of element of $\AUTOGROUP_{A,B,C,D}$ both of which are arbitrarily close to the identity on $U$.
\end{theoremF}



\begin{corollaryF}
For every choice of parameters $(A,B,C,D) \in \parameterneighborhood$ there there is a set
\begin{align*}
\ourboundary_{A,B,C,D} \subset \partial \nondiscrete_{A,B,C,D} = \partial \discrete_{A,B,C,D}
\end{align*}
that has topological dimension equal to three and is invariant under $\AUTOGROUP_{A,B,C,D}$.
\end{corollaryF}

\begin{2GeneralRemark}
{\rm 
The invariant set $\ourboundary_{A,B,C,D} \subset S_{A,B,C,D}$ of topological dimension~$3$
``persists'' over the open subset of parameters $\parameterneighborhood \subset \mathbb{C}^4$.
The existence of persistent invariant sets of topological
dimension~$3$ for the action of a large group ($\AUTOGROUP$ is free on two generators)
hints at a fractal nature for $\ourboundary_{A,B,C,D}$. In fact, in the smooth category, by combining
general position arguments with Baire's theorem, it can be shown that two ``generic'' diffeomorphisms
cannot simultaneously preserve a smooth submanifold. Whereas the same methods are not immediately available
in the holomorphic setting, the same conclusion seems likely to still hold true.


}
\end{2GeneralRemark}

With the notation of Theorem~F, we expect that for each $(A,B,C,D) \in
\parameterneighborhood$ we have $U \subset \mathcal{J}_{A,B,C,D}$. However, at present, we
can only prove this under one further (weak) assumption, namely:
\begin{itemize}
\item[(P)] any fixed point of any $\gamma \in \AUTOGROUP_{A,B,C,D} \setminus \{{\rm id}\}$ is in
$\mathcal{J}_{A,B,C,D}$.
\end{itemize}
We prove in Proposition~\ref{PROP:GOODPARAMS_IN_COMPLETMENT_HYPERSURFACES} that
there is a countable union of real-algebraic hypersurfaces $\mathcal{H} \subset
\mathbb{C}^4$ such that Hypothesis (P) holds if $(A,B,C,D)~\in~\mathbb{C}^4
\setminus \mathcal{H}$.

\begin{theoremG} Let $\parameterneighborhood \subset \mathbb{C}^4$ be the open neighborhood 
of the Markoff Parameters and of the Dubrovin-Mazzocco parameters, $a \in (-2,2)$, 
given in Theorem~F and let $\mathcal{H} \subset \mathbb{C}^4$ be the countable union of real-algebraic hypersurfaces
provided by Proposition~\ref{PROP:GOODPARAMS_IN_COMPLETMENT_HYPERSURFACES}.
For any $(A,B,C,D) \in \parameterneighborhood \setminus \mathcal{H}$ we have:
\begin{align*}
U \subset \mathcal{J}_{A,B,C,D} \qquad \mbox{and} \qquad V_{\rm BQ} \subset \mathcal{F}_{A,B,C,D}.
\end{align*}
Here, $U$ and $V_{\rm BQ}$ are the (non-empty) open subsets of $S_{A,B,C,D}$ from the statement of Theorem F.
\end{theoremG}

\begin{corollaryG}
For any $(A,B,C,D) \in \parameterneighborhood \setminus \mathcal{H}$ there exists $p \in \mathcal{J}_{A,B,C,D}$ 
whose orbit closure has interior but is not all of the surface $S_{A,B,C,D}$.
\end{corollaryG}

\begin{remark}
{\rm 
Any $\gamma \in \AUTOGROUP_{A,B,C,D}$
preserves an invariant holomorphic $2$-form $\Omega$ whose equation is given in Section~\ref{SEC:VOLUME_FORM}.   This forces that the eigenvalues
of $D\gamma$ at any fixed point $p \in S_{A,B,C,D}$ are resonant, $\lambda_1 = 1/\lambda_2$,
making it rather unlikely that such a fixed point is in the Fatou set.

Indeed, the possibility that Hypothesis (P), above, fails for any
parameters $(A,B,C,D)$ is a known challenge in holomorphic dynamics.   It is
explained at length in the paper \cite{mcmullen2} by McMullen, see especially
the remark on p. 220 of that paper.
}
\end{remark}

\subsection{Strategy for Proof of Theorem G}
\label{SUBSEC:STRATEGY}
Let us briefly describe the strategy for proving Theorem~G, which is arguably the most elaborate result in our
paper. It is straightforward to prove that any Fatou component $V$ is Kobayashi
hyperbolic. In particular, if $\AUTOGROUP$ is locally non-discrete on any open subset of $V$ then
it is locally non-discrete on all of $V$. The idea is then to show that a region $U$ where
$\AUTOGROUP$ induces a ``complicated enough'' locally non-discrete dynamics is not compatible with the
structure of a (Kobayashi hyperbolic) Fatou set. This region must hence be contained in $\mathcal{J}_{A,B,C,D}$
and this yields Julia sets with non-empty interior.

In order to rule out the possibility that this region $U$ intersects an unbounded Fatou component,
the following theorem will be needed.   Note that even though we typically consider $s_x$ as an automorphism of a given surface $S_{A,B,C,D}$,
Equation (\ref{EQN:SX}) actually defines a polynomial automorphism of $\mathbb{C}^3$ which depends only on the parameters $A, B$, and $C$.  The
same holds for $s_y$ and $s_z$ and we will denote the group of automorphisms of $\mathbb{C}^3$ generated by these three mappings with $\AUTOGROUP^\pm_{A,B,C}$.
The index two subgroup generated by $g_x, g_y,$ and $g_z$ (considered as automorphisms of $\mathbb{C}^3$) will be denoted
by $\AUTOGROUP_{A,B,C}$.

\begin{theoremH}
Suppose that for some parameters $A,B,C$ there is a point $p \in \mathbb{C}^3$
and $\epsilon > 0$ such that for any two  vertices $v_i \neq v_j \in
        \mathcal{V}_\infty$, $i \neq j$, there is a hyperbolic element $\gamma_{i,j} \in
        \AUTOGROUP_{A,B,C}$ satisfying:
        \begin{itemize}
                \item[(A)] ${\rm Ind}(\gamma_{i,j}) = v_i$ and ${\rm Attr}(\gamma_{i,j}) = v_j$, and
                \item[(B)] $ \sup_{z \in B_{\epsilon} (p)} \Vert \gamma_{i,j}(z) - z \Vert < K(\epsilon)$.
        \end{itemize}
        Then, for any $D$, we have that $B_{\epsilon/2}(p) \cap S_{A,B,C,D}$ is disjoint from any unbounded
        Fatou components of $\AUTOGROUP_{A,B,C,D}$.
        Here, $K(\epsilon) > 0$ denotes the constant given in Proposition~\ref{PROP:CONV_TO_ID}.
\end{theoremH}

We refer the reader to Section \ref{dynamics_near_infinty} for the definition of $\mathcal{V}_\infty$ and to Proposition~\ref{PROP:DESCRIPTION_HYPERBOLIC_ELEMENTS}
for the definition of hyperbolic element $\gamma$ along with the corresponding points
${\rm Ind}(\gamma)$ and ${\rm Attr}(\gamma)$.  Hypothesis~(A) requires that the six elements $\gamma_{i,j}$
have sufficiently rich ``combinatorial behavior'' at infinity while Hypothesis (B)
requires that these six elements are sufficiently close to the identity on the ball $\mathbb{B}_\epsilon(p)$.
Note that the
conditions of Theorem H are explicit and easy to check.  In particular,
for any $(A,B,C,D) \in \parameterneighborhood$ they imply that $U$ is disjoint from any
unbounded Fatou component.

The idea of the proof of Theorem H is that
if $\mathbb{B}_\epsilon(p)$ were in an unbounded
Fatou component~$V$ then we use local non-discreteness to produce a sequence of elements
converging uniformly on compact subsets of $V$ to the identity and we use the prescribed 
combinatorial behavior at infinity to show that this same sequence of elements
sends compact subsets
of $V$ uniformly to infinity.

Having ruled out the possibility that an unbounded Fatou component intersects $U$
we then use the following theorem to prove that no bounded Fatou component intersects $U$.

\begin{theoremK}
Suppose that $(A,B,C,D) \in \mathbb{C}^4 \setminus \mathcal{H}$ and that $V$ is a bounded Fatou component
for $\AUTOGROUP_{A,B,C,D}$.  Then the stabilizer $\AUTOGROUP_V$ of $V$ is cyclic.
\end{theoremK}

It is a standard result that the group ${\rm Aut}(V)$ of holomorphic automorphisms of a
Kobayashi hyperbolic manifold $V$ is a real Lie group.
To prove Theorem K we use that $\AUTOGROUP_{A,B,C,D}$ preserves a volume form (see
Section \ref{SEC:VOLUME_FORM}) to show that the closure $G=\overline{\AUTOGROUP_V}$ of $\AUTOGROUP_V$ is a compact
Lie group. Checking that any element of $\AUTOGROUP_V$ has infinite order we
conclude that $G$ has positive dimension.  Supposing that $\AUTOGROUP_V$ is
non-cyclic we can conclude that it is non-Abelian and moreover that there are
non-commuting elements arbitrarily close to the identity.  This gives that the
connected component of the identity, $G_0$, is non-Abelian.  Since $G_0$ is
compact, it must therefore have real-dimension $3$.  The assumption that no
element of $\AUTOGROUP$ has a fixed point in $V$ gives that $G_0$ acts freely on $V$
and thus that $V/G$ is a manifold of real dimension~$1$.  This allows us to
derive a contradiction to the fact that the Julia set is connected (Theorem C).

Note that the works of Bedford-Smillie \cite{BS} and Bedford-Kim \cite{BK}
about Fatou components for volume preserving H\'enon maps and rational surface
automorphisms serve as a kind of prototype for how Lie Groups are used in the
proof of Theorem K.


\subsection{Open problems.}
Several interesting open problems arose while writing this paper. They are organized into
a short companion paper ``Questions about the dynamics on a natural family of affine cubic surfaces'' \cite{RR2}.

\subsection{Structure of the paper.}

In Section \ref{SEC:RELATED_WORKS} we describe several of the previous works
that are related to this paper.  We present in Section \ref{SEC:MOTIVATIONS} a
discussion of the motivations for a dynamical study of $\AUTOGROUP_{A,B,C,D}$,
with emphasis on the connections to dynamics on character varieties and to the
Painlev\'e 6 differential equation.  
In Section \ref{preliminaries} we present
several basic properties of the group $\AUTOGROUP \equiv \AUTOGROUP_{A,B,C,D}$
and of the surface $S \equiv S_{A,B,C,D}$ that will be used throughout the
paper.  In Section~\ref{Parabolic_maps_proofTheoremA} we study properties of
the ``parabolic'' elements of $\AUTOGROUP$ and use them and Montel's Theorem to
prove Theorems A, B, and C.  Section \ref{SEC:PICARD} is devoted to a careful
study of the Picard parameters $(A,B,C,D) = (0,0,0,4)$ and a proof of Theorem
D.  In Section~\ref{SEC:PROOF_THEOREME} we prove Theorem E about existence of
the unbounded Fatou components, following the techniques from
\cite{bowditch,TWZ,MPT,Hu}.  In Section~\ref{SEC:LOCALLY_NONDISCRETE} we
produce examples of locally non-discrete dynamics and that will be needed for
Theorem~F.  In Section \ref{dynamics_near_infinty} we collect several important
properties of how $\AUTOGROUP$ acts near infinity that will be needed later in
the paper, including a proof of Proposition
\ref{PROP:GOODPARAMS_IN_COMPLETMENT_HYPERSURFACES},  which plays a key role in
Theorem G.  Section \ref{SEC:FATOU} is devoted to a proof of Theorem H about
unbounded Fatou components and Theorem K about bounded Fatou components.  We
finish the proofs of our theorems with Section~\ref{SEC:PROOF_OF_THEOREM_G}
where we prove Theorems~F~and~G.

\vspace{0.1in}
\noindent

{\bf Remark:} Sections \ref{SEC:RELATED_WORKS} and \ref{SEC:MOTIVATIONS} are largely about placing this paper
in a broader context, highlighting further applications and connections with previous works. These sections
can be skimmed over on a first reading and then returned to later as additional information becomes needed.

\subsection{Acknowledgments}
We are very grateful to the anonymous referee for his or her detailed reading of this paper 
and thoughtful comments which have allowed us to substantially improve the quality of exposition.
We thank Micha\l \ Misiurewicz for ideas which helped with the proof
of Proposition \ref{PROP:SING_PTS_EIGVALS}. We also thank Eric Bedford, Philip Boalch,
Serge Cantat, Jeffrey Diller, Romain Dujardin, Simion Filip, Yan Mary He, John
Hubbard, Bernard Julia, Mikhail Lyubich, Jean-Pierre Ramis, and Ser Peow Tan
for interesting comments and discussions regarding
this work.  The second author is grateful to his colleagues Pavel Bleher,
Alexander Its, Eugene Mukhin, Vitaly Tarasov, and Maxim Yattselev for their
feedback on several early versions of this work.  This work was supported by
CNRS, through IEA ``Dynamics for groups of birational maps acting on
surfaces'', and by the US National Science Foundation through grant
DMS-2154414 and  DMS-1348589.

\section{Related works.}\label{SEC:RELATED_WORKS}

The previous works arising from dynamics on character varieties
can be traced back to Markoff's Theorem and to works on the Markoff Surface, see
\cite{bowditch} and \cite{series}.   To the best of our knowledge, deeper
investigations of these dynamics follow two main trends:
\begin{itemize}
\item[(i)] Global dynamics on the $2$-dimensional real (singular) surfaces
for real parameters $A,B,C,D$, as initiated by Goldman \cite{goldman-1} (see also
\cite{goldman-2}).

\item[(ii)] Study of domains in the complex surface $S_{A,B,C,D}$ on which the
group $\AUTOGROUP_{A,B,C,D}$ acts properly discontinuously, as initiated by
Bowditch \cite{bowditch} and later studied extensively by several authors.   We
refer the reader especially to the works by Tan, Wong, and Zhang \cite{TWZ},
Maloni, Palesi, and Tan \cite{MPT}, and Hu, Tan, and Zhang \cite{Hu}.
\end{itemize}

Based on motivations from the monodromy of the Painlev\'e~6 differential equation, the previous dynamical results follow three main trends:
\begin{itemize}
\item[(i)]
Finite orbits under $\AUTOGROUP_{A,B,C,D}$ correspond to algebraic solutions of Painlev\'e 6.   They were classified
by Dubrovin and Mazzocco \cite{dubrovinmazzocco} and by  Lisovyy and  Tykhyy
\cite{LT_ALGEBRAIC}.  See also \cite[Section 4]{cantat-2} for a classification of bounded orbits.

\item[(ii)]
Study of individual mappings from $\AUTOGROUP$ and $\AUTOGROUP^\pm$ displaying
rather interesting dynamics by Iwasaki  and Uehara \cite{IU_ERGODIC}, including
uniformly hyperbolic ones by Cantat \cite{cantat-1}.   These mappings share many features in common with
the complex H\'enon mappings.

\item[(iii)] Proof by Cantat and Loray \cite{cantat-2} that except for the case Picard parameters ($A=B=C=0$,
$D=4$), the action of $\AUTOGROUP$ on $S_{A,B,C,D}$ preserves neither an
(multi) affine structure nor a (multi) holomorphic foliation.  This corresponds to the
Malgrange irreducibility of Painlev\'e~6, as explained in their paper.

\end{itemize}

To the best of our knowledge, the previous dynamical works related to the
aperiodic Schr\"odinger equation primarily involve delicate issues about the
iteration of a single real mapping $\gamma \in \AUTOGROUP_{A,B,C,D}$ on the real slice of $S_{A,B,C,D}$.   There is a huge
literature on the subject and we refer the reader to the papers by Casdagli and
Roberts \cite{CASDALGI,ROBERTS} and references therein for an introduction.
For more contemporary works, see, for example, the paper of Damanik,
Gorodetski, and Yessen \cite{DGY} and the paper of Yessen \cite{YESSEN}.

The fact that mapping class group actions on character varieties can be formulated as a complex dynamical systems
allows one to see the present work as joining a recent trend of papers where
methods of holomorphic dynamics were used to investigate certain natural dynamical systems,
see \cite{cantat-1}, \cite{deroin-1}, and \cite{deroin-2}.

\subsection{Relationship to the recent work of S. Cantat and R. Dujardin.}
\label{SUBSEC:CANTAT_DUJARDIN}
A recent work by S. Cantat and R. Dujardin \cite{cantat-dujardin}
studies the dynamics of subgroups of the automorphism group of (compact) real and complex projective surfaces (or, more generally,
K\"ahler surfaces). They obtain deep results on the structure of stationary
measures for these subgroups, including criteria to determine when they must be invariant by the group
in question (stiffness) as well as detailed descriptions of the resulting invariant measures. Whereas their results
clearly have serious implications on the corresponding dynamics, the situation seems to be genuinely more subtle
in the case of groups of birational maps. Relevant examples include the case of the actions of $\AUTOGROUP$ and
$\AUTOGROUP^\pm$ since, as shown in the present paper, for certain values of the parameters these actions
display both non-empty Fatou set and Julia sets with non-empty interior. This phenomenon indicates
that the problem of invariance of stationary measures, and of their subsequent description, will hardly allow
for a reasonably ``compact'' classification. Also, the nature of the Fatou components constructed in our
Theorem~E shows that even the convergence of stochastic processes
will no longer be automatic which, in turn, seriously limits the applications of stationary measures in
topological problems such as description of non-compact invariant closed sets.

\subsection{Relationship to previous works using the locally non-discrete/discrete dichotomy}\label{furthercommentsdiscretecircle}
The
notion of locally non-discrete group (or even pseudogroup) first appeared in \cite{REBELO_ENS} after
previous works by A. Shcherbakov, I. Nakai, and E. Ghys, see \cite{AASh,nakai,GHYS}. The notion was further
elaborated in \cite{LORAY_REBELO}. The main common tool of \cite{REBELO_ENS} and \cite{LORAY_REBELO} is the
construction of certain vector fields obtained as a sort of limit of certain dynamics in the group/pseudogroup
in question which allow for a detailed analysis of the corresponding group dynamics. In turn, the method put
forward in those papers ensures the existence of the desired vector fields for locally non-discrete (pseudo-) groups
provided there is a local expansion in the dynamics: typically, we would like some element in group to have
a fixed point where all its eigenvalues are of modulus greater than~$1$. 
By way of contrast, elements of $\AUTOGROUP$ all preserve the same volume form (cf. Section \ref{SEC:VOLUME_FORM})  and
therefore never have repelling fixed points.  So we must pursue different
arguments.


It should be pointed out that the structure of {\it locally discrete}\, groups of
diffeomorphisms of the circle is rather well described, see for example \cite{VictorSeveralauthors}, \cite{advances}.
In particular, it follows that the class formed by {\it locally non-discrete}\, groups is extremely large and, as a matter
of fact, much can be said about the dynamics associated with the latter groups, see for example \cite{anas}
and the references therein.
However, in the case of the circle, the relatively simple nature of the
topological dynamics of groups acting on (real) $1$-dimensional manifolds leads to
the fact that coexistence of local discreteness and local non-discreteness is
impossible. More precisely, in ${\rm Diff}^{\omega} (S^1)$, if a non-Abelian
group $G \subset {\rm Diff}^{\omega} (S^1)$ is locally non-discrete on an
interval $I \subset S^1$, then every point $p \in S^1$ has a neighborhood where
$G$ is locally non-discrete, see \cite{anas}.  This dramatically contrasts with
our Theorem~F which asserts that the same group action $\AUTOGROUP_{A,B,C,D}$  can
simultaneously have open sets $U$ and $V$ on which it acts locally
non-discretely and locally discretely, respectively.   When combined with the
complicated behavior of the pointwise dynamics of $\AUTOGROUP_{A,B,C,D}$, in
particular the existence of invariant sets with topological dimension~$3$
(Corollary of Theorem F), we find very rich dynamics in the system that we study.

Let us also mention that V. Kleptsyn and his collaborators have found examples
of locally discrete subgroups of ${\rm Diff}^{\omega} (S^1)$ that are not
conjugate to Fuchsian groups, up to finite
covering~\cite{VictorSeveralauthors}. In our context, the Picard parameters
provide an analogous non-trivial example of a group that is ``purely locally discrete'',
i.e., there is no open set $U \subset S_{A,B,C,D}$, $U \neq \emptyset$, where
the group acts in a locally non-discrete way.


\section{Dynamics on Character Varieties and the Painlev\'e 6 equation.}\label{SEC:MOTIVATIONS}

The most compelling motivations for studying the dynamical systems considered here comes from
dynamics on character varieties and from the sixth Painlev\'e equation: the two motivations being closely related
as will soon be seen.

\subsection{Motivations from Dynamics on Character Varieties.} \label{SEC:INTRO_CHARACTER}
We begin with character varieties as the dynamical system
discussed in this paper is equivalent to the action of the mapping class group
of a surface on the space of ${\rm SL}(2,\mathbb{C})$-representations of its
fundamental group, up to conjugation.  This contains, in particular, the
standard action of mapping class groups on Teichm\"uller spaces and hence can
be viewed as belonging to the setting of ``higher Teichm\"uller theory'' as
well; see, e.g., \cite{WIENHARD}.


The dynamics we are interested in can be cast in the broader framework of (natural) dynamics
on character varieties as was first pointed out in
W.~Goldman's papers \cite{goldman-1} and \cite{goldman-2}. Let us briefly explain
how character varieties arise in this context. Consider, for example, the case of a punctured torus. Its fundamental group $\Pi$
is isomorphic to the free group ${\rm F}_2$
on two generators. The space of representations of $\Pi$ in ${\rm SL}(2,\mathbb{C})$ is clearly identified with
${\rm SL}(2,\mathbb{C}) \times {\rm SL}(2,\mathbb{C})$ which, in turn, is acted upon by ${\rm SL}(2,\mathbb{C})$
via simultaneous conjugation. The corresponding character variety is, by definition, the (categorical)
quotient of the latter action.
Next, note that automorphism group of $\Pi$ also acts on
the space of representations by pre-composition and this action descends to the character varieties. However, on
the character variety, the action of inner automorphism of $\Pi$ becomes trivial so that the action ${\rm Aut}(\Pi)$
factors through an action of the group ${\rm Out}\, (\Pi)$
consisting of outer automorphisms of $\Pi$.

Whereas the previous construction essentially makes sense for representations of $\Pi$ in any
group $G$, the fact that we are dealing with $G = {\rm SL}(2,\mathbb{C})$ can further be exploited
as follows. As pointed out, the character variety in question is identified with a pair of matrices in
${\rm SL}(2,\mathbb{C}) \times {\rm SL}(2,\mathbb{C})$ up to simultaneous conjugation.
A classical result of Fricke, however, shows that the latter space can be parameterized by $\C^3$.
We therefore inherit an action of ${\rm Out}\, (\Pi)$ on $\C^3$
which, in addition, coincides with the action on $\C^3$ of the group $\AUTOGROUP^\pm$ obtained by setting $A=B=C=0$.
Incidentally, the latter group coincides with the group of polynomial automorphisms of $\C^3$
preserving the (Markoff) polynomial $x^2 + y^2 + z^2 - xyz -2$, cf.
\cite{horowitz} and see \cite{procesi} for a general result on the polynomial nature of similar
actions. Finally, note that an analogous
construction applies to the quadruply-punctured sphere.
In the latter case, the corresponding action of ${\rm Out}\, (\Pi)$ recovers the
action of $\AUTOGROUP$ and $\AUTOGROUP^\pm$ with all the parameters $A$, $B$, $C$, and $D$, see \cite{benedettogoldman}
for a detailed account. Other than \cite{benedettogoldman}, more or less comprehensive versions of
the previous discussion appear in a number of papers including \cite{goldman-1}, \cite{goldman-2},
\cite{cantat-1}, \cite{cantat-2}, \cite{horowitz}.

\subsection{Properly Discontinuous Dynamics, quasi-Fuchsian representations, and Bowditch BQ Conditions}
\label{SUBSEC:TEICH}
We now elaborate how the dynamical context explained in Section \ref{SEC:INTRO_CHARACTER}
leads to the existence of parameters with non-empty Fatou sets.


Consider the general case of the $4$-holed sphere $\Sigma_{0,4}$ and note
that its fundamental group can be identified with the free group on three generators $\Pi \cong {\rm F}_3$.
Furthermore, since the orbits of the action of ${\rm SL}(2,\mathbb{C})$
on itself by conjugation are also of dimension~$3$, we see that the space of representations from $\Pi$
to ${\rm SL}(2,\mathbb{C})$ (up to conjugation) has $6$ complex dimensions.
A far more accurate description is possible: the space of representations can be identified with
the {\it quartic hypersurface of $\C^7$}\, given by
$$
\mathbb{S} = \{ (a_1,a_2,a_3,a_4, x,y,z) \in \C^7 \, : \, x^2 + y^2 + z^2 +xyz = Ax + By+Cz+D \} \, ,
$$
where $A,B,C,D$ are as follows:
\begin{equation}
A = a_1 a_4 + a_2a_3 \, , \; \; B = a_2a_4 + a_1 a_3 \, , \; \; C = a_3a_4 + a_1 a_2 \,  , \label{For-A_B_C}
\end{equation}
and
\begin{equation}
D = 4 - [a_1a_2a_3a_4 + a_1^2 + a_2^2 + a_3^2 + a_4^2 ] \, . \label{For-D_finally}
\end{equation}
In particular, by fixing the variables $a_1, \ldots ,a_4$, we obtain the surface $S_{A,B,C,D}$. The
parameters $a_1, \ldots , a_4$ are identified with the traces of the matrices in ${\rm SL}(2,\mathbb{C})$
arising from the loops around the holes of $\Sigma_{0,4}$.

\vspace{0.1in}

\noindent
{\bf Quasi-Fuchsian representations:}
Consider the space of quasi-Fuchsian representations inside~$\mathbb{S}$. The Bers Simultaneous Uniformization theorem
implies that the space of quasi-Fuchsian representations ${\rm Rep}_{qf} (\Sigma_{0,4})$ can be identified
with the product of two copies of the Teichm\"uller
space ${\rm Teich}\, (\Sigma_{0,4})$
of the $4$-holed sphere $\Sigma_{0,4}$, with geodesic boundaries, i.e.
\begin{equation}
{\rm Rep}_{qf} (\Sigma_{0,4}) = {\rm Teich}\, (\Sigma_{0,4}) \times {\rm Teich}\, (\Sigma_{0,4}) \, .\label{productTeichmuller}
\end{equation}
In turn, recalling that the group of outer automorphisms of the fundamental
group of $\Sigma_{0,4}$ is isomorphic to the mapping class group of
$\Sigma_{0,4}$, the latter acts on the space ${\rm Rep}_{qf} (\Sigma_{0,4}) =
{\rm Teich}\, (\Sigma_{0,4}) \times {\rm Teich}\, (\Sigma_{0,4})$ diagonally
with respect to its standard action on ${\rm Teich}\, (\Sigma_{0,4})$. In
particular, the action of the mapping class group on ${\rm Rep}_{qf}
(\Sigma_{0,4})$ is {\it properly discontinuous}.

Now, note that the $4$-holed sphere consists of two pair of pants joined by the
``waist''. Hence the real dimension of ${\rm Teich}\, (\Sigma_{0,4})$ is~$6$ so
that the real dimension of ${\rm Rep}_{qf} (\Sigma_{0,4})$ is~$12$ and
therefore ${\rm Rep}_{qf} (\Sigma_{0,4})$ has has non-empty interior in
$\mathbb{S}$.  Finally, if the parameters $a_1, \ldots , a_4$ are chosen so
that the corresponding surface $S_{A,B,C,D}$ intersects ${\rm Rep}_{qf}
(\Sigma_{0,4})$ and this intersection contains an open set $U \subset
S_{A,B,C,D}$, then the preceding shows that $\AUTOGROUP$ acts properly
discontinuously on $U$ and hence that $U$ is contained in the Fatou set of the
action of $\AUTOGROUP$ on $S_{A,B,C,D}$.

Note that it is easy to find parameters $a_1, \ldots , a_4$ so that the
resulting surface $S_{A,B,C,D}$ does not intersect ${\rm Rep}_{qf}
(\Sigma_{0,4})$.  Recalling that $a_1, \ldots , a_4$ are the traces of matrices
in ${\rm SL}(2,\mathbb{C})$ arising from loops around the holes, it is enough
to force one of these matrices to be an elliptic element conjugate to an
irrational rotation. An alternative argument leading to open sets of parameters
$a_1, \ldots , a_4$ for which $S_{A,B,C,D}$ is disjoint from ${\rm Rep}_{qf}
(\Sigma_{0,4})$ can be formulated by using the well-known Jorgensen inequality.

\vspace{0.1in}

\noindent
{\bf Bowditch Conjecture and BQ Conditions:}
Bowditch conjectures in \cite{bowditch} that a point $(x,y,z) \in S_{0,0,0,0}$ is in ${\rm Rep}_{qf}
(\Sigma_{0,4})$ if and only if the following two simple {\em BQ conditions} hold:
\begin{enumerate}
\item  None of the coordinates of $\gamma(x,y,z)$ is in $[-2,2]$  for any $\gamma \in \AUTOGROUP$, and
\item  $\gamma(x,y,z) \in \Big(\mathbb{C} \setminus \overline{\mathbb{D}_{2}}\Big)^3$ for all but finitely many $\gamma \in \AUTOGROUP$.
\end{enumerate}
It is easy to prove that if $(x,y,z) \in {\rm Rep}_{qf} (\Sigma_{0,4})$ then
(1) and (2) hold, but the converse is known as the {\em Bowditch Conjecture}.  
See \cite{STY,LEE_XU,series2} for recent related works and more details.

Conditions (1) and (2) were studied for the punctured torus parameters $(0,0,0,D)$ by Tan, Wong, and Zhang
\cite{TWZ} and for arbitrary parameters $(A,B,C,D)$ by Maloni, Palesi, and Tan \cite{MPT}.  
 (The ``radius'' $2$ in Condition (2) needs to be adjusted when $(A,B,C) \neq (0,0,0)$, as explained in \cite{MPT}). 
Let ${\rm Rep}_{BQ} (\Sigma_{0,4}) \subset \mathbb{S}$
denote the set of points for which the BQ Conditions hold.   It is proved in \cite{TWZ} and  \cite{MPT} that for any $(A,B,C,D)$ the set
\begin{align*}
V_{\rm BQ}(A,B,C,D):= {\rm Rep}_{BQ} (\Sigma_{0,4})  \cap S_{A,B,C,D}
\end{align*}
is open and that $\AUTOGROUP_{A,B,C,D}$ acts properly discontinuously on it.  
In particular, $V_{\rm BQ} \equiv V_{\rm BQ}(A,B,C,D)$ forms part of the Fatou set for $\AUTOGROUP_{A,B,C,D}$.

Notice that $V_{\rm BQ}$ can  be non-empty for parameters $a_1, \ldots , a_4$
beyond those for which $S_{A,B,C,D}$
intersects ${\rm Rep}_{qf} (\Sigma_{0,4})$. 
Several such examples are presented in \cite{TWZ} and also \cite[Theorem 5.3]{MPT}.

The proof of Theorem~E provided in this paper will be an adaptation of ideas
from the papers \cite{TWZ,MPT} and also the later paper by Hu, Tan, and Zhang
\cite{Hu}.

\subsection{Motivations from the Painlev\'e equation P6.}
From a different perspective, the complex dynamical system studied in this paper describes the transverse dynamics of the
celebrated Painlev\'e~6 equation.
In particular, the splitting of dynamics into two regions with contrasting dynamical behavior - a region
called {\it Fatou}\,
where the dynamics is simple (e.g. properly discontinuous) and another region named {\it Julia}\,
where the dynamics is chaotic - also accounts
for the somehow ``double nature'' of Painlev\'e~6: in the vast literature on the subject, it is possible to find
a thread where this equation is viewed as part of integrable system and another one where it is regarded as a complicated dynamical
systems. The simultaneous existence of large Fatou and Julia sets (both with non-empty interiors cf.
Theorem~G) justifies and conciliates both perspectives. In particular, the results obtained in this paper -
and especially those involving the {\it Julia set} - are in line with general programs aimed
at the dynamical study of Painlev\'e equations, cf. \cite{Iwasaki_CMP} and \cite{Ramis}.

To clarify this connection, let us begin by recalling that in
the most standard notation, the sixth Painlev\'e equation takes on the form
\begin{eqnarray}
\frac{d^2y}{dx^2}&=& \frac{1}{2}\left(\frac{1}{y}+\frac{1}{y-1}+\frac{1}{y-x}\right)\left(\frac{dy}{dx}\right)^2-
\left(\frac{1}{x}+\frac{1}{x-1}+\frac{1}{y-x}\right)\frac{dy}{dx} + \nonumber \\
& & \, + \frac{y(y-1)(y-x)}{x^2(x-1)^2}\left(\alpha+\beta\frac{x}{y^2}+\gamma\frac{x-1}{(y-1)^2}+
\delta\frac{x(x-1)}{(y-x)^2}\right)\,, \label{Painleve-6_BasicForm}
\end{eqnarray}
where $\alpha, \, \beta, \, \gamma$, and $\delta$ are complex parameters. K. Iwasaki in \cite{iwasaki} provided a useful
alternative way to represent the
parameters involved in this equation by denoting them as
$\kappa_1, \, \kappa_2, \, \kappa_3$, and $\kappa_4$ where the following formulas hold:
\begin{equation}
\alpha = \kappa_4^2/2 \; , \; \; \beta = -\kappa_1^2/2 \; \; \; \gamma = \kappa_2^2/2 \; , \; \; {\rm and} \; \;
\delta = (1-\kappa_3^2)/2 \, . \label{from-alphas-to-kappas}
\end{equation}

Now, as a non-autonomous differential equation of second order, P6 can equivalently be viewed as a vector field $Z_{\rm VI}$
on $\C^3$ having the form
\begin{equation}
Z_{\rm VI} = \frac{\partial}{\partial x} + z \frac{\partial}{\partial y} + \mathcal{H}_{\alpha, \beta,\gamma,\delta} (x,y,z)
\frac{\partial}{\partial z} \label{Painleve6-Vectorfield},
\end{equation}
where the function $\mathcal{H}_{\alpha, \beta,\gamma,\delta} (x,y,z)$ is
obtained from the right side of~(\ref{Painleve-6_BasicForm}) by substituting $z$ for $dy/dx$. In particular, the variable~$x$
can naturally be identified with ``time'' in the standard form~(\ref{Painleve-6_BasicForm}).

Let $\fol$ denote the foliation defined on $\C^3$ defined by the local orbits of $Z_{\rm VI}$. The foliation $\fol$
is holomorphic since it can be defined by a {\it polynomial vector field}\, obtained from $Z_{\rm VI}$ by multiplying it by the denominator
appearing in the rational expression for the function $\mathcal{H}_{\alpha, \beta,\gamma,\delta} (x,y,z)$. It is immediate to
check that the foliation $\fol$ on $\C^3$ is transverse to the fibers
of the standard fibration (projection) $\pi_x$ of $\C^3$ to the $x$-axis, away from the {\it invariant fibers}\, sitting over
$\{ x=0\}$ and~$\{ x=1\}$. In \cite{okamoto}, Okamoto obtained a much nicer birational model for the foliation $\fol$.
By compactifying the fibers of $\pi_x$ in a suitable Hirzebruch surface, performing a number of well chosen blow up maps,
and removing a certain resulting divisor the following setting for the foliation $\fol$
(see for example \cite{okamoto}, \cite{gausstopainleve}):
\begin{itemize}
  \item A complex (open) manifold $N$ of dimension~3 fibering over $\C \setminus \{ 0,1\}$ with an open surface
  denoted by $F$ as typical fiber. Moreover the projection map $\textsf{p} : M \rightarrow \C \setminus \{ 0,1\}$
  arises as the transform of the initial projection $\pi_x : \C^3 \rightarrow \C$ through the corresponding sequence of blow-up
  maps.

  \item $N$ is equipped with the corresponding transform (still denoted by $\fol$) of the extended foliation $\fol$
  of $\C \times {\rm F}_{\epsilon}$ by the corresponding blow-up maps.

  \item The foliation $\fol$ is transverse to the fibers of $\textsf{p}$ and the base $\C \setminus \{ 0,1\}$ can still
  be naturally identified with the ``time'' in~(\ref{Painleve-6_BasicForm}).

  \item The restriction of $\textsf{p}$ to a leaf $L$ of $\fol$ yields a covering map from $L$ to $\C \setminus \{ 0,1\}$.

\end{itemize}
Owing to the fourth condition, paths contained in $\C \setminus \{ 0,1\}$ can be lifted in the leaves of
$\fol$ so as to yield a homomorphism $\rho$ from the fundamental group of $\C \setminus \{ 0,1\}$
to the group of holomorphic diffeomorphisms ${\rm Diff}\, (F)$ of $F$. In other words,
$\rho (\pi_1 (\C \setminus \{ 0,1\}) \subset {\rm Diff}\, (F)$ is the holonomy group of $\fol$ whose action
on~$F$ encodes the transverse dynamics of the initial Painlev\'e~6 equation.

The Riemann-Hilbert map conjugates the monodromy action of $\rho (\pi_1 (\C
\setminus \{ 0,1\}))$ on $F$ at parameters $(\kappa_1,\ldots,\kappa_4)$ with the
action of $\AUTOGROUP_{A,B,C,D}$ on $S_{A,B,C,D}$.  More precisely there are two mappings
\begin{align*}
\mathfrak{rh}: \  \mathbb{C}^4  \rightarrow \mathbb{C}^4 \qquad \mbox{and} \qquad
{\rm RH}_{(\kappa_1,\kappa_2,\kappa_3,\kappa_4)}:  F  \rightarrow S_{A,B,C,D},
\end{align*}
where $(A,B,C,D) = \mathfrak{rh}(\kappa_1,\kappa_2,\kappa_3,\kappa_4)$.
They have the property that for any choice of parameters $(\kappa_1,\kappa_2,\kappa_3,\kappa_4)$ the mapping ${\rm RH}_{(\kappa_1,\kappa_2,\kappa_3,\kappa_4)}$ conjugates
the monodromy action of $\rho (\pi_1 (\C \setminus \{ 0,1\}))$ on $F$ to the 
action of $\AUTOGROUP_{A,B,C,D}$ on $S_{A,B,C,D}$.

The mapping $\mathfrak{rh}$ is relatively simple. If Iwasaki parameters $\kappa_1, \ldots ,\kappa_4$ are considered
for P6, then we define
$$
a_i = 2 \cos (\pi \kappa_i) \, ,
$$
for $i=1, \ldots ,4$. Next, 
to the $4$-tuple of complex numbers $(a_1, \ldots ,a_4)$, let us assign the $4$-tuple $(A,B,C,D)$ by
means of Formulas~(\ref{For-A_B_C}) and~(\ref{For-D_finally}).
Applying these changes of parameters defines $(A,B,C,D) = \mathfrak{rh} (\kappa_1, \kappa_2, \kappa_3, \kappa_4)$.

As mentioned, it is well known that the mapping
${\rm RH}_{(\kappa_1,\kappa_2,\kappa_3,\kappa_4)}$ provides a holomorphic conjugation
between the dynamics of $\rho (\pi_1 (\C \setminus \{ 0,1\})$ on $F$ and the dynamics of $\AUTOGROUP$ on
$S_{A,B,C,D}$, see \cite[Theorem 8.4]{inaba} and also to \cite{dubrovinmazzocco}
for further details.
Now, considering the
holonomy representation $\rho (\pi_1 (\C \setminus \{ 0,1\})) \subset {\rm Diff}\, (F)$, it is well known
that the mapping
${\rm RH}_{(\kappa_1,\kappa_2,\kappa_3,\kappa_4)}$ provides a holomorphic conjugation
between the dynamics of $\rho (\pi_1 (\C \setminus \{ 0,1\})$ on $F$ and the dynamics of $\AUTOGROUP$ on
$S_{A,B,C,D}$, see \cite[Theorem 8.4]{inaba} and also to \cite{dubrovinmazzocco}
for further details.

Whereas the a mapping ${\rm RH}_{(\kappa_1,\kappa_2,\kappa_3,\kappa_4)}$ can hardly be computed due to
its highly transcendental nature, it still provides a holomorphic conjugation between the dynamical systems
in question. In particular, at least on compact parts of $F$,
conjugation invariant properties of the dynamics of $\AUTOGROUP$ can immediately be
translated into dynamical properties of Painlev\'e~6, and conversely.
This applies in particular to our results
involving the dynamics of $\AUTOGROUP$ in its {\it Julia set}.

\subsection{Consequences of our results to Painlev\'e 6}
\label{SEC:RELATIONSHIP_TO_P6}
Taken together, our theorems provide plenty of rather explicit examples of parameters
$(A,B,C,D)$ along with open sets $U \subset {\mathcal J}_{A,B,C,D}$ in which the
action of $\AUTOGROUP_{A,B,C,D}$ has dense orbits. In connection with these examples, we would like
to point out a computational issue that deserves further attention. These have to do with the fact
that the group action $\AUTOGROUP$ is conjugate by the Riemann-Hilbert map to the monodromy of the actual
Painlev\'e~6 equation. In our examples, it is not too hard to estimate the ``size'' of the open sets $U$ provided by
Theorem~G: for example, we can provide some explicit $r >0$ and a point $p \in U$ such that the ball of radius~$r$
about $p$ is certainly contained in $U$. Yet, it for applications, it would be very interesting to have
similar estimates for the size/place of the image of $U$ by the Riemann-Hilbert map, so as to provide
accurate numerics for an open set in which the actual Painlev\'e~6 equation must have dense orbits.
The difficulty of this problem largely stems from well known difficulties in computing the Riemann-Hilbert map.

\section{Basic properties of $S_{A,B,C,D}$ and $\AUTOGROUP_{A,B,C,D}$}
\label{preliminaries}

\subsection{Projective compactification and triangle at infinity}\label{projective_compactification}
 We begin by considering the natural compactification of $\C^3$ in $\mathbb{CP}^3$ so that
$\bbCP^3 = \C^3 \cup \Pi_{\infty}$, where $\Pi_{\infty} \simeq \bbCP^2$ is the plane at infinity in $\bbCP^3$. Next, denote
by $\overline{S}_{A,B,C,D}$ the closure of $S_{A,B,C,D}$ in $\mathbb{CP}^3$. By using the standard affine atlas for
$\bbCP^3$, it is straightforward to check that $\overline{S}_{A,B,C,D} \cap \Pi_{\infty}$ consists of three projective
lines forming a triangle $\Delta_\infty$ in $\Pi_{\infty}$. Indeed, if $(u,v,w)$ are affine coordinates for
$\bbCP^3$ satisfying $(1/u ,v/u , w/u) = (x,y,z)$, then the surface $S_{A,B,C,D}$ is determined in $(u,v,w)$-coordinates by the
equation
$$
u+uv^2 + uw^2 + vw = Au^2 + Bu^2v + Cu^2w + Du^3 \, .
$$
In particular, it follows that $\overline{S}_{A,B,C,D} \cap \Pi_{\infty}$ locally coincides with the axes $\{ u=v=0\}$
and $\{ u=w=0\}$. An analogous use of the remaining coordinates shows that the
third side of the mentioned triangle coincides with the projective line of $\Pi_{\infty} \simeq \bbCP^2$
which is missed by the domain of the affine coordinates $(v,w) \simeq (u=0, v,w)$. A slightly less immediate computation with
the above equation for $\overline{S}_{A,B,C,D}$ also shows that
$\overline{S}_{A,B,C,D}$ is smooth on a neighborhood of $\Pi_{\infty}$.

\subsection{Relation to the family of all cubic surfaces.}\label{cubics} We now
explain how the families of surfaces $S_{A,B,C,D}$ (and their projective
compactifications $\overline{S}_{A,B,C,D}$) relate to the family of all cubic
surfaces.  This subsection is largely based on the work of Goldman and Toledo
\cite{toledo}.  Suitable general background on cubic surfaces is given in
\cite{segre,wall,manin}, and the papers quoted therein.

We start with the fact that two minimal projective cubic surfaces are
birationally equivalent if and only if they are projectively equivalent; see
\cite[p. 184]{manin}. Therefore, to decide if a given cubic surface appears in
the family $S_{A,B,C,D}$ (respectively $\overline{S}_{A,B,C,D}$) for any 
values of $(A,B,C,D)$, we can restrict ourselves to affine
(respectively projective) equivalence.

A plane meeting a cubic in three lines is called a {\it tritangent plane}. If
these three lines are in general position, then we say that we have a {\it
generic tritangent plane}.  As explained in
Section~\ref{projective_compactification}, for any choice of parameters
$(A,B,C,D)$, the plane at infinity $\Pi_\infty$ is a generic tritangent plane
for $\overline{S}_{A,B,C,D}$.  It is shown in \cite{toledo} that a projective
cubic surface belongs to the family $\overline{S}_{A,B,C,D}$ if and only if it
admits a generic tritangent plane.   Every smooth project cubic surface admits
a generic tritangent plane and therefore appears in $\overline{S}_{A,B,C,D}$
for a suitable choice of parameters. 
However, there exist singular cubic surfaces that do not admit a
generic tritangent plan, and therefore are not represented in the family
$\overline{S}_{A,B,C,D}$.  Alternatively, the singularities of the cubic
surfaces in $\overline{S}_{A,B,C,D}$ are classified in \cite{toledo} and the
possible types form a proper subset of the types of singular points occurring
for general cubic surfaces.

A generic projective cubic surface $\mathcal{S}$ has $45$ tritangent planes and
all of them are generic.  It follows from the argument in \cite{toledo} that
each choice of a tritangent plane of $\mathcal{S}$ leads to a projective
transformation of $\mathbb{CP}^3$ sending that tritangent plane to $\Pi_\infty$
and sending $\mathcal{S}$ to a member of the family $\overline{S}_{A,B,C,D}$.
One can then make any choice of $\epsilon_i \in \{-1,1\}$ for $i=1,2,3$ such
that $\epsilon_1 \epsilon_2 \epsilon_3 =1$ and apply the affine map $(x,y,z) \mapsto (\epsilon_1 x, \epsilon_2 y, \epsilon_3 z)$ 
to identify $\mathcal{S}$ with four different choices of surface
$\overline{S}_{A,B,C,D}$.   We have therefore identified $\mathcal{S}$ with
elements of the family $\overline{S}_{A,B,C,D}$ by applied $4 \times 45 = 180$
different projective transformations.   Since the automorphism group of a generic
projective cubic is trivial we conclude that $\mathcal{S}$ is identified a surface $\overline{S}_{A,B,C,D}$
for $180$ different choices of parameters $(A,B,C,D)$.

Note that if the generic cubic  $\mathcal{S}$ is identified with
$\overline{S}_{A,B,C,D}$ and also with $\overline{S}_{A',B',C',D'}$ by using
two different choices of tritangent planes of $\mathcal{S}$, then the
corresponding actions of $\AUTOGROUP_{A,B,C,D}$ and $\AUTOGROUP_{A',B',C',D'}$
are not typically conjugate because they are defined with respect to affine
coordinates corresponding to different choices of ``plane at infinity''.

\subsection{Singular points of $S_{A,B,C,D}$}
\label{SUBSEC_SINGULAR_POINTS}
As explained in Section \ref{projective_compactification}, for any parameters $(A,B,C,D)$ we have that $\overline{S}_{A,B,C,D}$ is smooth on a neighborhood of $\Pi_{\infty}$.
Therefore $\overline{S}_{A,B,C,D}$ is singular if and only if the affine surface
$S_{A,B,C,D}$ is so (and the corresponding singular sets always coincide). 
In particular, the singular set of $\overline{S}_{A,B,C,D}$ is a compact subvariety of $\mathbb{C}^3$
and it is therefore finite.

The surface
$S_{A,B,C,D}$ is singular if and only if at least one of the following conditions is satisfied (see, e.g.\ \cite{benedettogoldman,iwasaki}):
\begin{itemize}
\item We have $a_i = \pm 2$ for at least one~$i \in \{ 1,2,3,4\}$.

\item The coefficients $a_1, \ldots ,a_4$ satisfy the relation
$$
[2(a_1^2 + a_2^2 + a_3^2 +a_4^2) -a_1a_2a_3a_4 -16]^2 - (4-a_1^2)(4-a_2^2)(4-a_3^2)(4-a_4^2) = 0 \, .
$$
\end{itemize}
(The parameters $a_1,a_2,a_3$ and $a_4$ were introduced in Section \ref{SUBSEC:TEICH} and their relationship to $A,B,C,$ and $D$ was described there.)

Since $\AUTOGROUP$ must preserve ${\rm Sing}\, (S_{A,B,C,D})$, it that $\AUTOGROUP$ has a finite
orbit whenever ${\rm Sing}\, (S_{A,B,C,D}) \neq \emptyset$. The existence of a finite orbit for $\AUTOGROUP$ is naturally
yields some insights in the dynamics of $\AUTOGROUP$, as will be seen in the course of this work.

\begin{remark}
\label{REM:SINGULAR_POINTS_JULIA}
{\rm Every singular point $p$ of $S_{A,B,C,D}$ lies in the Julia set ${\mathcal J}_{A,B,C,D}$.
Indeed, if $U$ is a neighborhood
of $p$, \cite[Theorem C]{cantat-2} implies the existence of a point $q \in U$ and of a sequence $\gamma_n \in \AUTOGROUP$
such that $\gamma_n(q)$ diverges
to infinity.  Meanwhile, since every element of $\AUTOGROUP$ permutes the singular set of $S_{A,B,C,D}$, we  have that $\gamma_n(p)$ remains bounded.
}
\end{remark}

\subsection{Algebraic properties of $\AUTOGROUP^\pm_{A,B,C,D}$ and $\AUTOGROUP_{A,B,C,D}$.}
\label{SUBSEC:ALGEBRIAC_PROPERTIES}
The group of automorphisms of $\mathbb{C}^3$  generated by $s_x$,
$s_y$, and $s_z$ is isomorphic to the free product $\Z /2\Z \ast \Z /2\Z \ast
\Z /2\Z$. In particular, the group of automorphisms
generated by $g_x, g_y$, and $g_z$ is free on two generators.
However, a stronger statement holds: for every choice of the
parameters $A,B,C,$ and $D$, the group $\AUTOGROUP^\pm$
is isomorphic to the indicated free product and the group $\AUTOGROUP$ is free on two generators, {\em when
viewed as a group of
automorphisms of $S_{A,B,C,D}$.}

The non-existence of additional relations between the maps $s_x$, $s_y$, and
$s_z$ even when restricted to a particular surface $S_{A,B,C,D}$ is a
consequence of El-Huiti's theorem in \cite{huti} -- albeit not an immediate one.
For details the reader can check \cite{cantat-1} and \cite{cantat-2}.

\subsection{Invariant volume form}
\label{SEC:VOLUME_FORM}
Considering $s_x, s_y,$ and $s_z$ as self-mappings of $\mathbb{C}^3$, a simple calculation yields:
\begin{align*}
s_x^* (dx \wedge dy \wedge dz) = s_y^* (dx \wedge dy \wedge dz) = s_z^* (dx \wedge dy \wedge dz) = -dx \wedge dy \wedge dz.
\end{align*}
Therefore elements of the index two subgroup $\AUTOGROUP_{A,B,C} = \langle g_x,
g_y, g_z \rangle < {\rm Aut}(\mathbb{C}^3)$ preserve $dx \wedge dy \wedge dz$ and also the associated
Euclidean volume form $dx \wedge d\bar{x} \wedge dy \wedge d\bar{y} \wedge dz \wedge
d\bar{z}$ on $\mathbb{C}^3$.

The smooth part of $S_{A,B,C,D}$
comes equipped with a holomorphic volume form
\begin{align}\label{EQN:VOLUME_FORM}
\Omega = \frac{dx \wedge dy}{2z+xy-C} = \frac{dy \wedge dz}{2x+yz-A} = \frac{dz \wedge dx}{2y + zx-B}
\end{align}
which is obtained by contracting the form $dx \wedge dy \wedge dz$ with the gradient of the polynomial function on $\mathbb{C}^3$ that defines $S_{A,B,C,D}$.
We again have $s_x^* \Omega = -\Omega$ and similarly for
$s_y$ and $s_z$.  Therefore, the elements of $\AUTOGROUP_{A,B,C,D} < {\rm Aut}(S_{A,B,C,D})$ preserve $\Omega$ and
hence also preserve the associated real volume form $\Omega \wedge
\overline{\Omega}$.  It assigns infinite volume to the whole surface
$S_{A,B,C,D}$ but finite volume to a sufficiently small neighborhood of each singular point of $S_{A,B,C,D}$.

Let $p$ be a smooth point of $S_{A,B,C,D}$.  An immediate consequence of the
existence of the invariant volume form is that if $p$ is fixed for $\gamma \in 
\AUTOGROUP_{A,B,C,D}$ then ${\rm det}(D\gamma(p)) = 1$.  In particular, if $p$ is a
hyperbolic fixed point for $\gamma$, then it must be of saddle type.

\subsection{Pencils of rational curves}

Recall that $\pi_x: \mathbb{C}^3 \rightarrow \mathbb{C}$ sends $(x,y,z) \in \C^3$ to $x \in \C$.
Given $x_0 \in \mathbb{C}$, let
\begin{align*}
\Pi_{x=x_0} = \{(x,y,z) \in \mathbb{C}^3 \, : \,  \pi_x(x,y,z) = x_0\} \qquad \mbox{and} \qquad S_{x=x_0} = S_{A,B,C,D} \cap \Pi_{x=x_0}.
\end{align*}
Let $\overline{\Pi}_{x=x_0}$ be the closure of $\Pi_{x=x_0}$ in $\mathbb{CP}^3$ and let $\overline{S}_{x=x_0}$ denote the closure of $S_{x=x_0}$ in $\overline{S}_{A,B,C,D}$. Since $\overline{S}_{x=x_0}$ has degree two in $\overline{\Pi}_{x=x_0} \cong \mathbb{CP}^2$,
it is uniformized by the Riemann Sphere provided that $\overline{S_{x_0}}$ is smooth. The statement remains valid in the
case where $\overline{S}_{x=x_0}$ is singular (i.e.\ a union of two lines with a single simple intersection) up to passing from $\overline{S}_{x=x_0}$ to its normalization.

Denoting by $\pi_y$ and $\pi_z$ the projections of $\C^3$ to $\C$ respectively defined by $\pi_y(x,y,z) = y$ and
$\pi_z(x,y,z) = z$, the fibers $\Pi_{y=y_0}, S_{y=y_0}, \Pi_{z=z_0},$ and $S_{z=z_0}$ are analogously defined.

Clearly the collection of rational curves obtained from $\overline{S}_{x=x_0}$, $x_0 \in \C$, defines a {\it rational
pencil}\, in $\overline{S}_{A,B,C,D}$. By performing finitely many blow-ups, this pencil becomes a singular
rational fibration $\fol_x$, with connected fibers, over a suitable Riemann surface $\widetilde{\overline{S}}_{A,B,C,D}$.
Naturally, there are analogous pencils $\fol_y$ and $\fol_z$ defined on $\overline{S}_{A,B,C,D}$
with the help of the collection of rational curves $\overline{S}_{y=y_0}$ and $\overline{S}_{z=z_0}$
contained in $\overline{S}_{A,B,C,D}$.

Let us close this section with an elementary lemma, whose simple proof will be omitted.

\begin{lemma}
\label{intersectinginfinityatdistinctpoints}
For all but finitely many values of $x_0 \in \C$, the rational curve $\overline{S}_{x=x_0}$ is smooth and intersects the plane at
infinity $\Pi_{\infty}$ of $\bbCP^3$ in two distinct points. Analogous statements hold for the rational curves
$\overline{S}_{y=y_0}$ and $\overline{S}_{z=z_0}$.
\end{lemma}

For the pencil $\fol_x$ on $\overline{S}_{A,B,C,D}$, we fix a (minimal) blow-up procedure turning $\fol_x$
into a (singular) rational fibration $\widetilde{\fol}_x$.
Clearly there are only finitely many values of $x_0 \in \C$ corresponding to
singular fibers of this fibration. We then define a finite set $\mathcal{B}_x^+ \subset \C$ by saying that $x_0 \in \mathcal{B}_x$
if at least one of the following conditions fails to hold:
\begin{itemize}
  \item The rational curve $\overline{S}_{x=x_0}$ is smooth and intersects the plane at
infinity at two distinct points.
  \item The fibration $\widetilde{\fol}_x$ is regular on a neighborhood of the fiber sitting over $x_0$.
\end{itemize}
We also define $\mathcal{B}_x \subseteq \mathcal{B}_x^+$ to be the set of points at which the first of the two
conditions above fails to hold.
The corresponding sets for the coordinates~$y$ and~$z$ will similarly be denoted by
$\mathcal{B}_y^+$, $\mathcal{B}_y$ and by $\mathcal{B}_z^+$, $\mathcal{B}_z$.

\section{Dynamics of parabolic maps and proof of Theorems~A, B, and~C}\label{Parabolic_maps_proofTheoremA}

We begin by studying the generators $g_x = s_z \circ s_y, g_z = s_x \circ s_z,$ and $g_z = s_y \circ s_x$
of $\AUTOGROUP$. Here these mappings are explicitly written as
\begin{align*}
g_x\left(x, y, z\right) &= \left(x, \, -y-xz+B, \, xy + ({x}^{2}-1)z+C-Bx\right), \\
g_y\left(x, y, z \right) &=  \left(({y}^{2}-1)x+yz+A-Cy, \, y, \,  -yx-z+C \right), \quad \mbox{and} \\
g_z\left(x, y, z \right) &=  \left(-x-yz+A,\, zx + ({z}^{2}-1)y+B-Az,\,  z \right).
\end{align*}
The map $g_x$ preserves the coordinate~$x$ and hence each affine plane $\Pi_{x=x_0} \subset \C^3$. Since, in addition, $g_x$
clearly preserves the (affine) surface $S_{A,B,C,D}$, it follows that $g_x$ individually preserves each one of the curves
$S_{x=x_0}$ (and hence also the rational curves $\overline{S}_{x=x_0} \subset \overline{S}_{A,B,C,D}$). Similar conclusions
hold for the maps $g_y$ and $g_z$.

We can now complement the discussion in Section~\ref{preliminaries} revolving around
Lemma~\ref{intersectinginfinityatdistinctpoints}. Owing to these statements, we know that the rational curve
$\overline{S}_{x=x_0}$ is smooth and intersects the divisor at infinity $\Pi_{\infty}$ of $\bbCP^3$ in two distinct points
for all but finitely many values of $x_0 \in \C$. It is then natural to consider the automorphism of
$\overline{S}_{x=x_0}$ induced by $g_x$. The following proposition can be found in \cite{cantat-2} (see in particular
Proposition~4.1 in the paper in question). The proof is elementary, in the spirit of most of the discussion in the previous
section.

\begin{proposition}\label{PROP_ELLIPTIC_HYPERBOLIC}
Let $x_0 \in \mathbb{C}$ be such that the rational curve $\overline{S}_{x=x_0}$ is smooth and intersects the divisor
at infinity $\Pi_{\infty}$ in two distinct points. Then the restriction $\overline{g_{x_0}}$ of
$g_x$ to $\overline{S}_{x=x_0}$ is a M\"obius transformation whose two fixed points are at infinity. Furthermore, we have:
\begin{itemize}
\item If $x_0 \in (-2,2)$ then $\overline{g_{x_0}}$ is elliptic. It is periodic if and only if $x_0 = \pm 2 \cos(\theta \pi)$
with $\theta$ rational.
\item If $x_0 \in \mathbb{C} \setminus [-2,2]$ then $\overline{g_{x_0}}$ is loxodromic.
\end{itemize}
The analogous statements hold for restrictions of $g_y$ to the fibers $\overline{S}_{y=y_0}$ and of $g_z$ to the fibers~$\overline{S}_{z=z_0}$.
\end{proposition}

\begin{remark}
\label{when_x=+-2}
{\rm The case $x_0 = \pm 2$ deserves an additional comment for the sake of clarity.  First note that the affine singular
points of the curves $\overline{S}_{x=x_0}$ are contained in the set of points where the intersection of $S_{A,B,C,D}$ and $\Pi_{x=x_0}$
is not transverse.  They are obtained by solving for where the second and third components of the gradient of the defining polynomial
for $S_{A,B,C,D}$ both vanish.  A short calculation shows that they lie in the affine curve given by
$$
x_0 \longrightarrow \left( x_0, \frac{Cx_0 -2B}{x_0^2-4} , \frac{Bx_0 -2C}{x_0^2-4} \right) \, .
$$
We see that this curve intersects the plane at infinity for $x_0 = \pm 2$. In particular,
the affine curve $S_{x=2}$ (resp. $S_{x=-2}$) is always smooth. On the other hand, this curve intersects the plane $\Pi_{\infty} \subset \bbCP^3$
at a single point (with multiplicity~$2$) which, depending on the coefficients, may or may not be a singular point of
$\overline{S}_{x=2}$ (resp. $\overline{S}_{x=-2}$). The restriction $\overline{g_{2}}$ of $g_x$ to $\overline{S}_{x=2}$
is as follows:
\begin{enumerate}
  \item When $\overline{S}_{x=2}$ is smooth, then $\overline{g_{2}}$ is a parabolic map whose single fixed point
  coincides with the intersection of $\overline{S}_{x=2}$ with $\Pi_{\infty}$.

  \item Otherwise $\overline{S}_{x=2}$ consists of the union of two projective lines intersecting each other at a point
  in $\Pi_{\infty}$. Each line is then preserved by $\overline{g_{2}}$ and, in each of these lines, $\overline{g_{2}}$
  induces a parabolic map whose fixed point coincides with their intersection.
\end{enumerate}
The analogous statement holds for $\overline{g_{-2}}$ and $\overline{S}_{x=-2}$.}
\end{remark}

Proposition~\ref{PROP_ELLIPTIC_HYPERBOLIC} yields the following lemma:

\begin{lemma}\label{COR:GX_JULIA}
If $x_0 \in (-2,2) \setminus \mathcal{B}_x$ then $S_{x=x_0} \subset \mathcal{J}(g_{x})$.
Analogous statements hold for the $g_y$ and $g_z$ mappings.
\end{lemma}

\begin{proof}
Suppose that $U$ is an open neighborhood of a point $p \in S_{x=x_0}$. Because
$x_0 \in (-2,2) \setminus \mathcal{B}_x$, Proposition~\ref{PROP_ELLIPTIC_HYPERBOLIC} ensures that $g_x$ restricted to the fiber
$S_{x=x_0}$ is elliptic with both fixed points lying in $\Pi_{\infty}$.  In particular, the
iterates $g_x^n(p)$ remain bounded.  Meanwhile, since $U$ is open in $S$ there
exists a point $q \in U$ with $x_1 = \pi_x(q) \not \in [-2,2] \cup
\mathcal{B}_x$.  According to Proposition \ref{PROP_ELLIPTIC_HYPERBOLIC}, the
restriction of $g_x$ to the fiber $\overline{S}_{x=x_1}$ is hyperbolic with both
fixed points at infinity. Therefore the orbit $g_x^n(q)$ tends to infinity.
Hence $U$ cannot be contained in the Fatou set of $g_x$ since it contains
points with both bounded and unbounded orbits. The lemma follows.
\end{proof}

We will need the following elementary lemma, whose simple proof is omitted.

\begin{lemma}\label{LEM:TRANSVERSALITY_OF_LEAVES}
Fix any $x_0 \in \mathbb{C}$.  Then, for all but finitely many choices of $y_0$ the fibers $S_{x=x_0}$ and
$S_{y=y_0}$ intersect transversally.  When the fibers intersect transversally, they do so at two distinct points.
\end{lemma}





Now fix a point $x_0 \in \C \setminus \mathcal{B}_x^+$. For sufficiently small $\epsilon >0$, consider the ``tube''
$T^\epsilon_{x=x_0}$ defined by
\begin{align}
T^\epsilon_{x=x_0} = \{(x,y,z) \in S \, : \, |x-x_0| < \epsilon\} \label{Tube_epsilon_definition}
\end{align}
Clearly $T^\epsilon_{x=x_0}$ is filled (foliated) by the curves $S_{x=x_1}$ where $x_1$ satisfies $|x_1-x_0| < \epsilon$.
Let $y_0 \in \C$ be such that the curve $S_{y=y_0}$ intersects the curves $S_{x=x_1} \subset T^\epsilon_{x=x_0}$ transversally.

\begin{lemma}
\label{Kobayashi_hiperbolicity}
With the preceding notation and up to choosing $\epsilon >0$ sufficiently small, the open set
$T^\epsilon_{x=x_0} \setminus S_{y=y_0}$ is Kobayashi hyperbolic.
\end{lemma}

\begin{proof}
Since $\mathcal{B}_x^+$ is finite and $x_0 \in \C \setminus \mathcal{B}_x^+$, there is $\epsilon >0$
such that $\mathbb{D}_{\epsilon} = \{ x_1 \in \C \, ; \; |x_1-x_0| < \epsilon\}$ is contained in
$\C \setminus \mathcal{B}_x^+$. Now, in
view of the definition of $\mathcal{B}_x^+$, for every $S_{x=x_1} \subset T^\epsilon_{x=x_0}$, the corresponding rational
curve $\overline{S}_{x=x_1}$ intersects $\Pi_{\infty}$ transversally and at two distinct points. In other words,
$\overline{S}_{x=x_1} \setminus S_{x=x_1}$ consists of two distinct points provided that $S_{x=x_1} \subset T^\epsilon_{x=x_0}$. On the other
hand, up to a  birational transformation, the projective curves $\overline{S}_{x=x_1}$ define a (regular) holomorphic
fibration over the disc $\mathbb{D}_{\epsilon}$. Because the fibers are rational
curves and therefore pairwise isomorphic as Riemann surfaces, the theorem of Fischer and Grauert \cite{fischergrauert}
implies that this fibration is holomorphically trivial, i.e., it is holomorphically equivalent to
$\mathbb{D}_{\epsilon} \times \bbCP^1$. Since $\overline{S}_{x=x_1} \setminus S_{x=x_1}$ consists of two points, there also
follows that
\begin{align}\label{EQN:TEPSILON}
T^\epsilon_{x=x_0} = \mathbb{D}_\epsilon \times  \mathbb{C} \setminus \{0\}
\end{align}
as complex manifolds.

The hypothesis that $S_{y=y_0}$ intersects $S_{x=x_0}$ transversally implies that
their intersection consists of exactly two (distinct) points (Lemma~\ref{LEM:TRANSVERSALITY_OF_LEAVES}). We can then reduce
$\epsilon > 0$, if necessary, so that $S_{y=y_0}$ intersects each $S_{x=x_1}$ from
$T^{\epsilon}_{x=x_0}$ transversally in two points, each of them varying
holomorphically with $x_1$. These points will be referred to as the two branches of
$S_{y=y_0}$ in $T^{\epsilon}_{x=x_0}$. Since each branch is the graph of a holomorphic function on~$x_1$,
we can pick either one of them and construct a further holomorphic diffeomorphism to make it
correspond to the point $1 \in \mathbb{C} \setminus \{0\}$. By means of this construction, $T^\epsilon_{x=x_0} \setminus S_{y=y_0}$
becomes identified with an open set in
$\mathbb{D}_\epsilon \times  \mathbb{C} \setminus \{0,1\}$.
The lemma follows since the latter
manifold is Kobayashi hyperbolic as the product of two hyperbolic Riemann surfaces.
\end{proof}


\begin{lemma}\label{PROP:HITTING_OTHER_FIBERS}
Let $x_0 \in (-2,2) \setminus \mathcal{B}_x^+$ and let $y_0$ be any point chosen so that $S_{x=x_0}$ and $S_{y=y_0}$
intersect transversally.  For any open $U \subset S$ with $U \cap S_{x=x_0}
\neq \emptyset$ there is an iterate $n$ such that $g_x^n(U) \cap S_{y=y_0}~\neq~\emptyset$.

\vspace{0.05in}
\noindent
Analogous statements hold when $x$ and $y$ are replaced with any two distinct variables from $\{x,y,z\}$.
\end{lemma}

\begin{proof}
Owing to Lemma~\ref{Kobayashi_hiperbolicity}, we fix a tube $T^\epsilon_{x=x_0}$ as in~(\ref{Tube_epsilon_definition})
so that $T^\epsilon_{x=x_0} \setminus S_{y=y_0}$  is Kobayashi hyperbolic. Let then $U$ be a non-empty open set of $S_{A,B,C,D}$ intersecting
$S_{x=x_0}$. Up to trimming $U$, we can assume that $U \subset T^\epsilon_{x=x_0}$.
Since $g_x$ preserves the $S_{x={\rm const}}$ fibration and fixes the points in $\overline{S}_{x={\rm const}} \cap \Pi_{\infty}$
($g_x$ has no poles), it follows that $g_x^n(U)$ remains in $T^\epsilon_{x=x_0}$ for every $n \in \Z$. Now assume for
a contradiction that $g_x^n(U) \subset T^\epsilon_{x=x_0} \setminus S_{y=y_0}$ for
every $n$.  Since $T^{\epsilon}_{x=x_0} \setminus S_{y=y_0}$ is Kobayashi
hyperbolic, this implies that $\{ g_x^n \}$ forms a normal family on $U$.
This is, however, impossible since $U$ intersects $S_{x=x_0}$ and $S_{x=x_0} \subset \mathcal{J}(g_x)$
(Lemma~\ref{COR:GX_JULIA}). Thus there must exist $n$ such that
$g_x^n(U) \cap S_{y=y_0} \neq \emptyset$ and the lemma follows.
\end{proof}

\begin{remark}\label{REM:BRANCHES}
{\rm The above proof actually shows slightly more than the statement of Lemma~\ref{PROP:HITTING_OTHER_FIBERS}.
Once $\epsilon > 0$ is chosen sufficiently small so that $T^\epsilon_{x=x_0}$ intersects $S_{y=y_0}$ in two branches  (each a graph of
a holomorphic function of $x$) and once $U$ is chosen sufficiently small so that $U \subset T^\epsilon_{x=x_0}$
then for each branch of $S_{y=y_0}$ in $T^\epsilon_{x=x_0}$ there is an iterate $n$ so that $g_x^n(U)$ intersects
the branch in question.}
\end{remark}

Lemma~\ref{PROP:HITTING_OTHER_FIBERS} admits a useful quantitative version that can directly be proved, namely:

\begin{lemma}\label{PROP:HITTING_OTHER_FIBERS-QuantitativeVersion}
Assume that $x_0 = \pm 2 \cos(\theta \pi) \in (-2,2) \setminus \mathcal{B}_x^+$ with $\theta$ irrational
and let $U \subset S$ be an open set such that $U \cap S_{x=x_0} \neq \emptyset$. Next consider a sequence
$\{ q_j = (x_j , y_j ,z_j)\} \subset S_{A,B,C,D} \setminus S_{x=x_0}$ converging to some $q \in S_{x=x_0}$ and assume, in addition,
the existence of $\delta > 0$ such that the argument of $x_j$ lies in an interval of the form
$[\delta, \pi-\delta] \cup [\pi+\delta , 2\pi - \delta]$ for every $j$.
Then for every $j$ large enough, there exists $n_j \in \Z$ such that $g_x^{n_j} (q_j) \in U$.\\

\noindent
Analogous statements hold when $x$ is replaced by $y$ or $z$.
\end{lemma}

\begin{proof}
Recall that $g_x$ preserves every leaf of the foliation $\mathcal{D}_x$ induced by the collection
of rational curves $\overline{S}_{x={\rm const}} \subset \overline{S}_{A,B,C,D}$.
The assumption on $x_0$ implies that
$g_x$ on the curve $\overline{S}_{x=x_0}$ corresponds to an elliptic element that is conjugate to an irrational rotation.
Thus, up to saturating $U$ by the dynamics of $g_x$, there is no loss of generality in assuming
that $U$ is a disc bundle over an annulus $A \subset S_{x=x_0}$, where $A$ is invariant under $g_x$. Now
note that the action $g_{x_j}$
of $g_x$ on $S_{x=x_j}$ is loxodromic since $x_j \not\in \mathbb{R}$. Furthermore the multiplier at the attracting fixed point
of $g_{x_j}$ has modulus tending to~$1$ as $q_j \rightarrow q$. In particular, for $j$ large
enough, the annulus $A$ contains a fundamental domain of $g_{x_j}$ on $S_{x=x_j} \simeq \bbCP^1$. The lemma
follows at once.

\end{proof}

The preceding lemmas are summarized by the proposition below.

\begin{proposition}\label{PROP:TRANSITIVE_IN_NBHD_OF_ONE_FIBER}
Let $x_0 \in (-2,2) \setminus \mathcal{B}_x^+$. Given two open sets $U_1, U_2 \subset S$, both of which
intersect $S_{x=x_0}$, there is an iterate $n$ such that $g_x^n(U_1) \cap U_2 \neq \emptyset$.
Analogous statements hold when $x_0$ is replaced with $y_0$ or $z_0$.
\end{proposition}

\begin{proof}
As in the proof of Lemma~\ref{PROP:HITTING_OTHER_FIBERS}, we work within
the tube $T^\epsilon_{x=x_0}$ defined in Equation (\ref{EQN:TEPSILON}).  Since $U_2$
is open and $U_2 \cap S_{x=x_0} \neq \emptyset$ they have infinitely many points of
intersection. We can therefore pick a point $(x_0,y_0,z_0) \in U_2 \cap S_{x=x_0}$
so that $S_{y=y_0}$ intersects $S_{x=x_0}$ transversally (Lemma~\ref{LEM:TRANSVERSALITY_OF_LEAVES}). We can make $\epsilon > 0$ smaller, if necessary,
so that $S_{y=y_0} \cap T_\epsilon(x_0)$ is expressed as two smooth branches, each of them
coinciding with the graph of a holomorphic
function of $x$. Let $S_{y=y_0}(\epsilon)$ denote the branch passing through $(x_0,y_0,z_0)$.
Since $U_2$ is open, it follows that $U_2 \cap S_{y=y_0}(\epsilon)$ is an open neighborhood of $(x_0,y_0,z_0)$
in $S_{y=y_0}(\epsilon)$. This means that we can reduce $\epsilon > 0$ even further, if necessary, so
that we can assume the entire branch $S_{y=y_0}(\epsilon)$ is contained in $U_2$.

With the above setting, the combination of Lemma~\ref{PROP:HITTING_OTHER_FIBERS} and
Remark~\ref{REM:BRANCHES} ensures the existence of $n$ such that
$g_x^n(U_1) \cap S_{y=y_0}(\epsilon) \neq \emptyset$.  Since $S_{y=y_0}(\epsilon) \subset U_2$, the proposition follows.
\end{proof}

Recalling that the sets $\mathcal{B}_x^+$, $\mathcal{B}_y^+$, and $\mathcal{B}_z^+$ are all finite,
Lemma~\ref{LEM:TRANSVERSALITY_OF_LEAVES} allows us to
choose points $x_0 \neq x_1 \in (-2,2) \setminus \mathcal{B}_x$,
$y_0 \neq y_1 \in (-2,2) \setminus \mathcal{B}_y$, $z_0 \neq z_1 \in (-2,2) \setminus \mathcal{B}_z$
and to form the ``grid''
\begin{align}\label{EQN:GRID}
\GRID = S_{x=x_0} \cup S_{x=x_1} \cup S_{y=y_0} \cup S_{y=y_1} \cup S_{z=z_0} \cup S_{z=z_1}
\end{align}
so that
\begin{itemize}
\item[(i)]every pair of curves (fibers) have transverse intersection (possibly empty), and
\item[(ii)] each $x_0, x_1, y_0, y_1, z_0,$ and $z_1$ is of the form $\pm 2 \cos(\theta \pi)$
for some irrational $\theta$.
\end{itemize}
We will also say that any pair of irreducible components
of $\mathcal{G}$ with empty intersection in $S_{A,B,C,D}$ are ``parallel'', e.g. $S_{x=x_0}$ and $S_{x=x_1}$ are parallel.

\begin{proposition}\label{PROP:FROM_ONE_IRREDUCIBLE_COMP_TO_ANOTHER}
Let $U$ be any open set in $S_{A,B,C,D}$ that intersects
the grid $\GRID$.  Then, for any irreducible component of the grid (say $S_{z=x_0}$)
there exists $\gamma \in \AUTOGROUP$ with
$\gamma(U)$ intersecting that chosen irreducible component (say $\gamma(U) \cap S_{z=x_0} \neq \emptyset$).
\end{proposition}

\begin{proof}
It is clearly enough to show the existence of $\gamma\in \AUTOGROUP$ with $\gamma(U) \cap S_{z=x_0} \neq \emptyset$.
If $U$ already has non-trivial intersection with $S_{x=x_0}$ then there is nothing to prove.  Otherwise, there
are two cases to be considered:

\vspace{0.1in}

\noindent {\bf Case 1}: $U$ intersects an irreducible components of $\GRID$ that is not parallel to $S_{x=x_0}$.
Without loss of generality, we can suppose $U$ intersects $S_{y=y_0}$.
Lemma~\ref{PROP:HITTING_OTHER_FIBERS} then implies the existence of an iterate $g_y^n$ of $g_y$ so
that $g_y^n(U) \cap S_{x=x_0} \neq \emptyset$.

\vspace{0.1in}

\noindent {\bf Case 2}: $U$ intersects the component $S_{x=x_1}$ that is parallel to $S_{x=x_0}$.
In this case, we first apply Lemma~\ref{PROP:HITTING_OTHER_FIBERS} to find an iterate $g_x^n$ of $g_x$ such that
$g_x^n(U) \cap S_{y=y_0} \neq \emptyset$. Hence the problem is reduced to the situation treated in Case 1 so that it suffices
to proceed accordingly.
\end{proof}

\begin{corollary}\label{COR:TRANSITIVE_ON_G}
Let $U_1$ and $U_2$ be any two open sets in $S_{A,B,C,D}$ both of which intersect the grid~$\GRID$.  Then, there
exists $\gamma \in \AUTOGROUP$ with $\gamma(U_1) \cap U_2 \neq \emptyset$.
\end{corollary}

\begin{proof}
Using Propositition \ref{PROP:FROM_ONE_IRREDUCIBLE_COMP_TO_ANOTHER} we can find some $\gamma_1 \in \AUTOGROUP$
so that $\gamma_1(U_1)$ and $U_2$ both intersect the same irreducible component of $\GRID$.
The result then follows immediately
from Proposition~\ref{PROP:TRANSITIVE_IN_NBHD_OF_ONE_FIBER}.
\end{proof}

\begin{remark}\label{REM:PROBLEM_WITH_PROP}
{\rm Concerning Corollary~\ref{COR:TRANSITIVE_ON_G}, note that the non-empty intersection
$\gamma(U_1) \cap U_2$ may be disjoint from $\GRID$ (and hence it might be disjoint
from $\mathcal{J}_{A,B,C,D}$ as well).}
\end{remark}

Proposition~\ref{PROP:HITTING_G} below is the last ingredient in the proof of
Theorem~A. Notwithstanding its very elementary nature, this proposition is likely to
find further applications in the study of the dynamics associated with the
group $\AUTOGROUP$.

\begin{proposition}\label{PROP:HITTING_G}
For any open set $U$ intersecting the Julia set $\mathcal{J}_{A,B,C,D}$ of $\AUTOGROUP$ non-trivially,
there exists some element $\gamma \in \AUTOGROUP$ such that
$\gamma(U) \cap \GRID \neq \emptyset$.
\end{proposition}

\begin{proof}
Assume aiming at a contradiction that $\gamma(U)$ is disjoint from $\GRID$ for every $\gamma \in \AUTOGROUP$.
Then, for every $\gamma \in \AUTOGROUP$ we have that
\begin{align*}
\iota \circ \gamma(U) \subset \big(\mathbb{C} \setminus \{x_0,x_1\}\big) \times \big(\mathbb{C} \setminus \{y_0,y_1\}\big) \times \big(\mathbb{C} \setminus \{z_0,z_1\}\big),
\end{align*}
where $\iota: S_{A,B,C,D} \hookrightarrow \mathbb{C}^3$ denotes the inclusion.
Applying Montel's Theorem to each coordinate, this implies that the whole group
$\AUTOGROUP$ forms a normal family on $U$. This contradicts the assumption that $U \cap  \mathcal{J}_{A,B,C,D} \neq \emptyset$
and establishes the statement.
\end{proof}


\begin{proof}[Proof of Theorem A]
The proof is based on Baire's argument. First note that ${\mathcal
J}_{A,B,C,D}$ has the Baire property since it is a complete metric space as a
closed subset of a manifold. The topology in ${\mathcal J}_{A,B,C,D}$ is the
one inherited from the topology of $S_{A,B,C,D}$ and hence it is second
countable, i.e., there is a countable basis $\{V_{k}\}_{k=1}^{\infty}$ for the
topology of ${\mathcal J}_{A,B,C,D}$. By definition the open sets $V_{k}
\subset {\mathcal J}_{A,B,C,D}$ are given by
\begin{align}\label{EQN:DEF_UV}
V_{k} = U_k \cap {\mathcal J}_{A,B,C,D} \, ,
\end{align}
where $U_k$ is an open set of $S_{A,B,C,D}$. To prove the theorem, it suffices to show that the restriction to
${\mathcal J}_{A,B,C,D}$ of the action of $\AUTOGROUP$ is topologically transitive in the basis
$\{V_{k}\}_{k=1}^{\infty}$. Namely, given $k_1$ and $k_2$, we need to prove the existence of
$\gamma \in \AUTOGROUP$ such that $\gamma (V_{k_1}) \cap V_{k_2} \neq \emptyset$.
In fact, assuming the existence of these elements $\gamma$, for every~$n \in \N$, consider the set
$$
\bigcup_{\gamma \in \AUTOGROUP} \gamma^{-1}(V_{n})
$$
formed by all points in ${\mathcal J}_{A,B,C,D}$ whose orbit intersects
$V_{n}$. This set is clearly open since $V_{n}$ is so. It is
also dense in ${\mathcal J}_{A,B,C,D}$ since it intersects non-trivially
every set $V_{k}$ defining a basis for the topology of ${\mathcal J}_{A,B,C,D}$.
Taking then the intersection over~$n$
$$
\bigcap_{n=1}^\infty \bigcup_{\gamma \in \AUTOGROUP} \gamma^{-1}(V_{n})
$$
we obtain a $G_{\delta}$-dense subset of ${\mathcal J}_{A,B,C,D}$. By definition, the $\AUTOGROUP$-orbit
of any point in this intersection visits all the open sets $V_{k}$ so that these points have dense
orbits in ${\mathcal J}_{A,B,C,D}$.

It remains to show that for any two open sets $V_{k_1}$ and $V_{k_2}$ from our basis
there exists $\gamma \in \AUTOGROUP$ satisfying
$\gamma (V_{k_1}) \cap V_{k_2} \neq \emptyset$.

We start by working with the corresponding open sets $U_{k_1}$ and $U_{k_2}$ of
$S$.  We first use Proposition~\ref{PROP:HITTING_G} to find $\gamma_1, \gamma_2
\in \AUTOGROUP$ such that $\gamma_1(U_{k_1})$ and $\gamma_2(U_{k_2})$ each hit the
grid $\GRID$.  We can then use Proposition~\ref{PROP:FROM_ONE_IRREDUCIBLE_COMP_TO_ANOTHER} to find $\gamma_3 \in \AUTOGROUP$
such that $\gamma_3 \circ \gamma_1(U_{k_1})$ and $\gamma_2(U_{k_2})$ intersect the
same irreducible component of $\GRID$.  Without loss of generality, we
suppose it is $S_{x=x_0}$; i.e.\ that it is the first of the six irreducible components of $\GRID$ listed in (\ref{EQN:GRID}).

Since $S_{y=y_0}$ is transverse to $S_{x=x_0}$, we can choose a
sequence of points $\{q_j\}_{j=1}^\infty \subset S_{y=y_0}$ converging to $q \in S_{x=x_0} \cap
S_{y=y_0}$ that satisfies the hypotheses of Lemma
\ref{PROP:HITTING_OTHER_FIBERS-QuantitativeVersion}.  Moreover, by Lemma
\ref{COR:GX_JULIA}, $S_{y=y_0} \subset {\mathcal J}_{A,B,C,D}$ so each element of
the sequence is in ${\mathcal J}_{A,B,C,D}$.  We therefore find a point $q_N
\in {\mathcal J}_{A,B,C,D}$ and $\gamma_4, \gamma_5 \in \AUTOGROUP$ with
$\gamma_4(q_N) \in \gamma_3 \circ \gamma_1(U_{k_1})$ and $\gamma_5(q_N) \in \gamma_2(U_{k_2})$.
In other words,
\begin{align*}
q_N \in \gamma_4^{-1} \circ \gamma_3 \circ \gamma_1(U_{k_1}) \quad \bigcap \quad  \gamma_5^{-1} \circ \gamma_2(U_{k_2})  \quad \bigcap \quad {\mathcal J}_{A,B,C,D}.
\end{align*}
Since ${\mathcal J}_{A,B,C,D}$ is invariant under $\AUTOGROUP$, this proves that
\begin{align*}
\gamma_4^{-1} \circ \gamma_3 \circ \gamma_1(V_{k_1}) \cap \gamma_5^{-1} \circ \gamma_2(V_{k_2}) \neq \emptyset.
\end{align*}
We conclude that $\gamma(V_{k_1}) \cap V_{k_2} \neq \emptyset$ with
$\gamma = \gamma_2^{-1} \circ \gamma_5 \circ \gamma_4^{-1} \circ \gamma_3 \circ
\gamma_1$.
\end{proof}

\begin{remark}
{\rm After finding $\gamma_1, \gamma_2
\in \AUTOGROUP$ such that $\gamma_1(U_{k_1})$ and $\gamma_2(U_{k_2})$ each hit the
grid $\GRID$, it is
tempting to use Corollary \ref{COR:TRANSITIVE_ON_G} to find $\gamma_3$ with
\begin{align}\label{EQN:NONTRIVIAL_INTERSECTION}
\gamma_3(\gamma_1(U_{k_1})) \cap \gamma_2(U_{k_2}) \neq \emptyset.
\end{align}
However, this does not necessarily prove that $\gamma_3(\gamma_1(V_{k_1})) \cap \gamma_2(V_{k_2}) \neq \emptyset$ because the intersection
(\ref{EQN:NONTRIVIAL_INTERSECTION}) need not be in $\GRID$ and hence it
potentially might not contain any points of ${\mathcal J}_{A,B,C,D}$; see Remark~\ref{REM:PROBLEM_WITH_PROP}.  This is why we use the ``quantitative'' Lemma \ref{PROP:HITTING_OTHER_FIBERS-QuantitativeVersion} instead.}
\end{remark}


\begin{proof}[Proof of Theorem B]
Let us define a new ``grid'' $\GRID'$ using
Equation (\ref{EQN:GRID}), but this time we will use
\begin{align*}
x_0 = y_0 = z_0 = 0 \qquad \mbox{and}  \qquad x_1 = y_1 = z_1 = \sqrt{2}.
\end{align*}
These values are chosen so that 
\begin{align*}
g_x^2 |_{S_{x=x_0}} = {\rm id}, \quad g_y^2 |_{S_{y=y_0}} = {\rm id}, \quad g_z^2 |_{S_{z=z_0}} = {\rm id}, \qquad g_x^4 |_{S_{x=x_1}} = {\rm id}, \quad g_y^4 |_{S_{y=y_1}} = {\rm id}, \quad \mbox{and} \quad g_z^4 |_{S_{z=z_1}} = {\rm id}.
\end{align*}
As in Proposition \ref{PROP:HITTING_G}, if $U$ is any open set that intersects
${\mathcal J}_{A,B,C,D}$ non-trivially then there is an element $\gamma \in
\AUTOGROUP$ with $\gamma(U)$ intersecting $\GRID'$. In fact, the proof of
Proposition \ref{PROP:HITTING_G} does not use the choices that $x_0, x_1 \not \in
\mathcal{B}_x$, $y_0, y_1 \not \in \mathcal{B}_y$, or $z_0, z_1 \not \in
\mathcal{B}_z$ that were made in the construction of our original grid $\GRID$ so that it applies
equally well to $\GRID'$.

Conjugating by $\gamma$, if necessary, it then suffices to prove that shear fixed
points are dense in our newly chosen grid $\GRID'$.  We will prove it for
$S_{x=x_0}$ and $S_{x=x_1}$ and leave the completely analogous proofs for $S_{y=y_0},
S_{z=z_0}, S_{y=y_1}$ and $S_{z=z_1}$ to the reader.

Every point of $S_{x=x_0}$ is a fixed point for $g_x^2$.  We will show that all but finitely many of them are shear fixed points.
Clearly this assertion is, in turn, equivalent to showing that
the derivative $D(g_x)^2$ has two dimensional generalized eigenspace associated to eigenvalue $1$ but only one eigenvector associated to eigenvalue~$1$.
Considering $g_x^2$ as a mapping from $\mathbb{C}^3 \rightarrow \mathbb{C}^3$ we have
\begin{align*}
D(g_x^2)|_{x=0} =  \left[ \begin {array}{ccc} 1&0&0\\ \noalign{\medskip}2\,z-C&1&0
\\ \noalign{\medskip}B-2\,y&0&1\end {array} \right].
\end{align*}
If $z \neq C/2$ or $y \neq B/2$ then this matrix has generalized eigenspace of dimension $3$ associated to the eigenvalue $1$ but only two eigenvectors,
namely ${\bm e}_2=[0,1,0]$ and ${\bm e}_3 =[0,0,1]$.
Therefore, it suffices to prove that 
\begin{align}\label{EQN:CONDITION_ON_TANGENT_SPACE}
T_{p} S_{A,B,C,D} \neq {\rm span}({\bm e}_2,{\bm e}_3)
\end{align}
for every $p \in S_{x=x_0}$ bar some finite set. Taking the gradient of the defining equation for $S_{A,B,C,D}$ yields that
\begin{align*}
T_{p} S_{A,B,C,D} = {\rm ker} \left[\begin {array}{ccc} yz-A+2\,x&zx-B+2\,y&xy-C+2\,z\end {array}\right].
\end{align*}
Restricted to $x=x_0=0$ we can only have ${\bm e}_2 \in T_{p} S_{A,B,C,D}$ if $y=B/2$ and we can only have ${\bm e}_3 \in T_{p} S_{A,B,C,D}$ if $z=C/2$.  
Combined with $x= 0$ each of these conditions amounts to at most two points of $S_{x=x_0}$.  Therefore,
all but at most finitely many points of $S_{x=x_0}$ are shear fixed points of~$g_x^2$.

The situation for $S_{x=x_1}$ is essentially the same, except that one must work with $g_x^4$.
We leave the details to the reader.
\end{proof}


\begin{proof}[Proof of Theorem C]
Let $p_1$ and $p_2$ be arbitrary points in ${\mathcal J}_{A,B,C,D}$. We will show that for any neighborhoods
$U_1$ of $p_1$ and $U_2$ of $p_2$, there is a path in ${\mathcal J}_{A,B,C,D}$ from $U_1$ to $U_2$.
Theorem~C will immediately follow.

The grid $\GRID$ given in~(\ref{EQN:GRID}) is path connected and, by virtue of
Lemma~\ref{COR:GX_JULIA}, we have
$\GRID \subset {\mathcal J}_{A,B,C,D}$. Therefore, it suffices to find a path from 
$U_1$ to $\GRID$ and a path from $U_2$ to $\GRID$.
As the situation is symmetric, it suffices to consider $U_1$.

Since $p_1 \in U_1 \cap {\mathcal J}_{A,B,C,D}$, Proposition~\ref{PROP:HITTING_G}
gives some $\gamma \in \AUTOGROUP$ such that $\gamma(U_1) \cap
\GRID \neq \emptyset$.  Let $C$ be an irreducible component of
$\GRID$ with $\gamma(U_1) \cap C \neq \emptyset$.  Since we have chosen
the irreducible components of $\GRID$ to be smooth, $C$ is biholomorphic
to $\mathbb{C} \setminus \{0\}$.

Consider now the Riemann surface $\gamma^{-1}(C)$ contained in $S_{A,B,C,D} \subset \C^3$.
Since $\gamma$ is a holomorphic diffeomorphism of $S_{A,B,C,D}$, it follows that
$\gamma^{-1}(C)$ is again biholomorphic to $\mathbb{C} \setminus \{0\}$ and hence it is
uniformized by $\C$ and contained in $\C^3$. We claim that $\gamma^{-1}(C)$
intersects the grid $\GRID$. Indeed, if we had $\gamma^{-1}(C) \cap \GRID = \emptyset$,
the uniformization map from $\C$ to
$\gamma^{-1}(C)$ would yield a (non-constant) holomorphic map from $\C$ to $\C^3$ each of whose
coordinates omits two values in $\C$. Picard's Theorem would then imply that this map must be constant
and this is impossible.

Finally, note that $\gamma^{-1}(C) \subset {\mathcal J}_{A,B,C,D}$ since ${\mathcal J}_{A,B,C,D}$
is invariant by $\AUTOGROUP$ and $C\subset \GRID \subset {\mathcal J}_{A,B,C,D}$. Furthermore,
by construction, $\gamma^{-1}(C)$ also intersects $U_1$. Since $\gamma^{-1}(C)$ is path connected,
we can therefore find a path contained in $\gamma^{-1}(C) \subset {\mathcal J}_{A,B,C,D}$ going from
$U_1$ to $\GRID$. The proof of Theorem~C is complete.
\end{proof}

\section{Picard parameters and proof of Theorem D}\label{SEC:PICARD}

The parameters $(A,B,C,D) = (0,0,0,4)$ are quite special for at least two reasons:
\begin{itemize}
\item The surface $S_{(0,0,0,4)}$ has the maximal number of singularities among all cubic
surfaces and for this reason it is called the {\em Cayley Cubic}. They are at the four points
\begin{align}\label{EQN:SING_PTS_CAYLEY}
 \{(-2,-2,-2), \, (-2,2,2), \, (2,-2,2), \, (2,2,-2)\}.
\end{align}
\item  It was proved by Cantat-Loray \cite[Theorem 5.4]{cantat-2} that $\AUTOGROUP$ has an invariant affine structure
on $S_{A,B,C,D}$ if and only if $(A,B,C,D) = (0,0,0,4)$.  
\end{itemize}
An {\em affine structure} on a complex surface consists of a collection of coordinate charts whose
transition functions are (restrictions of) affine mappings of $\mathbb{C}^2$.   One says that a group $G$ {\em preserves an
affine structure} if the expression of each element of $G$ in the preferred collection of charts (associated to the specified
affine structure) consists again of affine
mappings.
Existence of the invariant affine structure dates back to work of Picard on the Painlev\'e~6 equation corresponding
to the parameters $(A,B,C,D) = (0,0,0,4)$. It is for this reason that the parameters 
$(A,B,C,D) = (0,0,0,4)$ are called the {\em Picard Parameters}.

From our point of view, this case is also very interesting as it will soon be clear. More importantly, however,
the information collected in the course of this discussion will enable us to prove
Proposition~\ref{PROP:GOODPARAMS_IN_COMPLETMENT_HYPERSURFACES}
in Section~\ref{SEC:FATOU}. Albeit a somewhat technical statement,
Proposition~\ref{PROP:GOODPARAMS_IN_COMPLETMENT_HYPERSURFACES} plays an important role in the proofs
of Theorems~G and~K.

Throughout this section we will typically drop the parameters from our
notation, writing $S \equiv S_{0,0,0,4}$, $\AUTOGROUP \equiv \AUTOGROUP_{0,0,0,4}$,
$\mathcal{J} \equiv \mathcal{J}_{0,0,0,4}$, and so on. The singular
locus of $S$ will be denoted by $S_{\rm sing}$.

\begin{proposition}\label{PROP:PICARD_PARAM_DISCRETE}
For the Picard parameters, $\AUTOGROUP$ acts locally discretely on any open 
$U \subset S$.  
\end{proposition}

We will use the existence of a semi-conjugacy
between the action of $\AUTOGROUP$ and the group action of monomial mappings on $\mathbb{C}^*
\times \mathbb{C}^*$, which we describe now (see for example \cite[Section 1.5]{cantat-2}).
Consider the following two matrix groups
\begin{align}\label{definition_Gamma2}
\tilde{\MATRIXGP}_2 := \{M \in {\rm SL}(2,\mathbb{Z}) \, : \, M \equiv {\rm Id} \,  {\rm mod} 2\} \qquad \mbox{and} \qquad \MATRIXGP_2 := \{[M] \in {\rm PSL}(2,\mathbb{Z}) \, : \, M \equiv {\rm Id} \,  {\rm mod} 2\}.
\end{align}
The square brackets around $M$ in the definition of $\MATRIXGP_2$ denote that we take the equivalence class modulo multiplication
by $\pm {\rm Id}$.   Note that  $\MATRIXGP_2$  is the famous congruence subgroup which has well-known generating set consisting of $[M_x]$ and $[M_y]$ where 
\begin{align}\label{EQN:MATRIX_GP_GENS}
M_x := \left[\begin{array}{cc} 1 & 2 \\ 0 & 1 \end{array}\right] \quad \mbox{and} \quad  M_y := \left[\begin{array}{cc} 1 & 0 \\ 2 & 1\end{array}\right].
\end{align}
See, for example, \cite[Section 16.3]{CONWAY} or Exercises 6 and 7 from \cite[Chapter 13]{ZAK}.   Meanwhile, $\tilde{\MATRIXGP}_2$ is 
generated by $M_x, M_y$, and $-{\rm Id}$.

There is a group isomorphism from $\MATRIXGP_2$ to $\AUTOGROUP$ induced by
sending $[M_x]$ to $g_x$ and $[M_y]$ to $g_y$.
More generally, we denote the image of any $[M] \in
\MATRIXGP_2$ under this isomorphism by $f_{[M]} \in \AUTOGROUP$.  This
isomorphism can be seen directly, but it also fits nicely within the
context of dynamics on character varieties; see \cite[Section 2.3]{cantat-2} for more details.

Associated with a matrix $M = \{ m_{ij} \}\in \tilde{\MATRIXGP}_2$ is a monomial mapping $\eta_M: \mathbb{C}^* \times
\mathbb{C}^* \rightarrow \mathbb{C}^* \times \mathbb{C}^*$ given by
\begin{align*}
\eta_M \left(u, v \right) = \left(u^{m_{11}} v^{m_{12}}, u^{m_{21}} v^{m_{22}}\right).
\end{align*}
Also let $\Phi: \mathbb{C}^* \times \mathbb{C}^* \rightarrow S$ be defined by
\begin{align*}
\Phi(u,v) = \left(-u - 1/u, -v - 1/v, -u/v - v/u \right).
\end{align*}
It turns out that $\Phi$ is a degree two orbifold cover. Furthermore, a straightforward verification
shows that that the critical points of
$\Phi$ are precisely the four points $(u,v) = (\pm 1, \pm 1)$ while the corresponding critical values
are the four points of $S_{\rm sing}$.

\begin{proposition}\label{PROP:SEMI_CONJUGACY}
$\Phi$ semi-conjugates the action of $\tilde{\MATRIXGP}_2$ on $\mathbb{C}^* \times \mathbb{C}^*$ to the action of
$\AUTOGROUP$ on $S$. More specifically, given $M \in \tilde{\MATRIXGP}_2$ and $(u,v) \in \mathbb{C}^* \times \mathbb{C}^*$,
we have
\begin{align}\label{EQN:SEMI_CONJ}
\Phi \circ \eta_M(u,v) = f_{[M]} \circ \Phi(u,v).
\end{align}
\end{proposition}

\begin{proof}
One can directly check Formula (\ref{EQN:SEMI_CONJ}) for the three generators of $\tilde{\MATRIXGP}_2$.
\end{proof}

\begin{proof}[Proof of Proposition \ref{PROP:PICARD_PARAM_DISCRETE}]
Suppose that there is an open $U \subset S$ and a sequence ${[M_n]} \in \MATRIXGP_2 \setminus \{[{\rm Id}]\}$  such that
$f_{[M_n]} |_U$ converges locally uniformly to the identity on $U$. Making $U$ smaller,
if necessary, we can assume that $U$
is evenly covered by $\Phi$ and that $V$ is one of the two connected components of $\Phi^{-1}(U)$. Then, up to appropriately
choosing the matrix $M_n$ representing $[M_n]$,  the monomial
maps $\eta_{M_n}$ converge locally uniformly to the identity on $V$.
However, this is impossible because $\tilde{\MATRIXGP}_2$ is a discrete subgroup of $SL(2,\mathbb{Z})$. Since
the action is question is linear, from the discrete character of $\tilde{\MATRIXGP}_2$ it follows that the induced
action on pairs $(\log|u|,\log|v|)$ is locally discrete as well.
\end{proof}

The semiconjugacy from Proposition \ref{PROP:SEMI_CONJUGACY} also allows us to completely determine the Julia set
associated with Picard parameters. Namely, we have:

\begin{proposition}\label{PROP:PICARD_PARAM_JULIA}
The Julia set ${\mathcal J}_{0,0,0,4}$ is the whole surface $S$.
\end{proposition}

The proof of Proposition~\ref{PROP:PICARD_PARAM_JULIA} will, however, require the following lemma:

\begin{lemma}\label{LEM:EIGENDIRECTIONS} The set formed by the union over all hyperbolic elements of $\tilde{\MATRIXGP}_2$
of the corresponding eigendirections is dense in $\mathbb{P}^1(\mathbb{R})$.
\end{lemma}

\begin{proof}
Conjugating by elements of $\tilde{\MATRIXGP}_2$, the problem reduces to considering
directions between $(1,0)^T$ and $(1,1)^T$.  Let $p$ be any prime number
and let $1 \leq q \leq p-1$ be any even number.  Then, there exist positive
integers $a$ and $b$ such that $b p - a q = 1$.  Since $q$ is even, $b$ is odd.
If $a$ is odd then we can replace $a$ and $b$ by $p+a$ and $q+b$, respectively.
This allows us to assume that
\begin{align*}
\left(\begin{array}{cc} p & a \\ q & b \end{array}\right) \in \tilde{\MATRIXGP}_2.
\end{align*}
Choosing $p$ sufficiently large, it is a consequence of the Perron-Frobenius
Theorem that this matrix has an eigenvector whose direction is arbitrarily
close to $(p,q)^T$. Finally, by appropriately choosing $q$, every vector
$(v_1,v_2)^T$ between $(1,0)^T$ and $(1,1)^T$ can be approximated. The lemma follows.
\end{proof}

\begin{proof}[Proof of Proposition \ref{PROP:PICARD_PARAM_JULIA}]
It suffices to prove that the Julia set $J$ for the monomial action on $\mathbb{C}^* \times \mathbb{C}^*$
is all of $\mathbb{C}^* \times \mathbb{C}^*$.
Indeed, suppose there is an open $U \subset S$ contained in the Fatou set for the action of $\AUTOGROUP$ on $S$.
Making $U$ smaller, if necessary, we can assume that $U$
is evenly covered by $\Phi$ and that $V$ is one of the two connected components of $\Phi^{-1}(U)$.
Because of the semi-conjugacy $\Phi$, the assumption on $U$ would imply
that $V$ is in the Fatou set for the monomial action of $\tilde{\MATRIXGP}_2$ on $\mathbb{C}^* \times \mathbb{C}^*$.

The unit torus $\mathbb{T}^2 \subset \mathbb{C}^* \times \mathbb{C}^*$ is invariant under the monomial
action of $\tilde{\MATRIXGP}_2$ with the hyperbolic
elements of $\tilde{\MATRIXGP}_2$ corresponding to Anosov mappings $\eta_M: \mathbb{T}^2 \rightarrow  \mathbb{T}^2$.
Therefore, $\mathbb{T}^2 \subset J$.
Moreover, for each hyperbolic $M \in \tilde{\MATRIXGP}_2$ the hyperbolic set $\mathbb{T}^2$ has stable and
unstable manifolds under $\eta_M$, namely:
\begin{align*}
\mathcal{W}^s_M(\mathbb{T}^2) & = \{(u,v) \in \mathbb{C}^* \times \mathbb{C}^* \, : \, (\log|u|,\log|v|)^T \,\,
\mbox{is a stable eigenvector for $M$}\}, \,\,  \mbox{and} \\
\mathcal{W}^u_M(\mathbb{T}^2) & = \{(u,v) \in \mathbb{C}^* \times \mathbb{C}^* \, : \, (\log|u|,\log|v|)^T \,\,
\mbox{is an  unstable eigenvector for $M$}\}.
\end{align*}
These invariant manifolds are of real-dimension three.  Note that $\mathcal{W}^s_M(\mathbb{T}^2)$ is
in the Julia set for the monomial mapping $\eta_M$ associated to $M$ and
$\mathcal{W}^u_M(\mathbb{T}^2)$ is in the Julia set for $\eta_{M}^{-1}$.
It follows from Lemma \ref{LEM:EIGENDIRECTIONS} that
the union of these stable and unstable manifolds is dense. Since $J$ is
closed, it must be all of $\mathbb{C}^* \times \mathbb{C}^*$.
\end{proof}

In the remainder of this section, we will focus on points stabilized by non-trivial elements of
$\AUTOGROUP$. As mentioned, the discussion below will allow us to prove
Proposition~\ref{PROP:GOODPARAMS_IN_COMPLETMENT_HYPERSURFACES}.

Let $S(\mathbb{R}) = S \cap \mathbb{R}^3$ denote the real slice of $S$.  It is well known
that $S(\mathbb{R}) \setminus S_{\rm sing}$ consists of one bounded component and three unbounded components
(see \cite{benedettogoldman}).
Let $S(\mathbb{R})_0$ denote the closure of the bounded component of $S(\mathbb{R}) \setminus S_{\rm sing}$.
It is straightforward to check that
\begin{align*}
S(\mathbb{R})_0 = S \cap [-2,2]^3 = \Phi(\mathbb{T}^2).
\end{align*}

For every $M \in \tilde{\MATRIXGP}_2$ and any fixed point $p \in \mathbb{C}^* \times \mathbb{C}^*$
of $\eta_M$, there exist suitable local coordinates (complexified angular coordinates) in which
$D \eta_M(p) = M$. In particular, the eigenvalues of $D\eta_M(p)$ and of $M$ coincide. Moreover,
by noticing that $(\eta_M)^k = \eta_{M^k}$ for every integer $k > 1$, it follows that the analogous
statement holds for periodic points as well.

Note also that when $M \in \tilde{\MATRIXGP}_2$ is parabolic, $\eta_M$ may have fixed
points away from of the real torus $\mathbb{T}^2 \subset \mathbb{C}^* \times
\mathbb{C}^*$. However, by the preceding discussion, the derivative
$D\eta_m(p)$ at such a fixed point will always have eigenvalues equal to $\pm
1$.  When $M \in \tilde{\MATRIXGP}_2$ is hyperbolic, any fixed point $p$ of $\eta_M$ is
on $\mathbb{T}^2$ and the fixed point is a saddle, with one of the eigenvalues
of $D \eta_M(p)$ having absolute value less than one and the other having
absolute value greater than one.

\begin{lemma}\label{LEM:EIGENVALUES_FIXED_POINTS}
Let $M \in \tilde{\MATRIXGP}_2$ and let $f_{[M]}: S \rightarrow S$ be the associated element of $\AUTOGROUP$.
If $p$ is a smooth point of $S$ and a fixed point of $f_{[M]}$ then the eigenvalues of $Df(p)$ 
have the same absolute values as the eigenvalues of $M$.
\end{lemma}


\begin{proof}
The critical values of $\Psi$ are precisely the singular point of $S$, which
are permuted by elements of $\AUTOGROUP$.  Therefore, $\Psi^{-1}(p) = \{q_1,q_2\}$ 
and they will either each be a fixed point for 
$\eta_M$ or they will form a period two cycle for $\eta_M$. In either case $(Df_{[M]}(p))^2$ 
and $D\eta_M(q_2) D\eta_M(q_1)$ will be conjugate matrices and hence have the same eigenvalues.
Meanwhile $D\eta_M(q_2) D\eta_M(q_1)$ and $M^2$ have the same eigenvalues, so the result follows.
\end{proof}

\begin{proposition}\label{PROP:SADDLE_POINTS_NOT_DENSE_PICARD}
For the Picard parameters, whenever $M$ is a hyperbolic matrix, every fixed point
of the corresponding mapping $f_{[M]}$ not lying in $S_{\rm sing}$ must be a hyperbolic saddle. In addition,
these fixed points are all located on $S(\mathbb{R})_0 = \Phi(\mathbb{T}^2)$.

In particular, there is no dense subset ${\mathcal J}^*_{0,0,0,4} \subset {\mathcal J}_{0,0,0,4}$
consisting of points with hyperbolic stabilizers.
\end{proposition}

\begin{proof}
It follows from Lemma \ref{LEM:EIGENVALUES_FIXED_POINTS} and from the discussion in the paragraph
before this lemma that a hyperbolic fixed point of an arbitrary element in $\AUTOGROUP$, in fact, must be
a fixed point of some mapping $f_{[M]}$, where $M$ is hyperbolic. Every such fixed point is therefore a hyperbolic
saddle and is contained in $S(\mathbb{R})_0$. Clearly $S(\mathbb{R})_0$ is a proper subset of $S$ which,
in turn, coincides with $\mathcal{J}$ in view of Proposition~\ref{PROP:PICARD_PARAM_JULIA}. The proposition
follows.
\end{proof}

\begin{proof}[Proof of Theorem D]
It follows directly from the combination of Propositions~\ref{PROP:PICARD_PARAM_DISCRETE},
\ref{PROP:PICARD_PARAM_JULIA}, and~\ref{PROP:SADDLE_POINTS_NOT_DENSE_PICARD}.
\end{proof}

Recall that, by construction, every mapping $f_{[M]}: S \rightarrow S$ is the restriction of a polynomial
diffeomorphism of $\mathbb{C}^3$ which will be denoted by
$F_{[M]}: \mathbb{C}^3 \rightarrow \mathbb{C}^3$. These maps $F_{[M]}$ leave invariant all the surfaces
of the form $S_{0,0,0,D}$, with $D \in \mathbb{C}$. From this it follows that if $p \in
S_{0,0,0,4}$ is a fixed point of $f_{[M]}$ then two of the eigenvalues of the $3
\times 3$ matrix $DF_{[M]}(p)$ are the same as those of $Df_{[M]}(p)$ and the third
eigenvalue is $1$, provided that $p$ is a regular point of $S_{0,0,0,D}$. Owing to
Lemma~\ref{LEM:EIGENVALUES_FIXED_POINTS}, we conclude that two
of the eigenvalues of $DF_{[M]}(p)$ have the same absolute values as the eigenvalues
of $M$ and the remaining eigenvalue is $1$.

The following proposition describes the eigenvalues of $DF_{[M]}(p)$ at singular points $p$ of $S$.
The fact that the eigenvalues of $M$ are squared is essentially the same phenomenon that occurs for the classical
one-dimensional Chebeyshev map, and we are grateful to Micha{\l} Misiurewicz for explaining it to us.

\begin{proposition}\label{PROP:SING_PTS_EIGVALS}
Let $(A,B,C,D)$ be the Picard Parameters $(0,0,0,4)$.  For any $M \in \tilde{\MATRIXGP}_2$ let $f_{[M]}: S \rightarrow S$
be the corresponding element of $\AUTOGROUP$ and let $F_{[M]}: \mathbb{C}^3 \rightarrow \mathbb{C}^3$ be its extension to $\mathbb{C}^3$.  For any $p \in S_{\rm sing}$
two of the eigenvalues of 
$D F_{[M]}(p)$ are the squares of the eigenvalues of $M$ and the remaining eigenvalue is $1$.
\end{proposition}

\begin{proof}
The proof relies upon the semiconjugacy $\Psi$ from Proposition~\ref{PROP:SEMI_CONJUGACY}.
We apply it to points  $(u,v) = ({\rm e}^{i \theta},{\rm e}^{i \phi}) \in \mathbb{T}^2$ and abuse notation
slightly by
writing
\begin{align*}
(x,y,z) = \Psi(\theta,\phi) = (-2\cos \theta, -2 \cos \phi, -2 \cos(\theta-\phi)).
\end{align*}
The points $(0,0), (0,\pi), (\pi,0),$ and $(\pi,\pi)$ map by $\Psi$ to the
singular points in (\ref{EQN:SING_PTS_CAYLEY}) in the respective order that they are listed there.

Here we focus on the singular point $p=(-2,-2,-2) = \Psi(0,0)$. The minor adaptations required for the
other singular points essentially amount to some sign modifications in the equations below and thus can
safely be left to the reader.

Since $M \in  \tilde{\MATRIXGP}_2$ we have $\det(M) = 1$.  Using this, the characteristic polynomial for $M^2$ is
\begin{align*}
P_{M^2}(x) = x^2 - (m_{11}^2+m_{22}^2+2 m_{12} m_{21}) \ x + 1,
\end{align*}
where $m_{jk}$ denotes $jk$-th entry of $M$.


Let $N = DF_{[M]}(p)$ and 
recall from Section \ref{SEC:VOLUME_FORM} that any element of 
$\AUTOGROUP_{A,B,C} = \langle g_x, g_y, g_z \rangle~<~{\rm Aut}(\mathbb{C}^3)$ preserves
the Euclidean volume form on $\mathbb{C}^3$.  This implies that ${\rm det}(N) = 1$.
Moreover, $F_{[M]}(S_D)~=~S_D$ for every $D \in \mathbb{C}$, implying that $Q \circ F_{[M]} (x,y,z) = Q (x,y,z)$ for
the polynomial $Q(x,y,z) = x^2+ y^2 + z^2 +xyz$.
This gives that one of the eigenvalues of $N$ equals~$1$.

Hence,the characteristic polynomial of $N$ is
\begin{align*}
P_{N}(x) = x^3 - (n_{11}+n_{22}+n_{33}) \ x^2 + (n_{11}+n_{22}+n_{33}) \ x -1 \,,
\end{align*}
where $n_{jk}$ denotes $jk$-th entry of $N$.
In the sequel we will show that 
\begin{align}\label{EQN:DESIRED_RELATION_M_N}
n_{11}+n_{22}+n_{33} = m_{11}^2+m_{22}^2+2 m_{12} m_{21}+1 \, .
\end{align}
This will imply that $P_N(x) = P_{M^2}(x)(x-1)$ therefore completing the proof of
Proposition~\ref{PROP:SING_PTS_EIGVALS}.

To begin, consider the $x$-coordinate of the semi-conjugacy (\ref{EQN:SEMI_CONJ}):
\begin{align*}
-2 \cos(m_{11} \theta + m_{12} \phi) = F_{[M],1}(-2\cos \theta,-2\cos \phi,-2\cos(\theta-\phi)),
\end{align*}
where we have added the subscript $1$ to denote the first coordinate of $F_{[M]}$.
Setting $\phi = 0$ and taking the partial derivative with respect to $\theta$ yields
\begin{align*}
2 \sin(m_{11} \theta) m_{11} = \frac{\partial F_{[M],1}}{\partial x}(-2\cos \theta, -2, -2 \cos \theta) 2\sin \theta + \frac{\partial F_{[M],1}}{\partial z}(-2\cos \theta, -2, -2 \cos \theta) 2\sin \theta.
\end{align*}
Next, for $\theta \neq 0$, we divide both sides of the above equation by $2 \theta$ so as to obtain
\begin{align*}
 \frac{\sin(m_{11} \theta)}{m_{11} \theta}  m_{11}^2 = \frac{\partial F_{[M],1}}{\partial x}(-2\cos \theta, -2, -2 \cos \theta) \frac{\sin \theta}{\theta} + \frac{\partial F_{[M],1}}{\partial z}(-2\cos \theta, -2, -2 \cos \theta) \frac{\sin \theta}{\theta} \, .
\end{align*}
Now it suffices to take the limit as $\theta$ goes to $0$ to conclude that
$m_{11}^2 = n_{11} + n_{13}$.

Similarly, setting $\theta = 0$ and doing the analogous computation involving partial derivatives with respect
to $\phi$ yields
$m_{12}^2 = n_{12} + n_{13}$.
Finally, the analogous computation with $\phi = \theta$ lead to
$(m_{11}+m_{12})^2 = n_{11}+n_{12}$.
The three previous equations can be solved for $n_{11}$ to find
\begin{align}\label{EQN:N11}
n_{11} = m_{11}^2 + m_{11} m_{12}.
\end{align}
The same computations with the second and third coordinate of the semi-conjugacy~(\ref{EQN:SEMI_CONJ}) yield
\begin{align} \label{EQN:N22_N33}
n_{22} = m_{22}^2 + m_{22} m_{21}, \qquad \mbox{and} \qquad 
n_{33} = (m_{11} - m_{21}) (m_{22} - m_{12}). 
\end{align}
Combined with the fact that ${\rm det}(M) = 1$, Equations~(\ref{EQN:N11}) and~(\ref{EQN:N22_N33})
imply that the condition expressed by~(\ref{EQN:DESIRED_RELATION_M_N}) holds. The proof of the proposition
is completed.
\end{proof}


\begin{corollary}\label{COR:C3_EIGVALS}
Let $(A,B,C,D)$ be the Picard Parameters $(0,0,0,4)$. Assume that $M \in
\tilde{\MATRIXGP}_2$ is hyperbolic. Denote by $f_{[M]}: S \rightarrow S$ the element of $\AUTOGROUP$ associated with
$M$ and let $F_{[M]}: \mathbb{C}^3 \rightarrow \mathbb{C}^3$ be the corresponding extension of $f_{[M]}$
to $\mathbb{C}^3$. If $p \in S$ is a fixed point of $F_{[M]}$ then 
$D F_{[M]}(p)$ has one eigenvalue of modulus less than one, one eigenvalue equal to one, and one eigenvalue
of modulus greater than one.
\end{corollary}

\section{Existence of Fatou Components and Proof of Theorem E}\label{SEC:PROOF_THEOREME}

Recall from Section \ref{SUBSEC:STRATEGY} that for fixed choice of $(A,B,C)$ the equations for $s_x, s_y$, and $s_z$ can be interpreted as polynomial
automorphisms of $\mathbb{C}^3$.  We denote the group of automorphisms of $\mathbb{C}^3$ generated by $s_x, s_y$, and $s_z$ by $\AUTOGROUP^\pm_{A,B,C}$.
We can consider the Fatou set of this action on $\mathbb{C}^3$ and denote it by $\BIGFATOU^\pm_{A,B,C}$.
As in Section~\ref{SUBSEC:FATOU_JULIA}, convergence to infinity is again allowed in our definition of normal families.

The first step toward proving Theorem~E will be Proposition~\ref{PROP:FATOU_C3}, below, where 
for any choice of parameters $(A,B,C)$ will provide a point $p_0 \in \mathbb{C}^3$
and $\epsilon > 0$ 
so that the ball $B_\epsilon(p_0)$ of radius $\epsilon$ around $p_0$ is contained in
$\BIGFATOU_{A,B,C}$.   Then to 
obtain $4$-tuples of parameters
$(A,B,C,D) \in \mathbb{C}^4$ for which the
action of $\AUTOGROUP_{A,B,C,D}$ on $S_{A,B,C,D}$ has non-empty Fatou set
$\mathcal{F}_{A,B,C,D}$, it will be enough to select $D$ so that
$S_{A,B,C,D} \cap B_\epsilon(p_0) \neq \emptyset$; see Corollary \ref{COR:FATOU_EXISTENCE}.


\begin{proposition}\label{PROP:FATOU_C3}
For any choice of parameters $(A,B,C) \in \mathbb{C}^3$, let  $r={\rm max}\{|A|,|B|,|C|\}$. Next, given
$R > 2+\sqrt{r}$, let
\begin{align}\label{EQN:FATOU_EPS-I}
\epsilon = {\rm min}\{R-(2+\sqrt{r}),R+1-\sqrt {4\,R+r+1}\} > 0.
\end{align}
If $|u| = R$  and $p_0 = (u,u,u) \in \mathbb{C}^3$ then the open ball
$B_{\epsilon}(p_0)$ is contained in the Fatou set $\BIGFATOU^\pm_{A,B,C}$ for the action of
$\AUTOGROUP^\pm_{A,B,C}$ on $\mathbb{C}^3$. In particular $\BIGFATOU^\pm_{A,B,C} \neq \emptyset$.
\end{proposition}

The idea for the proof of Proposition~\ref{PROP:FATOU_C3} comes from the papers
of Bowditch \cite{bowditch}, Tan, Wong, and Zhang \cite{TWZ},
Maloni, Palesi, and Tan \cite{MPT}, and  Hu, Tan, and Zhang
\cite{Hu}.  See Section \ref{SUBSEC:TEICH} for more details.

\begin{proof}
Let $p = (x,y,z) \in B_\epsilon(p_0)$ and note that this implies that
each coordinate of $p$ has modulus larger than $2+\sqrt{r}$. 
 We will show that for any integer $k \geq 1$, and any reduced word $w_k w_{k-1} \ldots w_1$ in
the mappings $s_x, s_y$, and $s_z$ that each
coordinate of $w_k w_{k-1} \ldots w_1(p)$ has modulus at least as large as the
corresponding coordinate for $w_{k-1} \ldots w_1(p)$.
In particular, this will imply that
\begin{align*}
w_k w_{k-1} \ldots w_1 \left( B_\epsilon(p_0) \right) \subset \left(\mathbb{C} \setminus \overline{\mathbb{D}_{2+\sqrt{r}}(0)}\right)^3
\end{align*}
for any such word of any length $k \geq 1$.
Applying Montel's Theorem to each coordinate implies that the action of $\AUTOGROUP^\pm$ is normal on~$B_\epsilon(p_0)$
so that the statement follows.

We first check that our claim holds for $k = 1$. Consider the involution $s_x$ and the
points $p$ and~$s_x(p)$. Clearly the coordinates $y$ and $z$ of these two points coincide. 
To show that the modulus of the $x$ coordinate of $s_x(p)$ is strictly larger 
than the modulus of the $x$ coordinate of $p$, note that
\begin{align}\label{EQN:CHECK1}
|\pi_x(s_x(p))| &= |-yz - x +A| > (R-\epsilon)^2-(R+\epsilon)-r \geq R+\epsilon > |\pi_x(p)| \, ,
\end{align}
where the second inequality follows from the assumption that $\epsilon \leq
R+1-\sqrt {4\,R+r+1}$. Indeed, this condition can be reformulated as $R - \epsilon + 1 \geq
\sqrt {4\,R+r+1}$ which, by taking squares on both sides, leads right away to the inequality in question.
Naturally, analogous
estimates hold with respect to the coordinates $y$ or $z$ when
$s_x$ is replaced by $s_y$ and $s_z$. Therefore, we have shown that for every point $p \in B_\epsilon(p_0)$,
applying $s_x, s_y,$ or $s_z$ to $p$ strictly increases the modulus of one of the coordinates while
leaving the other two coordinates unchanged.

We now prove the claim for arbitrary $k \geq 2$ by means of contradiction.  We
therefore assume that $k \geq 2$ is the smallest index for which there is a
reduced word $w_k w_{k-1} w_{k-2} \ldots w_1$ such that some coordinate of $w_{k-1} w_{k-2} \ldots
w_1(p)$ has has modulus strictly larger than the corresponding coordinate of
$w_{k} w_{k-1} w_{k-2} \ldots w_{1}(p)$.
Note also that taking $k$ to be minimal
implies that each coordinate of $w_{k-1} \ldots w_1(p)$ has modulus greater than or equal to 
the minimal modulus of a coordinate of $p$ which, in turn, exceeds $2+\sqrt{r} = 2+\sqrt{{\rm max}\{|A|,|B|,|C|\}}$.
Let
\begin{align*}
(x,y,z) = w_{k-2} \ldots w_{1}(p), \qquad (x',y',z') = w_{k-1} \ldots w_{1}(p), \quad \mbox{and} \quad (x'',y'',z'') = w_{k} \ldots w_{1}(p).
\end{align*}
(If $k= 2$ then we interpret $w_{k-2} \ldots w_{1}$ as the identity mapping.)
In particular, the preceding ensures that ${\rm min}\{|x'|,|y'|,|z'|\} \geq 2+\sqrt{r}$. On the other hand,
since the word $w$ is reduced and all generators $s_x$, $s_y$, $s_z$ are involutions, we must have
$w_{k-1} \neq w_{k}$. Without loss of generality, we can then suppose $w_{k-1} = s_x$
and $w_{k} = s_y$. This yields
\begin{align*}
(x',y',z') = (-x - yz+A, y,z),  \qquad \mbox{and}  \qquad (x'',y'',z'') = (x', -y' -x'z'+B, z').
\end{align*}

Our assumption on $k$ implies $|x| \leq |x'|$ and $|y'| > |y''|$. Therefore, 
\begin{align}\label{EQN:INEQUALITIES-I}
2|x'| &\geq |x'+x| = |-yz+A| = |-y'z'+A| \qquad \mbox{and} \\
2|y'| &> |y'+y''| = |-x'z'+B|. \nonumber
\end{align}
We now split the discussion in two cases. Assume first that $|x'| \geq |y'|$.
Then, the second inequality from~(\ref{EQN:INEQUALITIES-I}) gives
\begin{align*}
|x'z'| - |B| \leq 2|y'| \leq 2|x'|.
\end{align*}
In turn, moving $|B|$ to the right side, dividing by $|x'|$, and recalling that $|x'| \geq 2+\sqrt{r}$ leads to
\begin{align*}
|z'| \leq 2 + \frac{|B|}{|x'|} < 2 + \frac{r}{2+\sqrt{r}} < 2 + \sqrt{r}.
\end{align*}
This is impossible since ${\rm min}\{|x'|,|y'|,|z'|\} \geq 2+\sqrt{r}$. If we consider now the case where
$|y'| > |x'|$, we just need to use the first inequality from (\ref{EQN:INEQUALITIES-I}) to
similarly show that
$$
|z'| < 2 +  \frac{|A|}{|y'|} < 2+\sqrt{r} \, .
$$
Thus, in any event, we obtain a contradiction that
proves our initial claim.

In summary, for any $p \in B_\epsilon(p_0)$ we have shown
that for any integer $k \geq 1$, and any reduced word $w_k w_{k-1} \ldots w_1$ in
the mappings $s_x, s_y$, and $s_z$ that each
coordinate of $w_k w_{k-1} \ldots w_1(p)$ has modulus at least as large as the
corresponding coordinate for $w_{k-1} \ldots w_1(p)$.   As already pointed out, this implies that the ball $B_\epsilon(p)$ must therefore lie in the Fatou set
of the action of $\AUTOGROUP^\pm_{A,B,C}$ on $\mathbb{C}^3$. 
\end{proof}

\begin{corollary}\label{COR:FATOU_EXISTENCE}  For any parameters $(A_0,B_0,C_0) \in \mathbb{C}^3$
suppose that $p_0 = (u,u,u)$ with $|u| > 2+\sqrt{r}$,
where $r$ is given as in
Proposition \ref{PROP:FATOU_C3}. Let $D_0$ be chosen so that $p_0 \in S_{A_0,B_0,C_0,D_0}$.   

Then there is some $\delta > 0$ such that for all parameters $(A,B,C,D) \in \mathbb{B}_\delta (A_0,B_0,C_0,D_0)
\subset \C^4$ the Fatou set
$\mathcal{F}_{A,B,C,D}$ for the action of $\AUTOGROUP_{A,B,C,D}$ on $S_{A,B,C,D}$ is non-empty.
\end{corollary}

\begin{proof}
The condition on $\epsilon > 0$ given in (\ref{EQN:FATOU_EPS-I}) depends
continuously on $r={\rm max}\{|A|,|B|,|C|\}$.  Therefore there exists
$\delta_0 > 0$ and $\epsilon_0 > 0$ such that if $(A,B,C) \in
\mathbb{B}_{\delta_0}(A_0,B_0,C_0) \subset \C^3$ then $B_{\epsilon_0}(p_0) \subset
\BIGFATOU^\pm_{A,B,C} \subset \C^3$. On the other hand, the point $(x,u,u)$ lies in the surface $S_{A,B,C,D}$ where
\begin{align*}
D = x^2+2u^2+xu^2-Ax-Bu-Cu.
\end{align*}
This polynomial is monic and non-constant in $x$ so that its roots vary continuously with $(A,B,C,D)$.  Since
it has a root at $x=u$ when $(A,B,C,D) = (A_0,B_0,C_0,D_0)$ we can find some  
$0 < \delta < \delta_1$ such that if $(A,B,C,D) \in \mathbb{B}_{\delta}(A_0,B_0,C_0,D_0)$ 
then 
\begin{align*}
(x,u,u) \in B_{\epsilon_0}(p_0) \cap S_{A,B,C,D} \, \subset \, \BIGFATOU^\pm_{A,B,C} \cap S_{A,B,C,D} \, \subset \, \mathcal{F}_{A,B,C,D}.
\end{align*}
\end{proof}

The proof of Theorem E will be a quick application of Corollary~\ref{COR:FATOU_EXISTENCE}
combined with the following elementary lemma whose proof we leave to the reader.

\begin{lemma}\label{windingnumber=2}
For every $D \in \C \setminus \{ 4\}$, the polynomial $q(u) = u^3 + 3u^2 = D$ has a solution 
with modulus strictly larger than~$2$.
\end{lemma}

\begin{proof}[Proof of Theorem E]

Consider first the Punctured Torus Parameters $A=B=C=0$.
The condition for a point $p_0 = (u,u,u)$ to lie in $S_{0,0,0,D_0}$ is
$u^3 + 3u^2 = D_0$.  If $D_0 \neq 4$, 
Lemma~\ref{windingnumber=2} ensures 
that there is a 
point $p_0 = (u, u, u) \in S_{0,0,0,D_0}$ with $\vert u \vert >2$.
Corollary~\ref{COR:FATOU_EXISTENCE} ensures the existence of
$\delta > 0$ such that for
all $(A,B,C,D) \in \mathbb{B}_\delta((0,0,0,D_0))$ the Fatou set
$\mathcal{F}_{A,B,C,D}$ is non-empty. This establishes the first part of Theorem~E.

Now consider the Dubrovin-Mazzocco parameters
$A(a) = B(a) = C(a) = 2a+4,$ and $D(a) = -(a^2 + 8a +8)$
for $a \in (-2,2)$. Let us denote the surface $S_{A,B,C,D}$ at these parameters
by $S_a$. The condition for $(u,u,u)$ to belong to $S_a$ is given by
\begin{align*}
q_a(u) = {u}^{3}-3\, \left( 2\,a+4 \right) u+3\,{u}^{2}+{a}^{2}+8\,a+8 = 0 \, .
\end{align*}
A direct calculation shows that if you substitute $u = -(2+\sqrt{2\,a + 4}) = -(2+\sqrt{r})$
into $q_a(u)$ the result is positive. Hence, there is a real $u_0 < -(2+\sqrt{r})$
satisfying $q(u_0) = 0$. Hence, again Corollary~\ref{COR:FATOU_EXISTENCE} implies that for any $a \in (-2,2)$
there exists $\delta > 0$ such that for
all $(A,B,C,D) \in \mathbb{B}_\delta((A(a),B(a),C(a),D(a)))$ the Fatou set $\mathcal{F}_{A,B,C,D}$ is non-empty.
The proof of Theorem~E is complete.
\end{proof}

\section{Locally non-discrete dynamics in $\AUTOGROUP_{A,B,C,D}$.}\label{SEC:LOCALLY_NONDISCRETE}

Let $M$ be a  (possibly open) connected complex manifold and consider a group $G$ of holomorphic diffeomorphisms of $M$.
The group $G$ is said to be {\it locally non-discrete} on an open $U \subset M$
if there is a sequence of maps $\{ f_n \}_{n=0}^{\infty} \in G$ satisfying the following conditions (see for example \cite{REBELO_REIS}):
\begin{enumerate}
 \item For every $n$, $f_n$ is different from the identity.
  \item The sequence of maps $f_n$ converges uniformly to the identity on compact subsets of $U$.
\end{enumerate}
If no such sequence of maps $\{ f_n \}_{n=0}^{\infty} \in G$ exists then $G$ is said to be {\it locally discrete on $U$}.
The reader will note that a group $G$ may be locally non-discrete on one open set $U$ and locally discrete on a disjoint open set $V$.


\vspace{0.1in}
We remind the reader from Section \ref{SUBSEC:LOCALLY_DISCR} that we denote the open subset of $S_{A,B,C,D}$ on which $\AUTOGROUP_{A,B,C,D}$ is locally
non-discrete (the ``locally non-discrete locus'') by $\nondiscrete_{A,B,C,D}$ and its complement (the ``locally discrete locus'') by $\discrete_{A,B,C,D}$.
\vspace{0.1in}

Let us return for a while to a more general context.
It is convenient to begin our discussion with Proposition~\ref{PROP:CONV_TO_ID}
below. This proposition constitutes a simple general result of which specific
variants appear in \cite[p. 9-10]{REBELO_REIS} and in \cite[Sec.
3]{LORAY_REBELO} while the main idea dates back to Ghys \cite{GHYS}.  Both Proposition~\ref{PROP:CONV_TO_ID}
and Lemma \ref{tangentidentitycloseidentity}, below, are stated in the wider context of pseudogroups $G$ of holomorphic maps of open sets of $M$
to $M$.  In this context, the definition of $G$ being locally non-discrete on $U$ becomes:
\begin{itemize}
  \item[(0)] The open set $U$ is contained in the domain of definition of $f_n$ (as an element of the pseudogroup $G$), for every~$n$.
  \item[(1)] For every~$n$, the restriction of $f_n$ to $U$ is different from the identity.
  \item[(2)] The sequence of maps $f_n$ converges uniformly to the identity on compact subsets of $U$.

\end{itemize}
Condition (0) is required for Conditions (1) and (2) to make sense and Condition (1) has been modified because of the possibility
that $U$ not be connected.

\vspace{0.05in}

Consider a ball $B_{\epsilon} (0) \subset \C^n$
of radius $\epsilon >0$ around the origin of $\C^n$. Assume we are given local holomorphic diffeomorphisms
$F_1, F_2 : B_{\epsilon} (0) \rightarrow \C^n$ and denote by $G$ the pseudogroup of maps from $B_{\epsilon} (0)$ to $\C^n$
generated by $F_1, F_2$. Naturally the inverses of $F_1$, $F_2$ are respectively denoted by $F_1^{-1}$ and $F_2^{-1}$. In what
follows we can assume without loss of generality that the domain of definition of $F_1^{-1}$ and $F_2^{-1}$ as elements
of $G$ is non-empty. Let us then define a sequence $S(n)$ of sets of elements in $G$ by letting
$S(0) = \{ F_1, F_1^{-1}, F_2, F_2^{-1} \}$. The sets $S(n)$ are now inductively defined by stating that
$S(n+1)$ is constituted by all elements of the form $[\gamma_i, \gamma_j] = \gamma_i \circ \gamma_j \circ \gamma_i^{-1} \circ
\gamma_j^{-1}$ with $\gamma_i, \gamma_j \in S(n)$. Note that the construction of these elements is so far purely formal
in the sense that the domain of definition (contained in $B_{\epsilon} (0)$) of diffeomorphisms in $S(n)$ viewed as elements
of $G$ may be empty. Nonetheless, we have:

\begin{proposition}\label{PROP:CONV_TO_ID}  Given $\epsilon > 0$, there is $K = K(\epsilon) > 0$ such that,
if
$$
\left\{ \sup_{z \in B_{\epsilon} (0)} \Vert F_1 (z) - z \Vert \; , \; \sup_{z \in B_{\epsilon} (0)} \Vert F_2 (z) - z \Vert
\right\} < K \, ,
$$
then the following hold:
\begin{itemize}
  \item[(1)] For every $n$ and every $\gamma \in S(n)$, the domain of definition of $\gamma$ as element in $G$
  contains the ball $B_{\epsilon/2} (0) \subset \C^n$ of radius~$\epsilon /2$ around the origin.

  \item[(2)] Furthermore, if $\gamma$ belongs to $S(n)$ then we have
$\sup_{p \in B_{\epsilon/2} (0)} \Vert\gamma (p) - p \Vert \leq \frac{K}{2^n}$.
\end{itemize}
\end{proposition}

Note that, in general, the above proposition falls short of implying
that the pseudogroup $G$ is locally non-discrete since the sets $S(n)$ may degenerate so as to
only contain the identity map.

As mentioned, variants of Proposition~\ref{PROP:CONV_TO_ID} can be found in the literature and,
to the best of our knowledge, the idea goes back to Ghys in \cite{GHYS}. Otherwise, its proof is
nearly identical to that from \cite[p. 9-10]{REBELO_REIS}, so we provide only a
sketch.

\begin{proof}[Proof of Proposition \ref{PROP:CONV_TO_ID}:]
The proof is based on the following estimate from
\cite[Lemma 3.0]{LORAY_REBELO}.  Let $B_\epsilon( 0) \subset \mathbb{C}^n$ be an open ball and suppose
$f_1, f_2: B_\epsilon( 0) \rightarrow \mathbb{C}^n$
are holomorphic local diffeomorphisms.  If
\begin{align}\label{EQN:HYPOTH_CLOSE_TO_ID}
\max \left\{\sup _{z \in B_{\epsilon}}\left\|f_{1}^{ \pm 1}(z)-z\right\|,
\sup _{z \in B_{\epsilon}}\left\|f_{2}^{ \pm 1}(z)-z\right\|\right\} \leq K
\end{align}
for some $K > 0$, then for any $\tau > 0$ satisfying $4K + \tau < \epsilon$ the commutator $[f_1,f_2]$ is defined
on the ball of radius $\epsilon-4K -\tau$ and satisfies
\begin{align}\label{EQN:CONCLUSION_CLOSE_TO_ID}
\sup _{z \in B_{\epsilon-4 \delta-\tau}}\left\|\left[f_{1}, f_{2}\right](z)-z\right\| \leq \frac{2}{\tau}
\sup _{z \in B_{\epsilon}}\left\|f_{1}(z)-z\right\| \cdot \sup _{z \in B_{\epsilon}}\left\|f_{2}(z)-z\right\|.
\end{align}
Starting with $S(0)$, we choose $\tau=\tau_0=K=K_0$ and $\epsilon_1 = \epsilon -8K$ so that the preceding yields
$$
\sup_{z \in B_{\epsilon_1}}  \left\|\gamma (z)-z\right\| \leq \frac{K}{2} \, ,
$$
for every $\gamma \in S(1)$. Now inductively setting $K_i = K/2^i$, $\tau_i = 4 K_i$ it is straightforward to conclude
that
$$
\sup_{z \in B_{\epsilon_n}}  \left\|\gamma (z)-z\right\| \leq \frac{K}{2^n} \, ,
$$
for every $\gamma \in S(n)$, where
\begin{align}\label{EQN:EPSILON_K_LINEAR_RELATION}
\epsilon_n = \epsilon-8K - K \sum_{j=1}^{n-1} 2^{3-j}.
\end{align}
The proposition then follows by choosing $K$ sufficiently small that $\epsilon_n \geq \epsilon/2$ for all $n$.
\end{proof}

The following simple lemma complements Proposition~\ref{PROP:CONV_TO_ID}.

\begin{lemma}
\label{tangentidentitycloseidentity}
Let $F_1, F_2 : B_{\epsilon} (0) \rightarrow \C^n$ be as above with $F_1 (0) = F_2(0) =0$. Assume that
their derivatives at the origin satisfy
\begin{align}\label{EQN:ESTIMATE_ON_DERIVATIVES}
\max \left\{ \Vert D_0F_1 - {\rm Id} \Vert , \Vert D_0F_2 - {\rm Id} \Vert \right\} < \tau,
\end{align}
where $\tau >0$ is some uniform (universal) to be determined later (the norm used here is the standard norm on linear
operators on $\C^n$).
Then, there is some $0 < \delta < \epsilon$ such that $F_1$ and $F_2$ satisfy the hypotheses
of Proposition~\ref{PROP:CONV_TO_ID} on~$B_\delta(0)$.
\end{lemma}

\begin{proof}
It follows from the proof of Proposition~\ref{PROP:CONV_TO_ID} that the
relation between $\epsilon$ and $K = K(\epsilon)$ given by
(\ref{EQN:EPSILON_K_LINEAR_RELATION}) is linear. Thus for $\epsilon =1$, $K =
K(1)$ becomes a uniform (universal) constant. We then choose $\tau = K/2$ so
that $\tau$ is also universal.

In view of the above remark, we proceed as follows. Fix $\delta < \epsilon$ and consider the homothety $\Lambda_{\delta} : \C^n
\rightarrow \C^n$ sending $(x_1, \ldots ,x_n)$ to $(\delta x_1, \ldots ,\delta x_n)$. Set $F_{j,\delta} = \Lambda_{\delta}^{-1} \circ
F_j \circ \Lambda_{\delta}$ and note that $F_1, F_2$ satisfy the conditions of Proposition~\ref{PROP:CONV_TO_ID} on the
ball of radius $\delta$ if and only if we have
\begin{equation}
\max \left\{ \sup _{z \in B_{1}}\left\| F_{1,\delta} (z) -z \right\|,  \sup _{z \in B_{1}}\left\|
F_{2,\delta} (z) -z \right\| \right\} < K = K(1) \, . \label{estimate_for_applications}
\end{equation}
We will now check that Estimate~(\ref{estimate_for_applications}) is always satisfied provided
that $\delta$ is small enough. Clearly, it suffices to consider the case of $F_{1,\delta}$. Owing to Taylor's formula, we have
\begin{eqnarray*}
\sup_{z \in B_1} \Vert F_{1,\delta} (z) -z \Vert  \leq   \Vert D_0 F_1 (z) - z \Vert + O \, (\delta) 
  \leq  K/2 + O \, (\delta) < K
\end{eqnarray*}
provided that $\delta$ is small enough.   The result is proved.
\end{proof}

Let $\mathcal{ND} \subset \C^4$ the set of those parameters $(A,B,C,D)$ giving
rise to a group $\AUTOGROUP_{A,B,C,D}$ acting locally non-discretely on some non-empty open subset of
the surface $S_{A,B,C,D}$. The remainder of this section is devoted to exhibiting 
several explicit examples of parameters in the interior of $\mathcal{ND}$. 

\begin{proposition}\label{PROP:ND_INTERIOR_PTS}
Suppose that $(A_0,B_0,C_0)$ are parameters for which there are two non-commuting elements $F_1,F_2
\in \AUTOGROUP$ sharing a common fixed point $p \in \mathbb{C}^3$. Assume also that the derivatives of $F_1, F_2$
at $p$ satisfy inequality~(\ref{EQN:ESTIMATE_ON_DERIVATIVES}). Then the following holds:
\begin{enumerate} 
\item There exists $r > 0$ such that for all parameters $(A,B,C)$ sufficiently close to $(A_0,B_0,C_0)$ the group
$\AUTOGROUP$ acting on $\mathbb{C}^3$ is locally non-discrete on the open ball $B_r(p) \subset \mathbb{C}^3$.

\vspace{0.05in}

\item Let $D_0$ be chosen so that $p \in S_{A_0,B_0,C_0,D_0}$.   Then for all parameters $(A,B,C,D)$ sufficiently close
to $(A_0,B_0,C_0,D_0)$ the group $\AUTOGROUP$ acting on $S_{A,B,C,D}$ is locally non-discrete on the open set $S_{A,B,C,D} \cap B_r(p)$.
In other words, $(A_0,B_0,C_0,D_0)$ is an interior point of $\mathcal{ND}$.
\end{enumerate}
\end{proposition}

Recall from the introduction that El-Huiti's theorem implies that $\AUTOGROUP$ is
the free group on two generators (one can choose any two of the three mappings
$g_x, g_y,$ and $g_z$ as generators).  Therefore,  two elements $\gamma_1,
\gamma_2 \in \AUTOGROUP$ commute if and only if there exists $a \in \AUTOGROUP$ so that
$\gamma_1$ and $\gamma_2$ are both powers of $a$.  This is an immediate
consequence of the Nielsen-Schreier Theorem which states that any subgroup of a
free group is free.

The proof of Proposition \ref{PROP:ND_INTERIOR_PTS} will require a simple algebraic lemma:

\begin{lemma}\label{LEM:ALGEBRA}
Let $F_n$ denote the free group on $n \geq 2$ symbols and suppose 
$a, b \in F_n$ do not commute.  Then the subgroup 
\begin{align*}
H = \langle [a,b], [a^{-1},b^{-1}] \rangle \leq F_n
\end{align*}
is again free on two or more symbols.
\end{lemma}

\begin{proof}
The Nielsen-Schreier Theorem implies that $H$ is a free group.  To see that it has rank two, it suffices
check that $[a,b]$ and $[a^{-1},b^{-1}]$ do not commute.  This follows because
\begin{align*}
[a,b] [a^{-1},b^{-1}] [a,b]^{-1} [a^{-1},b^{-1}]^{-1}
\end{align*}
is a non-trivial reduced word in $a, b, a^{-1},$ and $b^{-1}$.  Since $a$ and $b$ do not
commute this word does not reduce to the identity in $F_n$.
\end{proof}

\begin{proof}[Proof of Proposition \ref{PROP:ND_INTERIOR_PTS}]
Lemma \ref{tangentidentitycloseidentity} implies the existence of some
$\delta > 0$ such that for parameters $(A_0,B_0,C_0)$ the mappings  $F_1$ and
$F_2$ satisfy the hypotheses of Proposition \ref{PROP:CONV_TO_ID} on
$B_\delta(p)$.  Moreover the conditions of Proposition \ref{PROP:CONV_TO_ID}
are open in the $C^0$ topology (on $B_\delta(p)$).  This implies that for all
parameters $(A,B,C)$ sufficiently close to $(A_0,B_0,C_0)$ the mappings $F_1$
and $F_2$ continue to satisfy the hypotheses of Proposition
\ref{PROP:CONV_TO_ID} on $B_\delta(p)$. Thus, for these parameters $(A,B,C)$, the
elements in the iterated commutators $S(n)$ converge uniformly to the identity
on the ball $B_{\delta/2}(0) \subset \mathbb{C}^3$.

Since $F_1$ and $F_2$ do not commute, Lemma~\ref{LEM:ALGEBRA} can inductively
be applied to ensure that each set $S(n)$ contains at least two
non-commuting elements.  Therefore, for every $n \geq 0$, there are elements in $S(n)$ that are
different from the identity which, in turn, proves
that the pseudogroup generated by $F_1$ and
$F_2$ is locally non-discrete on $B_r(p)$. 

Statement (2) then follows immediately because elements of
$\AUTOGROUP$ different from the identity cannot coincide with the identity when restricted to any
surface $S_{A,B,C,D}$.
In other words, for every choice of parameters $(A,B,C,D)$, any non-trivial reduced word in $g_x$ and $g_y$
(or in any two of the three mappings $g_x, g_y,$ and $g_z$) induces a mapping of $S_{A,B,C,D}$ that
is different from the identity.   This is a consequence of El-Huiti's theorem in \cite{huti}; see Section \ref{SUBSEC:ALGEBRIAC_PROPERTIES}.
\end{proof}

\vspace{0.1cm}

\noindent {\bf Example 1 - Markoff Parameters}. It is an easy observation that for the parameters $A = B = C = 0$
the origin $(0,0,0)  \in \mathbb{C}^3$ is a common fixed
point for $g_x, g_y,$ and $g_z$ and that their derivatives at the origin satisfy
$D_0 g_x = {\rm diag}(1,-1,-1), \, D_0 g_y = {\rm diag}(-1,1,-1), \, \mbox{and}
\,  D_0 g_z = {\rm diag}(-1,-1,1)$.

If we let
$h_x = g_x^2, h_y=g_y^2,$ and $h_z=g_z^2$, 
then each of these maps is tangent to the identity at the origin. 
Notice that when $D = 0$ the surface $S_{0,0,0,D}$ passes through $(0,0,0)$.  Therefore,
applying Proposition \ref{PROP:ND_INTERIOR_PTS}
to the non-commuting pair of elements $h_x$ and $h_y$ yields: 

\begin{lemma}
\label{Firstexamples}
There is a neighborhood $W$ of the origin in $\mathbb{C}^4$ such that for all $(A,B,C,D) \in W$, the action of
$\AUTOGROUP$ is locally non-discrete on an open subset of $S_{A,B,C,D}$ obtained by intersecting $S_{A,B,C,D}$
with a small ball centered at the origin in $\mathbb{C}^3$. In other words, $(0,0,0,0)$ is an interior
point of~${\mathcal ND}$.
\end{lemma}

\noindent {\bf Example 2 - Dubrovin-Mazzocco parameters}.
Recall from the introduction the 
real $1$-parameter family studied by Dubrovin and  Mazzocco \cite{dubrovinmazzocco}:
In our notations, the Dubrovin-Mazzocco parameters correspond to
\begin{align*}
A(a) = B(a) = C(a) = 2a+4, \quad \mbox{and} \quad D(a) = -(a^2 + 8a +8)
\end{align*}
for $a \in (-2,2)$.
To simplify notations, let us denote the surface $S_{A,B,C,D}$ for these parameters by $S_a$ and
the group $\AUTOGROUP_{A,B,C,D}$ by $\AUTOGROUP_a$.

\begin{lemma}
\label{lemma_for_DubrovinMazzocco}
For the  Dubrovin-Mazzocco family introduced above the surface $S_a$ contains
exactly three singular points $p_1, p_2, p_3$ given by
\begin{align}\label{EQN:SING_PTS}
p_1 = (x_1,y_1,z_1) =  (a,2,2), \quad p_2 = (x_2,y_2,z_2) = (2,a,2), \quad \mbox{and} \quad 
p_3= (x_3,y_3,z_3) = (2,2,a).
\end{align}
\end{lemma}

\begin{proof}
The singular points correspond to the common fixed points of $s_x, s_y,$ and $s_z$ which, in turn, correspond
to solutions in $(x,y,z)$ of the equations
\begin{align}\label{EQN:FIXED_PTS}
-x-yz+2a+4 = x, \qquad -y-xz+2a+4 = y, \qquad \mbox{and} \qquad -z-xy+2a+4 = z \; .
\end{align}
It is immediate to check that the three above singular points satisfy these equations.

Meanwhile, we must check that these three points lie on the surface $S_a$. By symmetry, it suffices to check
for $p_1 = (x_1,y_1,z_1)$ and, in this case, we have
\begin{eqnarray}
x_1^2+y_1^2+z_1^2 + x_1 y_1 z_1 - Ax_1 - By_1 - \! C z_1  & \! = \! & a^2+4+4+4a-(2a+4)[a-2-2] \nonumber \\
& = & -(a^2+8a+8) = D \,  \label{EQN:SURFACE}
\end{eqnarray}
so that $p_1 \in S_a$. Note that there are other common fixed points for $s_x, s_y,$ and $s_z$ but they are located
away from $S_a$.
\end{proof}

\begin{proposition}
\label{prop_stillDubrovinMazzocco}
For every $a \in (-2,2)$ the group $\AUTOGROUP_a$ acting on $S_a$ has locally non-discrete dynamics
in some neighborhood of each of the singular points $p_1, p_2$ and $p_3$.

Moreover, the Dubrovin-Mazzocco parameters at this fixed value of $a$ yield an interior point of ${\mathcal ND}$. More precisely,
given parameters $(A,B,C,D)$ sufficiently close to the Dubrovin-Mazzocco parameters in question,
the group $\AUTOGROUP$ is locally non-discrete
on the intersection of $S_{A,B,C,D}$ with $B_r(p_1) \cup B_r(p_2) \cup B_r(p_3)$ for some $r > 0$.
\end{proposition}

Note that the parameter $a=-2$, which we do not consider a Dubrovin Mazzocco parameter, corresponds to the Picard Parameters $(A,B,C,D) = (0,0,0,4)$ for
which $\AUTOGROUP$ is locally discrete on all of the corresponding surface, as was shown in Section \ref{SEC:PICARD}.

\begin{proof}
Because $A(a) = B(a) = C(a)$ the action of $\AUTOGROUP_{A,B,C,D}$ on $S_{A,B,C,D}$ is 
self-conjugate under any permutation of the variables $x$, $y$, and $z$.
Therefore it suffices to consider 
$p_1$, with the corresponding results for $p_2$ and $p_3$ following via these conjugacies.   Note that $p_1$ is fixed by every element in $\AUTOGROUP$ (in fact, by every element of the group generated by
$s_x$, $s_y$, and $s_z$).

Let us first show that a suitable iterate of $g_x = s_z \circ s_y$ is close to the identity on some suitable neighborhood
of $p_1$.  
A direct calculation yields that
the eigenvalues of $D_{p_1} g_x$ are
\begin{align*}
\lambda_1 = \frac{a^{2}}{2}-1+\frac{\sqrt{a^{4}-4 a^{2}}}{2}, \quad \lambda_2 =  \frac{a^{2}}{2}-1-\frac{\sqrt{a^{4}-4 a^{2}}}{2}, \quad \mbox{and} \quad \lambda_3 = 1 \, .
\end{align*}
For $a \in (-2,0) \cup (0,2)$ we see that $\lambda_1$ and $\lambda_2$ form a complex conjugate pair of eigenvalues, each
of which has modulus one.  In this case, $D_{p_1} g_x$ is conjugate to a rotation in suitable coordinates. 
In particular, for any $\tau > 0$ there is $k$ sufficiently large so that
\begin{align*}
\|\big(D_{p_1}g_x \big)^k - {\rm Id} \| < \tau.
\end{align*}
To find a second mapping satisfying the hypotheses of Lemma~\ref{tangentidentitycloseidentity},
it suffices to consider the map $h$ obtained by
conjugating $g_x^k$ by, say, $g_y$. Since $p_1$ is fixed by $g_y$ as well, it follows that $D_{p_1} h$ is conjugate
to $D_{p_1} g_x^k$ and that the conjugating matrix does not depend on~$k$ (it is simply the matrix given by
$D g_y (p_1)$). Therefore, up to making $k$ larger if needed, we
can assume that both $g_x^k$ and $h$ satisfy the hypotheses of Lemma~\ref{tangentidentitycloseidentity}.
Since $g_x^k$ and $h$ do not commute in $\AUTOGROUP$, the result then follows from Proposition~\ref{PROP:ND_INTERIOR_PTS}.

In the special case that $a = 0$ another direct calculation yields that 
$\left(D_{p_1} g_x\right)^2 = {\rm Id}$,
in which case $g_x^2$ is tangent to the identity at $p_1$.  Letting $h = g_y
g_x^2 g_y^{-1}$ we obtain a second mapping tangent to the identity at $p_1$.
Since, as previously seen, these two maps do not commute, the desired result in this
special case follows again from
Proposition~\ref{PROP:ND_INTERIOR_PTS}.
\end{proof}


\begin{remark}\label{genuinelynonlinear}
The combination of the results obtained in this section
and in the previous one substantiate our claim that, in general, the action of
$\AUTOGROUP_{A,B,C,D}$ is genuinely non-linear in the sense that
it cannot preserve a rigid geometric structure.  More precisely, whenever the
locally non-discrete locus $\nondiscrete_{A,B,C,D}$ is non-empty but not dense (as in the above
examples), the group $\AUTOGROUP_{A,B,C,D}$ cannot preserve any rigid geometric structure
on $S_{A,B,C,D}$.
		
Roughly speaking, a structure on a manifold $M$
is called rigid if there is a uniform~$k$ such that the $k$-jet of a local
isometry at a given point determines its $C^{\infty}$-jet; see \cite{gromov}
for detail. Examples of these structures include (pseudo-) Riemannian metrics
and affine connections. Also, the action of a finite dimensional Lie group $G$
on a manifold can be viewed as an action by isometries of some rigid structure
provided that there is some $k$ such that the induced action on the $k$-jet
bundle $J^k (M)$ over $M$ is free and proper.
		
It follows essentially from the Frobenius
theorem that, at least in real analytic category, every germ of isometry can be
extended along paths in $M$. In particular, the standard results on
continuous/differentiable dependence on the initial conditions for solutions of
differential equations can be applied to these extensions. Consider a sequence
of local isometries $\{f_n \}$ of some $k$-rigid structure fixing a point $p$
and let $c: [0,1] \rightarrow M$ be a fixed path with $c(0) =p$. Denoting by
$\{ f_{n,c} \}$ the extension of $f$ along~$c$, the preceding implies that if
the $k$-jets of $f_n$ at $p$ converge to the identity then so do the $k$-jets
of $f_{n,c}$ at the point $c(1)$. Since holomorphic convergence implies
$C^{\infty}$-convergence, we conclude that if
$\AUTOGROUP_{A,B,C,D}$ preserves some rigid structure on $M$ then the locally
non-discrete locus $\nondiscrete_{A,B,C,D}$ is either empty or coincides with the whole surface
$S_{A,B,C,D}$.
		
A specific instance of this is the case of the Picard parameters where the
action of $\AUTOGROUP_{0,0,0,4}$ is known to preserve an affine structure, which is 
an example of a rigid structure.   We have seen in Theorem~D(ii) that for these parameters the locally
non-discrete locus is empty.
\end{remark}

\section{Dynamics near infinity}
\label{dynamics_near_infinty}

In this section we will collect a few important, if slightly technical, results on the dynamics of the group
$\AUTOGROUP$ as well as the dynamics of (individual) hyperbolic maps near $\Delta_{\infty}$. The corresponding
results, especially Proposition~\ref{fixedpoints_remainincompactparts} and Lemma~\ref{LEM:ESCAPE_TO_INFINITY},
will play major roles in the proofs of Theorems~H and~K to be supplied in the forthcoming sections.

Let us begin with a general review of the behavior of a given map in $\AUTOGROUP = \AUTOGROUP_{A,B,C,D}$
near infinity. Up to compactifying $\C^3$ into the projective space, we begin by recalling that the
closure of any surface $S_{A,B,C,D}$ intersects the hyperplane at infinity $\Pi_\infty \subset \mathbb{CP}^3$
in a triangle $\Delta_\infty$. In homogeneous coordinates, this triangle is given by
\begin{align*}
\Delta_\infty = \{(X:Y:Z:W) \in \mathbb{CP}^3 \quad : \quad W = 0 \quad \mbox{and} \quad XYZ = 0\}.
\end{align*}
The vertices of $\Delta_\infty$ are denoted by
\begin{align*}
\mathcal{V}_\infty =\big\{v_1 = (1:0:0:0), \qquad v_2 = (0:1:0:0), \qquad \mbox{and} \qquad v_3 = (0:0:1:0) \big\}.
\end{align*}

As previously seen in Section \ref{projective_compactification}, for every choice of the parameters $(A,B,C,D)$, the surface $S_{A,B,C,D}$ is smooth on
a neighborhood of $\Delta_{\infty}$. Furthermore, $S_{A,B,C,D}$ is tangent to the plane at infinity
exactly at the vertices in $\mathcal{V}_{\infty}$ and, hence, it has transverse intersection with the
plane at infinity elsewhere in $\Delta_{\infty}$. In the sequel, to abridge notation, the
parameters $A,B,C,D$ will be dropped whenever there is no possibility of misunderstanding. Thus,
we will sometimes write $\AUTOGROUP$ for $\AUTOGROUP_{A,B,C,D}$,
$S$ for $S_{A,B,C,D}$, and $\overline{S}$ for the closure of $S$.

Because $\overline{S}$ is tangent to the plane at infinity at the vertices of $\Delta_{\infty}$,
near $v_1$ we can use the affine coordinates
$(Y/X,Z/X,W/X)$ on $\mathbb{CP}^3$ and express the surface $\overline{S}$ so that $W/X$ becomes a
holomorphic function of $(Y/X,Z/X)$. In other words, $S$ is locally given as the graph of a holomorphic
function in the variables $(Y/X,Z/X)$ on a neighborhood of $v_1$. This local representation
will be referred to as the ``standard
coordinates'' on $\overline{S}$ in a neighborhood of $v_1$. The analogous
construction leads to ``standard coordinates'' on $\overline{S}$ near $v_2$ and
near $v_3$, respectively.

Inside the plane at infinity, consider a neighborhood $N_1$ of $v_1$ where $\overline{S}$ is
the graph of some holomorphic function $h_1$ as indicated above. Neighborhoods $N_2$ and
$N_3$ respectively of $v_2$ and $v_3$ are similarly defined along with the corresponding holomorphic
functions $h_2$ and $h_3$.  The following lemma is rather immediate:

\begin{lemma}\label{surfaces_movecontinuously}
	Consider parameters $A_0, B_0, C_0$ and $D_0$ and apply the above construction to the surface
	$\overline{S}_{A_0, B_0, C_0, D_0}$. Up to reducing the neighborhoods $N_i$, $i=1,2,3$, there exists
	a neighborhood $\mathcal{N} \subset \C^4$ of $(A_0, B_0, C_0, D_0)$ such that all of the following
	holds:
\begin{itemize}
	\item[(1)] For every $(A,B,C,D) \in \mathcal{N}$, the surface $\overline{S}_{A, B, C, D}$ is (locally)
	the graph of a function $h_i$ defined on $N_i$. Furthermore, for $i=1,2,3$, the functions $h_i$
	vary continuously with the parameters.
	
	\item[(2)] For every $(A,B,C,D) \in \mathcal{N}$, the intersection of the surface
	$\overline{S}_{A, B, C, D}$ with the plane at infinity over the set $\Delta_{\infty} \setminus
	(N_1 \cup N_2 \cup N_3)$ is uniformly transverse. Furthermore, the slopes vary continuously with the
	point and with the parameters.
\end{itemize}
\mbox{}\qed
\end{lemma}

We will also need some facts about birational extensions of elements from $\AUTOGROUP^\pm_{A,B,C,D}$ and $\AUTOGROUP_{A,B,C,D}$ to $\overline{S}_{A,B,C,D}$.

\begin{remark}
Doing this birational extension is more subtle than one might initially guess.  
Any $f \in \AUTOGROUP^\pm_{A,B,C,D}$ is obtained as the restriction of a polynomial automorphism $F: \mathbb{C}^3 \rightarrow
\mathbb{C}^3$ to $S_{A,B,C,D}$.   However, the birational extension $F: \mathbb{CP}^3 \dashrightarrow \mathbb{CP}^3$ typically
has entire sides of the triangle at infinity $\Delta_\infty$ contained in its indeterminacy set.   Meanwhile the birational
extension $f: \overline{S}_{A,B,C,D} \dashrightarrow \overline{S}_{A,B,C,D}$ must have a finite set of indeterminate
points because the indeterminacy locus of any rational mapping is always of codimension at least two.  Indeed, restricting the mapping $F$ to the
surface $\overline{S}_{A,B,C,D}$ resolves this ``excess indeterminacy'' for the birational extension $F: \mathbb{CP}^3 \dashrightarrow \mathbb{CP}^3$.
For this reason, working with the birational extension $f: \overline{S}_{A,B,C,D} \dashrightarrow \overline{S}_{A,B,C,D}$
is best done by working within local coordinates on $\overline{S}_{A,B,C,D}$ in the domain and codomain of $f$, as opposed to
working with $3$-dimensional coordinates.
\end{remark}

Recalling that
$\AUTOGROUP^\pm$ is generated by the involutions $s_x$, $s_y$, and $s_z$.  An element $\gamma$ in
$\AUTOGROUP^\pm$ is said to be {\it cyclically reduced} if its reduced spelling as a word in the
``letters'' $s_x$, $s_y$, and $s_z$ has the property that every cyclic permutation of the
word has no further possible cancellations.   This can be expressed equivalently by saying that
the word's reduced spelling in the
``letters'' $s_x$, $s_y$, and $s_z$ is not conjugate in $\AUTOGROUP^\pm$ to another element
in $\AUTOGROUP^\pm$ having strictly smaller length when spelled in the same ``letters''.

We also recall that a meromorphic self-map from a complex surface to itself is said to be
{\it algebraically stable}\, if it does not contract a hypersurface to its indeterminate set,
cf. \cite{Sibony,DF}.
In the present case where elements of $\AUTOGROUP$ act on the surface $S$, to be algebraically stable
amounts to saying that $\gamma$ does not contradict any of the sides of $\Delta_{\infty}$ to an indeterminacy
point which, in turn, necessarily lies in $\Delta_{\infty}$ as well.

With this terminology, the following definition/proposition and remark summarize 
\cite[Prop. 3.2]{cantat-1} and \cite[Prop. 2.3]{cantat-2}.

\begin{definitionproposition}\label{PROP:DESCRIPTION_HYPERBOLIC_ELEMENTS}
	For any parameters $A,B,C,D$ and any $\gamma \in \AUTOGROUP$ we have
	\begin{enumerate}
		\item[(i)] $\gamma$ is said to be hyperbolic if and only if it conjugate to a cyclically
		reduced word in $s_x, s_y,$ and $s_z$ that contains all three mappings.
		
		\item[(ii)] A hyperbolic map $\gamma \in \AUTOGROUP$ possesses a single indeterminate point which
		coincides with a vertex of $\Delta_{\infty}$ and will be denoted by ${\rm Ind}(\gamma)$.
		
		\item[(iii)] A hyperbolic map $\gamma \in \AUTOGROUP$ contracts all of  $\Delta_{\infty} \setminus
		\{ {\rm Ind}(\gamma) \}$ to a vertex of $\Delta_{\infty}$ denoted by ${\rm Attr}(\gamma)$. The
		vertices ${\rm Ind}(\gamma)$ and ${\rm Attr}(\gamma)$ may or may not coincide. In particular, the map
		$\gamma$ is algebraically stable if and only if ${\rm Ind}(\gamma) \neq {\rm Attr}(\gamma)$.
		
		\item[(iv)] Alternatively, $\gamma: \overline{S} \dashrightarrow \overline{S}$ is algebraically
		stable if and only if it is a cyclically reduced composition of $s_x, s_y,$ and $s_z$ of length
		at least two.

		\item[(v)]  If $\gamma$ is algebraically stable and hyperbolic then $\gamma$ is holomorphic around
		${\rm Attr}(\gamma)$ and, in fact, ${\rm Attr}(\gamma)$ is a superattracting fixed point of $\gamma$.
		Moreover, the roles of
		${\rm Ind}(\gamma)$ and ${\rm Attr}(\gamma)$ are interchanged if we pass from $\gamma$ to
		$\gamma^{-1}$, i.e., ${\rm Attr}(\gamma)={\rm Ind}(\gamma^{-1})$ and
		${\rm Ind}(\gamma) = {\rm Attr}(\gamma^{-1})$.
		
		\item[(vi)] An element $\gamma$ is said to be parabolic if it is conjugate in $\gamma$ to one
		of the maps $g_x$, $g_y$, or $g_z$. Every element of $\AUTOGROUP$ different from the identity is either
		hyperbolic or parabolic.
	\end{enumerate}
\end{definitionproposition}

\begin{remark}\label{REM:ELLIPTICAL_ELEMENTS}
While $\AUTOGROUP$ only contains hyperbolic and parabolic elements, the bigger
group $\AUTOGROUP^\pm$ also contains elliptic elements.
An element of $\AUTOGROUP^\pm$ is elliptic if and only if it is conjugate in $\AUTOGROUP^\pm$ to one of the involutions $s_x, s_y,$ or $s_z$, if and only if it is periodic.  (The characterizations
of parabolic and hyperbolic elements of $\AUTOGROUP$ carry over to $\AUTOGROUP^\pm$, with conjugation
in $\AUTOGROUP$ replaced by conjugation in $\AUTOGROUP^\pm$.)
\end{remark}

\begin{remark} 

There is an equivalent but more intrinsic way of classifying the elements of $\AUTOGROUP^\pm_{A,B,C,D}$ 
as being  hyperbolic, parabolic, or elliptic.   It will not be needed in this paper, but we describe it here
for the benefit of the reader.
Let us introduce one additional matrix group:
\begin{align*}
\MATRIXGP^\pm_2 := \{M \in {\rm PGL}(2,\mathbb{Z}) \, : \, M \equiv {\rm Id} \,  {\rm mod} 2\}.
\end{align*}
which has the classical congruence group $\MATRIXGP_2$ as an index two subgroup.  It is generated by (the projective equivalence
classes of) the following
matrices
\begin{align*}
N_x:= \left[\begin{array}{cc} -1 & -2 \\ 0 & 1 \end{array}\right], \qquad N_y:= \left[\begin{array}{cc} 1 & 0 \\ -2 & -1\end{array}\right], \qquad \mbox{and} \qquad
N_z:=\left[\begin{array}{cc} 1 & 0 \\ 0 & -1\end{array}\right].
\end{align*}
For any parameters $(A,B,C,D)$ there is a group isomorphism from $\MATRIXGP^\pm_2$
to $\AUTOGROUP^\pm_{A,B,C,D}$ induced by sending $N_x, N_y,$ and $N_z$
to the involutions $s_x, s_y,$ and $s_z$, respectively.   It is an extension of
the group isomorphism $\MATRIXGP_2 \rightarrow \AUTOGROUP_{A,B,C,D}$ discussed
in the case of the Picard Parameters (Section \ref{SEC:PICARD}) and it arises
naturally from the context of dynamics on character varieties;
see \cite[Section 2.3]{cantat-2}.

Given $[N] \in \MATRIXGP^\pm_2$ let us denote $f_{[N]} \in \AUTOGROUP^\pm_{A,B,C,D}$ the associated automorphism.
One then has that $f_{[N]}$ is
hyperbolic, parabolic, or elliptic if and only if the associated matrix $N$
induces a hyperbolic, parabolic, or elliptic transformation of the hyperbolic
plane $\mathbb{H}$.

Moreover, the degrees of the $k$-th iterate $f_{[N]}^k$ grows at the same rate
as the entries of $N^k$.  Therefore, one has yet another equivalent
classification that $f_{[N]}$ is hyperbolic, parabolic, or elliptic if and only
if the degrees of $f^k_{[N]}$ grow exponentially, grow polynomially, or remain
bounded.   Therefore, the 
classification of elements of $\AUTOGROUP^\pm$ into the classes (hyperbolic, parabolic, elliptic) is
compatible with use of this terminology for elements of the Cremona group (see \cite[Theorem 4.6]{CANTAT_CREMONA}) and also with the use of this terminology
for automorphisms of compact surfaces (see \cite[Section 2.4.4]{CANTAT_SURVEY}).
\end{remark}

Note that the properties described in
Definition/Proposition~\ref{PROP:DESCRIPTION_HYPERBOLIC_ELEMENTS} are independent of the parameters
$A,B,C,D$. They depend only on the spelling of $\gamma$ as an element in the abstract group
generated by three letters $a$, $b$, and $c$ with the relation $abc ={\rm id}$ (up to the substitutions
$a=g_x$, $b=g_y$, and $c=g_z$).
In particular, the points ${\rm Ind}(\gamma)$ and ${\rm Ind}(\gamma^{-1})
={\rm Attr}(\gamma)$ are independent of the parameters and this plays an important role
in the statement of the following proposition.

\begin{proposition}\label{fixedpoints_remainincompactparts}
	Consider an hyperbolic element $\gamma \in \AUTOGROUP$ and, for a choice of parameters
	$A_0$, $B_0$, $C_0$, and $D_0$, let $\overline{f}_{A_0,B_0,C_0,D_0}$ be the resulting birational map
	of the (compact) surface $\overline{S}_{A_0,B_0,C_0,D_0}$. Then there exist a neighborhood $U_{\Delta}$
	of $\Delta_{\infty} \subset \mathbb{CP}^3$ and a neighborhood $\mathcal{U}_0 \subset \C^4$
	of $(A_0,B_0,C_0,D_0) \in \C^4$ such that the following holds:
	\begin{itemize}
		\item For every $(A,B,C,D) \in \mathcal{U}_0$, the intersection $U_{\Delta} \cap
		\overline{S}_{A,B,C,D}$ is a neighborhood of $\Delta_{\infty}$ in $\overline{S}_{A,B,C,D}$.
		
		\item For every $(A,B,C,D) \in \mathcal{U}_0$, the map $\overline{f}_{A,B,C,D}$ has no 
		periodic point in $(U_{\Delta} \cap \overline{S}_{A,B,C,D})\setminus
		\Delta_{\infty}$.
	\end{itemize}
\end{proposition}

\begin{proof}
It suffices to prove the proposition for an algebraically stable hyperbolic $\gamma$ since, every
hyperbolic element in $\AUTOGROUP$ is conjugate in $\AUTOGROUP$ to an algebraically stable one and all
elements in $\AUTOGROUP$ preserve infinity.

Let then $\overline{f}_{A_0,B_0,C_0,D_0}$ be an algebraically stable hyperbolic map as indicated
above. The statement amounts to showing the existence of a neighborhood $V_0 \subset \overline{S}_{A_0,B_0,C_0,D_0}$
of $\Delta_{\infty} \subset \overline{S}_{A_0,B_0,C_0,D_0}$ satisfying the two conditions below:
	\begin{enumerate}
		\item $\overline{f}_{A_0,B_0,C_0,D_0}$ has no periodic point in
		$V_0 \setminus \Delta_{\infty}$.
		
		\item The  neighborhood $V_0$ can be chosen to vary continuously with the parameters $A$, $B$,
		$C$, and~$D$.
	\end{enumerate}
To prove the first assertion, we consider the points ${\rm Attr}(\overline{f}_{A_0,B_0,C_0,D_0})$
and ${\rm Ind}(\overline{f}_{A_0,B_0,C_0,D_0})$ in $\Delta_{\infty}$. As mentioned, these points
are distinct and do not depend on the choice of the parameters. To abridge notation, we then set
$P = {\rm Attr}(\overline{f}_{A_0,B_0,C_0,D_0})$ and $Q= {\rm Ind}(\overline{f}_{A_0,B_0,C_0,D_0})$.
We also recall that $P$ is a super-attracting fixed point of $\overline{f}_{A_0,B_0,C_0,D_0}$
and, in fact, a super-attracting fixed point of $\overline{f}_{A,B,C,D}$ for every choice of the parameters
$A$, $B$, $C$, and $D$. Similarly, $Q$ is a super-attracting fixed point
of $\overline{f}^{-1}_{A,B,C,D}$ for any $(A,B,C,D) \in \C^4$.
Now let $U \subset \overline{S}_{A_0,B_0,C_0,D_0}$
be a small neighborhood of $Q$ which is contained in the basin
of attraction of $Q$ for $\overline{f}^{-1}_{A_0,B_0,C_0,D_0}$.
Since $\overline{f}_{A_0,B_0,C_0,D_0}$ sends $\Delta_{\infty} \setminus
\{ Q \}$ to $P$, there exists
a neighborhood $V$ of $\Delta_{\infty} \setminus U$ which is 
in the basin of attraction of $P$, with respect to $\overline{f}_{A_0,B_0,C_0,D_0}$.
Therefore
$V \cup U$ contains no periodic point of $\overline{f}_{A_0,B_0,C_0,D_0}$: in fact, every point in
$(V \cup U) \setminus \{P,Q\}$ converges to $P$ under iteration by
$\overline{f}_{A_0,B_0,C_0,D_0}$ or to $Q$
under iteration by $\overline{f}^{-1}_{A_0,B_0,C_0,D_0}$.

It remains to prove that the above neighborhood can be assumed to vary continuously with the
parameters. For this, we first consider the (local) description of the surfaces
$\overline{S}_{A,B,C,D}$ provided by Lemma~\ref{surfaces_movecontinuously}. The lemma in question
shows that the surface $\overline{S}_{A,B,C,D}$ converges towards $\overline{S}_{A_0,B_0,C_0,D_0}$
as $(A,B,C,D) \rightarrow (A_0,B_0,C_0,D_0)$ on a neighborhood of $\Delta_{\infty}
\subset \mathbb{CP}^3$. In fact, on a neighborhood of the points $P$ and $Q$, this
follows from the local structure of these surfaces as graphs of holomorphic functions. Conversely,
away from these neighborhoods of $P$ and $Q$, the statement follows from the (uniform) transverse intersection
of these surfaces and the plane at infinity of $\mathbb{CP}^3$.
In particular the maps
$\overline{f}_{A,B,C,D}$ converge uniformly to $\overline{f}_{A_0,B_0,C_0,D_0}$ on a
neighborhood of $\Delta_{\infty}$.
Thus, the statement is reduced to showing the existence of
$\epsilon >0$ such that for all parameters $(A,B,C,D)$ sufficiently close to $(A_0,B_0,C_0,D_0)$,
the ball of radius $\epsilon$ around $P$ still is contained in the basin of attraction of $P$ with
respect to $\overline{f}_{A,B,C,D}$. Here, the same argument applies to $Q$ and
$\overline{f}^{-1}_{A,B,C,D}$ and the notion of ``ball'' is relative to some auxiliary metric,
for example, the Euclidean metric in the ``standard coordinates'' of Lemma~\ref{surfaces_movecontinuously}.

The last claim can be proved as follows. The germ of $\overline{f}_{A,B,C,D}$ is a rigid, reducible,
super-attracting germ in the sense of \cite{Favre} (cf. \cite{cantat-1}). It actually falls in the
``class 6'' of the classification provided in \cite{Favre} and hence it is conjugate to a monomial
map, owing to results of Dloussky and Favre. Since the monomial map does not depend on the parameters,
the claim follows from checking directly in the argument of \cite{Favre} that the conjugating map
depends continuously on the initial map.
\end{proof}


As a consequence of Proposition~\ref{fixedpoints_remainincompactparts}, we also obtain
the following result (cf. Lemma~16 from~\cite{IU_ERGODIC}).

\begin{corollary}
\label{isolatedfixedpoints_hyperbolicmaps}
For every choice of the parameters $(A,B,C,D) \in \C^4$ and every hyperbolic element
in $\AUTOGROUP$, all the fixed points of the resulting (hyperbolic) map $f_{A,B,C,D}$ are isolated.
\end{corollary}

\begin{proof}
As a composition of the algebraic maps $g_x$, $g_y$, and $g_z$, $f_{A,B,C,D}$ is itself an algebraic
map from $\C^3$ to $\C^3$. Thus the set of its fixed points is an algebraic set of $\C^3$ so that
its intersection with $S_{A,B,C,D}$ is an algebraic subset of $S_{A,B,C,D}$. 
Since this set is compact (Lemma \ref{fixedpoints_remainincompactparts}) is must be finite.
\end{proof}

The remainder of this section will be devoted to the proof of Lemma~\ref{LEM:ESCAPE_TO_INFINITY}
below which is crucial for understanding the structure of unbounded Fatou components as will
be seen later on.

We resume the terminology used at the beginning of the section. Thus $(X:Y:Z:W)$ are
homogeneous coordinates on $\mathbb{CP}^3$. For any $A,B,C,D$, the surface
$\overline{S} = \overline{S}_{A,B,C,D}$ is tangent to the hyperplane at infinity at each of the vertices
$v_1, v_2,$ and $v_3$ of $\Delta_\infty$. Again, near, say $v_1$, we can use the affine coordinates
$(Y/X,Z/X,W/X)$ on $\mathbb{CP}^3$ and express the surface so that $W/X$ is a
holomorphic function of $(Y/X,Z/X)$. Similar descriptions apply to the other two vertices,
and the resulting coordinates were called ``standard coordinates'' around the vertices in question.

Next, let $\infinityneighborhood_\infty(i)$ be an open neighborhood of $v_i$ in $\overline{S}$ so that
$\infinityneighborhood_\infty(i)$ is contained in the graph of $h_i : N_i \rightarrow \C$, $i=1,2,3$. Without loss of generality, we assume that the
three neighborhoods $\infinityneighborhood_\infty(i)$, $i=1,2,3$, are pairwise disjoint.
If $p \in \infinityneighborhood_\infty(i)$, we denote by
	${\rm dist}(p,v_i)$
the Euclidean distance between $p$ and $v_i$ in the standard coordinate chart on $\infinityneighborhood_\infty(i)$.
Let
\begin{align*}
	\infinityneighborhood_\infty = \infinityneighborhood_\infty(1) \cup \infinityneighborhood_\infty(2) \cup \infinityneighborhood_\infty(3).
\end{align*}
For any $p \in \infinityneighborhood_\infty$, we define
	${\rm dist}(p,\mathcal{V}_\infty)$
to equal ${\rm dist}(p,v_i)$ where $\infinityneighborhood_\infty(i)$ is the unique one of the three neighborhoods containing $p$.

Let $S(0)$ be a set consisting of six hyperbolic elements $\gamma_{i,j} \in \AUTOGROUP$, $i,j \in \{ 1,2,3\}$, $i\neq j$.
As in Proposition~\ref{PROP:CONV_TO_ID},
for every natural number $n$ we will consider the inductively defined
sets of iterated commutators $S(n)$.  For every $n$ the set $S(n+1)$ contains
every possible commutator of any two distinct elements of $S(n)$.

\begin{lemma}\label{LEM:ESCAPE_TO_INFINITY}
Given parameters $A, B, C, D$, assume there are six hyperbolic elements $\gamma_{i,j} \in \AUTOGROUP$ as above
such that for every pair $i\neq j \in \{ 1,2,3\}$ we have
${\rm Ind}(\gamma_{i,j}) = v_i$ and ${\rm Attr}(\gamma_{i,j}) = v_j$.
Let $S(0)$ be the set consisting of the six elements $\gamma_{i,j}$ and let $\{ S(n) \}$ be the
corresponding sequence of inductively defined subsets of $\AUTOGROUP$. Then, up to reducing the neighborhood
$\infinityneighborhood_\infty \subset \overline{S}_{A,B,C,D}$, for any point $q \in \infinityneighborhood_\infty$,
there exists a constant $0 < \lambda < 1$ and a sequence $\{\eta_n\}_{n=0}^\infty \subset \AUTOGROUP$ 
satisfying the two conditions below:
\begin{itemize}
	\item[(i)] $\eta_n \in S(n)$ for every $n \geq 0$, and
	\item[(ii)] ${\rm dist}(\eta_n(q),\mathcal{V}_\infty) \leq \lambda^{4^n}$ for every $n \geq 0$.
\end{itemize}
\end{lemma}

\begin{proof}
Note that we might have asked the elements $\gamma_{i,j}$ of the set $S(0)$ to satisfy $\gamma_{i,j} =
\gamma_{j,i}^{-1}$ though this is not necessary. Also, in view of
Definition/Proposition~\ref{PROP:DESCRIPTION_HYPERBOLIC_ELEMENTS}, all elements $\gamma_{i,j}$ are algebraically
stable and $\gamma_{i,j}$ is holomorphic around ${\rm Attr}(\gamma_{i,j}) = v_j$.

The proposition will be proved by induction. In fact,	
we will prove a stronger statement. Namely, there exists $\lambda$, $0 < \lambda < 1$,
such that for each integer $n \geq 0$, we have:
	
\vspace{0.1in}
\noindent If $n$ is even then for every pair of distinct $i,j \in \{1,2,3\}$ there exists some
$\gamma_{i,j}^{(n)} \in S(n)$ such that
\begin{itemize}
	\item[(E1)] $\gamma_{i,j}^{(n)}$ is holomorphic on $\infinityneighborhood_\infty \setminus \infinityneighborhood_\infty(i)$, satisfies
	$\gamma_{i,j}^{(n)}(\infinityneighborhood_\infty \setminus \infinityneighborhood_\infty(i)) \subset \infinityneighborhood_\infty(j)$,
	and for every $q \in \infinityneighborhood_\infty \setminus \infinityneighborhood_\infty(i)$ we
	have ${\rm dist}(\gamma_{i,j}^{(n)}(q),v_j) \leq \lambda^{4^n} \ {\rm
			dist}(q,\mathcal{V}_\infty)$.
		
	\item[(E2)] 
	$(\gamma_{i,j}^{(n)})^{-1}$ is holomorphic on $\infinityneighborhood_\infty \setminus \infinityneighborhood_\infty(j)$, 
	satisfies $(\gamma_{i,j}^{(n)})^{-1}(\infinityneighborhood_\infty \setminus \infinityneighborhood_\infty(j)) \subset \infinityneighborhood_\infty(i)$,
	and for every $q \in \infinityneighborhood_\infty \setminus
	\infinityneighborhood_\infty(j)$ we have ${\rm dist}\left((\gamma_{i,j}^{(n)})^{-1}(q),v_i\right) \leq
	\lambda^{4^n} \ {\rm dist}(q,\mathcal{V}_\infty)$.
\end{itemize}
	
\vspace{0.1in}
\noindent If $n$ is odd then for each $i \in \{1,2,3\}$ there exists some $\tau_i^{(n)} \in S(n)$ such that
\begin{itemize}
	\item[(O)] $\tau_i^{(n)}$ and $(\tau_i^{(n)})^{-1}$ are holomorphic on
	$\infinityneighborhood_\infty \setminus \infinityneighborhood_\infty(i)$, they satisfy
	$(\tau_i^{(n)})^{\pm 1}(\infinityneighborhood_\infty \setminus \infinityneighborhood_\infty(i))~\subset~\infinityneighborhood_\infty(i)$, and for any $q \in \infinityneighborhood_\infty \setminus
	\infinityneighborhood_\infty(i)$ we have ${\rm dist}\left((\tau_i^{(n)})^{\pm 1}(q),
	v_i\right) \leq \lambda^{4^n} \ {\rm dist}(q,\mathcal{V}_\infty)$.
\end{itemize}

The base of the induction is $n=0$, in which case for each pair $i \neq j$
we let $\gamma_{i,j}^{(0)} = \gamma_{i,j}$ be the corresponding element of $S(0)$.
Fix then a pair $i\neq j$ and let $k \in \{1,2,3\}$ be such that $k \neq i$ and $k \neq j$.
Consider standard local coordinates $(u_1,u_2)$ around $v_k$ and let $(w_1,w_2)$
stand for standard local coordinates in a neighborhood of $v_j$.
By hypothesis
we have ${\rm Ind}(\gamma_{i,j}) = v_i$ and ${\rm Attr}(\gamma_{i,j}) = v_j$.
Therefore, item~(iii) of Definition/Proposition~\ref{PROP:DESCRIPTION_HYPERBOLIC_ELEMENTS} gives 
that $\gamma_{i,j}(\Delta_\infty \setminus \{v_i\}) = v_j$. Hence, if $\gamma_{i,j}$ is expressed
in local coordinates under the form
$(w_1,w_2) = \gamma_{i,j}(u_1,u_2)$, both
coordinates of $\gamma_{i,j}(u_1,u_2)$ will vanish along both axes $\{u_1=0\}$ and $\{u_2 = 0\}$.
Similarly, if we express $\gamma_{i,j}$ from the $(w_1,w_2)$ coordinates to
themselves, then both coordinates of $\gamma_{i,j}(w_1,w_2)$ vanish along both axes $\{w_1=0\}$
and $\{w_2 = 0\}$.  This implies that for any $0 < \lambda < 1$, we can choose
the neighborhoods $\infinityneighborhood_\infty(k)$ and $\infinityneighborhood_\infty(j)$
sufficiently small so as to ensure that for any $q \in \infinityneighborhood_\infty(k) \cup \infinityneighborhood_\infty(j)$ the
estimate
\begin{align}\label{EQN:CONTRACTION1}
	{\rm dist}(\gamma_{i,j}(q),v_j) \leq \lambda \ {\rm dist}(q,\mathcal{V}_\infty)
\end{align}
holds.
After choosing a sufficiently small neighborhood $\infinityneighborhood_\infty(i)$ of $v_i$ and perhaps making $\infinityneighborhood_\infty(k)$ smaller, the same reasoning 
applies to show
that for any $q \in \infinityneighborhood_\infty(k) \cup \infinityneighborhood_\infty(i)$ we have
\begin{align}\label{EQN:CONTRACTION2}
	{\rm dist}\left(\gamma_{i,j}^{-1}(q),v_i \right) \leq \lambda \ {\rm dist}(q,\mathcal{V}_\infty).
\end{align}
Repeating for all six distinct pairs $i \neq j$ we obtain sufficiently small neighborhoods $\infinityneighborhood_\infty(1)$,
$\infinityneighborhood_\infty(2)$, and $\infinityneighborhood_\infty(3)$ such that (\ref{EQN:CONTRACTION1}) and (\ref{EQN:CONTRACTION2}) hold.
If we then let $\infinityneighborhood_\infty(1)$, $\infinityneighborhood_\infty(2)$, and $\infinityneighborhood_\infty(3)$ be round balls of equal sufficiently small radius in the standard local coordinates, both estimates~(\ref{EQN:CONTRACTION1})
and~(\ref{EQN:CONTRACTION2}) will
continue to hold. In addition, for all six distinct pairs $i \neq j$ we have
$\gamma_{i,j}(\infinityneighborhood_\infty \setminus \infinityneighborhood_\infty(i)) \subset \infinityneighborhood_\infty(j)$ and $\gamma_{i,j}^{-1}(\infinityneighborhood_\infty \setminus \infinityneighborhood_\infty(j)) \subset \infinityneighborhood_\infty(i)$.
Therefore we can assume that (E1) and (E2) hold when $n=0$.

\vspace{0.1in}
	
For the remainder of the proof we keep the neighborhood $\infinityneighborhood_\infty =
\infinityneighborhood_\infty(1) \cup \infinityneighborhood_\infty(2) \cup \infinityneighborhood_\infty(3)$
fixed and inductively prove that (E1) and (E2) hold on $\infinityneighborhood_\infty$ for
every even $n$ and that (O) holds on $\infinityneighborhood_\infty$ for every odd
$n$.

\vspace{0.1in}
	
Suppose now that $n$ is even and that the collection of six elements
$\gamma_{i,j}^{(n)} \in S(n)$ exist and satisfy~(E1) and~(E2). For each $i \in \{1,2,3\}$, we will prove
the existence of an element $\tau_i^{(n+1)}  \in S(n+1)$ satisfying~(O).

Fix then $i \in \{1,2,3\}$ and let $j,k \in \{1,2,3\} \setminus \{i\}$ be the other two indices. We define
\begin{align*}
	\tau_i^{(n+1)} = \left[\gamma_{i,j}^{(n)},\gamma_{i,k}^{(n)}\right] = \left(\gamma_{i,j}^{(n)}\right)^{-1} \left(\gamma_{i,k}^{(n)}\right)^{-1} \gamma_{i,j}^{(n)} \, \gamma_{i,k}^{(n)}.
\end{align*}
Using (E1) and (E2) we can see that the above composition is holomorphic on all
of $\mathcal{V}_\infty \setminus \mathcal{V}_\infty(i)$ and that it maps
$\mathcal{V}_\infty \setminus \mathcal{V}_\infty(i)$ into
$\mathcal{V}_\infty(i)$.  Moreover, for any $q \in \mathcal{V}_\infty \setminus \mathcal{V}_\infty(i)$ we have
that
\begin{align*}
	\gamma_{i,k}^{(n)}(q) \in \mathcal{V}_\infty(k), \quad \gamma_{i,j}^{(n)} \, \gamma_{i,k}^{(n)}(q) \in \mathcal{V}_\infty(j), \quad \mbox{and} \quad \left(\gamma_{i,k}^{(n)}\right)^{-1} \gamma_{i,j}^{(n)} \, \gamma_{i,k}^{(n)}(q) \in  \mathcal{V}_\infty(i).
\end{align*}
Again using (E1) and (E2) we have the each of the four mappings in the commutator used
to define $\tau_i^{(n+1)}$  contracts distance to $\mathcal{V}_\infty$ by
a factor of $\lambda^{4^n}$ and hence
\begin{align*}
	{\rm dist}\left(\tau_i^{(n+1)}(q),v_i \right) = {\rm dist}\left(\left(\gamma_{i,j}^{(n)}\right)^{-1} \left(\gamma_{i,k}^{(n)}\right)^{-1} \gamma_{i,j}^{(n)} \gamma_{i,k}^{(n)}(q),v_i \right) \leq \lambda^{4^{n+1}} {\rm dist}(q,\mathcal{V}_\infty).
\end{align*}
	
Notice that 
\begin{align*}
	(\tau_i^{(n+1)})^{-1} = \left[\gamma_{i,k}^{(n)},\gamma_{i,j}^{(n)}\right] = \left(\gamma_{i,k}^{(n)}\right)^{-1} \left(\gamma_{i,j}^{(n)}\right)^{-1} \gamma_{i,k}^{(n)} \, \gamma_{i,j}^{(n)}.
\end{align*}
Therefore the same proof as in the previous paragraph applies to $(\tau_i^{(n+1)})^{-1}$ after switching $j$ and~$k$.
We conclude that (O) holds for $n+1$.

\vspace{0.1in}
	
Suppose now that $n$ is odd and that the collection of three elements
$\tau_{i}^{(n)} \in S(n)$ exist and satisfy~(O). We will prove
that all six elements $\gamma_{i,j}^{(n+1)}  \in S(n+1)$ satisfying~(E1) and~(E2) exist.
	
For any distinct $i,j \in \{1,2,3\}$ let $k \in \{1,2,3\} \setminus \{i,j\}$ be the remaining element.
Now define
\begin{align*}
	\gamma_{i,j}^{(n+1)} = \left[\tau_j^{(n)}, \tau_i^{(n)}   \right] = \left(\tau_j^{(n)}\right)^{-1}  \left(\tau_i^{(n)}\right)^{-1} \tau_j^{(n)} \tau_i^{(n)}.
\end{align*}
Using~(O) we can see that the above composition of mappings is holomorphic on $\infinityneighborhood_\infty \setminus \infinityneighborhood_\infty(i)$ and that it maps $\infinityneighborhood_\infty \setminus \infinityneighborhood_\infty(i)$
into $\infinityneighborhood_\infty(j)$. 
Again using~(O), each of these four mappings contracts distance to $\mathcal{V}_\infty$ by a factor of $\lambda^{4^n}$
and hence 
\begin{align*}
	{\rm dist}\left(\gamma_{i,j}^{(n+1)}(q),v_j\right) &= {\rm dist}\left(\left(\tau_j^{(n)}\right)^{-1}  \left(\tau_i^{(n)}\right)^{-1} \tau_j^{(n)} \tau_i^{(n)}(q), v_j \right) \leq \lambda^{4^{n+1}}{\rm dist}(q,\mathcal{V}_\infty).
\end{align*}
We conclude that (E1) holds for $n+1$.
	
To see that (E2) holds for $n+1$, note that
\begin{align*}
	(\gamma_{i,j}^{(n+1)})^{-1} = \left[\tau_i^{(n)}, \tau_j^{(n)}   \right] = \left(\tau_i^{(n)}\right)^{-1}  \left(\tau_j^{(n)}\right)^{-1} \tau_i^{(n)} \tau_j^{(n)}.
\end{align*}
Therefore the same proof as in the previous paragraph applies to
$(\gamma_{i,j}^{(n+1)})^{-1}$ after switching $i$ and~$j$.
We conclude that (E2) holds for $n+1$.
	
Therefore, we conclude that statements (E1) and (E2) hold for every even $n \geq 0$ and that (O) holds for every odd $n \geq 0$.  
\end{proof}

\section{General properties of Fatou components and good set of parameters $\goodparams$.}\label{SEC:FATOU}

Recall that the Fatou set ${\mathcal F}_{A,B,C,D}$ for the action of
$\AUTOGROUP_{A,B,C,D}$ on $S_{A,B,C,D}$ is the set of points $p$ admitting a
neighborhood on which the restrictions of all elements in $\AUTOGROUP$ form a normal family.
Recall also that, by way of definition, this normal family may contain sequences of maps
converging to infinity, cf. Section~\ref{SUBSEC:FATOU_JULIA}. In particular, ${\mathcal F}_{A,B,C,D}$ is an
open (possibly empty) set and, according to Remark~\ref{REM:SINGULAR_POINTS_JULIA}, none
of the possible singular points of $S_{A,B,C,D}$ lies in ${\mathcal F}_{A,B,C,D}$.

A {\it Fatou component} $V \subset S_{A,B,C,D}$ is a connected component of
${\mathcal F}_{A,B,C,D}$.  Since the Fatou set is invariant, one can consider
the stabilizer $\AUTOGROUP_V \leq \AUTOGROUP$ of $V$, which consists of those elements
of $\AUTOGROUP$ that map $V$ to~$V$.  The purpose of this section is to establish
several general properties of Fatou components.  By combining these properties
with the previous material, the proofs of Theorems~H and~K will quickly be
derived in the next section.

We begin with the following dichotomy which plays an important role in the proof of Theorem~F.

\begin{proposition}\label{PROP:FATOU_DICHOTOMY}
Let $V$ be a Fatou component for $\AUTOGROUP_{A,B,C,D}$.   Then either:
\begin{itemize}
\item[(i)]  there is a sequence of mappings $\gamma_n \in \AUTOGROUP_{A,B,C,D} \setminus \{\rm id\}$ that converge uniformly on compact subsets of $V$ to the identity, or
\item[(ii)] the action of $\AUTOGROUP_{A,B,C,D}$ on $V$ is properly discontinuous.
\end{itemize}
\end{proposition}

\begin{proof}

First suppose that $V$ intersects the locally non-discrete locus $\nondiscrete$
non-trivially.  Then there exists a non-empty open set $W \subset V$ and a
sequence $\{ f_j\} \subset \AUTOGROUP$, $f_j \neq {\rm id}$ for all
$j \in \N$, such that the restrictions of the elements $f_j$ to $W$ converge
uniformly to the identity. To show that Claim~(i) holds in this case, it suffices to show that
$\{ f_{j_k}\}$ actually converges to the identity on compact subsets of $V$. This can be checked as follows.
Consider a (connected, relatively compact) set $U \subset V$ containing $W$. Owing to the fact that $\{ f_j\}$ is a normal
family on $V$, we can assume without loss of generality that $\{ f_j\}$ converges uniformly on $U$. However, by assumption, the
limit map coincides with the identity on $W$ and hence must coincide with the identity on all of $U$ as required.

Now suppose that $V$ is entirely contained in the locally-discrete locus $\discrete$.  
In this case we will prove that $\AUTOGROUP$ acts properly discontinuously on $V$.
 Let $K \subset V$ be a compact set and, aiming at
a contradiction, assume the existence of an infinite sequence $\{ f_j\}$ of pairwise
distinct elements in $\AUTOGROUP$ such that $K \cap f_j (K) \neq
\emptyset$ for all $j$. Thus, since $K$ is compact, we can find a subsequence
$j_k$ and points $p$ and $q$ in $K$ such that the sequence $\{ y_k = f_{j_k}
(p)\}_{k \in \N}$ converges to $q$.  Up to enlarging $K$, we can assume without
loss of generality that $q$ lies in the interior of $K$.  Setting $l_k =
j_{k+1} - j_k$, it follows that $g_k = f^{l_k}$ sends $y_k$ to $y_{k+1}$ and
both points converge towards $q$ as $k \rightarrow \infty$.

On the other hand, since $V$ is contained in the Fatou set, normality
implies that the derivatives of the elements $g_k$ are uniformly bounded in
$K$. Similarly, the derivatives of their inverses $g_k^{-1}$ are also bounded on $g_k (K)$.
Using the uniform bound on the derivatives of $g_k$, we conclude that
$g_k$ sends some fixed neighborhood $U_q$ of
$q$ to a bounded subset of $K$ for $k$ large enough. Again, normality implies that a subsequence
$\{g_{k(i)} \}_i$ of $\{ g_k\}$ converges uniformly on (compact subsets of)
$U_q$ to a non-constant map $g_{\infty} : U_q \rightarrow K$. However, since there is also
convergence of derivatives (Cauchy formula), the existence of uniform bounds on the derivatives of
$g_k$ and $g_k^{-1}$ implies that limit map $g_\infty$ is locally invertible.
Up to reducing the size of $U_q$ we can suppose that $g_\infty$ is actually invertible. It follows
that the sequence of maps $h_i = g_{k(i+1)}^{-1} \circ g_{k(i)}$ converges
uniformly to the identity on compact subsets of $U_q$. This contradicts the assumption that
$V$ is entirely contained in the locally-discrete locus $\discrete$.
We therefore conclude that Claim (ii) holds.
\end{proof}


\begin{proposition}\label{PROP:FATOU_COMP_ARE_HYPERBOLIC}
Any Fatou component $V$ of $\AUTOGROUP_{A,B,C,D}$ is Kobayashi hyperbolic.
\end{proposition}

\begin{proof}
Since $V$ is an open subset of $S_{A,B,C,D}$ and $V$ does not contain any singular points of
$S_{A,B,C,D}$, $V$ is itself a complex (open) manifold. Recall the ``grid''
\begin{align*}
\GRID = S_{x=x_0} \cup S_{x=x_1} \cup S_{y=y_0} \cup S_{y=y_1} \cup S_{z=z_0} \cup S_{z=z_1},
\end{align*}
defined in~(\ref{EQN:GRID}), which is a subset of the Julia set $\mathcal{J}_{A,B,C,D}$.
Hence $V$ does not intersect $\mathcal{G}$ and we can therefore consider the inclusion
\begin{align*}
\iota: V \rightarrow  \big(\mathbb{C} \setminus \{x_0,x_1\}\big) \times \big(\mathbb{C} \setminus \{y_0,y_1\}\big) \times \big(\mathbb{C} \setminus \{z_0,z_1\}\big).
\end{align*}
Clearly the image of $\iota$ is contained in a product of hyperbolic Riemann surfaces which is naturally
a Kobayashi hyperbolic domain.  The fact that holomorphic mappings do not increase the
Kobayashi psuedometric then implies that $V$ is also Kobayashi hyperbolic as well; see
\cite[Proposition 3.2.2]{KOBAYASHI}.
\end{proof}

Let then $V$ be a given component of ${\mathcal F}_{A,B,C,D}$ and let ${\rm Aut}(V)$ denote its
group of holomorphic automorphisms. By building on the general theory of topological transformation
groups of Gleason, Montgomery, and Zippin \cite{topologicaltrans-groups}, Cartan was able to show
that the automorphism group of a bounded domain in $\C^n$ is a finite-dimensional real Lie group. In turn,
Kobayashi \cite{KOBAYASHI} was able to extend Cartan's theorem to general (Kobayashi) hyperbolic manifolds.
Owing to Proposition~\ref{PROP:FATOU_COMP_ARE_HYPERBOLIC}, Corollary~\ref{COR:KOBAYASHI_CONSEQUENCES}
below summarizes these results in the case of a Fatou component.

Recall that the action $\varphi : G \times M \rightarrow M$ of a group $G$ on a manifold $M$ is said to
be {\it proper}\, if the preimage by $\varphi$ of any compact set of $M$ is again compact in $G \times M$.

\begin{corollary}\label{COR:KOBAYASHI_CONSEQUENCES}
Let $V$ be a Fatou component of $\AUTOGROUP_{A,B,C,D}$.  Then,
\begin{enumerate}
\item ${\rm Aut}(V)$ is a real Lie Group of finite dimension in the topology of uniform convergence on compact sets.
\item ${\rm Aut}(V)$ acts properly on $V$.
\item For any $p \in V$ the stabilizer ${\rm Aut}(V)_p = \{f \in {\rm Aut}(V) \, : \, f(p) = p\}$ is compact,
since the action of ${\rm Aut}(V)$ on $V$ is proper.
\end{enumerate}
\end{corollary}
\noindent
For more details, see \cite[Theorems 5.4.1 and  5.4.2]{KOBAYASHI}.

Let us now consider the stabilizer $\AUTOGROUP_V \leq \AUTOGROUP$ of the hyperbolic component $V$.  

\begin{proposition}\label{PROP_CLOSURE_GAMMA_V_IS_LIE}
Suppose that $\AUTOGROUP_{A,B,C,D}$ is locally non-discrete on an connected open
$U \subset {\mathcal F}_{A,B,C,D}$ and let $V$ be the Fatou component containing $U$. Recalling that
$\AUTOGROUP_V$ stands for the stabilizer of $V$ in $\AUTOGROUP$, the following holds:
\begin{enumerate}
\item The closure $G= \overline{\AUTOGROUP_V}$ of $\AUTOGROUP_V$ in ${\rm Aut}(V)$ is a real Lie Group of dimension at least~$1$.
\item For every point $p \in V$, the stabilizer $G_p$ of $p$ in $G$ is such that its local action around $p$
is conjugate to the local (linear) action of a closed subgroup of ${\rm SU}\, (2)$ on a neighborhood of
$(0,0) \in \C^2$.
\end{enumerate}
\end{proposition}

\begin{proof}
As a closed subgroup of a Lie Group, $G$ is itself a real Lie Group.
Moreover, since $\AUTOGROUP$ is locally
non-discrete on the open set $U \subset V$, modulo reducing $U$, there are elements
$\{\gamma_n\}_{n=1}^\infty \subset \AUTOGROUP$ that converge uniformly to the identity on $U$.
In particular, for $n$ large enough, we have $\gamma_n (V) \cap V \neq \emptyset$ and thus
$\gamma_n (V) = V$ since $\mathcal{F}$ is invariant under $\AUTOGROUP$. Hence, up to
dropping finitely many elements in the sequence in question, we can assume without loss of generality that
$\{\gamma_n\}_{n=1}^\infty \subset \AUTOGROUP_V$ for every $n \in \N$. Next we have:

\vspace{0.1in}

\noindent {\it Claim}. The sequence $\{\gamma_n\}_{n=1}^\infty$ actually converges to the
identity uniformly on compact subsets of $V$, i.e.\ as elements of ${\rm Aut}(V)$.

\begin{proof}[Proof of the claim]
Consider a relatively compact open
set $U' \subset V$ with $U \subset U'$. The claim amounts to checking that $\{\gamma_n\}_{n=1}^\infty$
converges uniformly to the identity in $U'$. If this were not the case then, up to passing to a subsequence,
there would exist $\varepsilon >0$ such that
$$
\sup_{x \in U'} \Vert \gamma_n (x) -x \Vert \geq \varepsilon >0 \, .
$$
Since $\{\gamma_n\}_{n=1}^\infty$ is contained in a normal family on $V$, we can extract a limit map
$\gamma_{\infty}$ defined on $U'$ and thus satisfying 
$\sup_{x \in U'} \Vert \gamma_{\infty} (x) -x \Vert \geq \varepsilon$ so that $\gamma_{\infty}$
does not coincide with the identity on $U'$. However, $\gamma_{\infty}$ must coincide with the identity
on $U \subset U'$ since $\{\gamma_n\}_{n=1}^\infty$ converges uniformly to the identity on $U$.
The resulting contradiction proves our claim.
\end{proof}

Since $\{\gamma_n\}_{n=1}^\infty$ converges to the identity
on compact subsets of $V$, it follows that the elements of ${\rm Aut}(V)$ obtained by restricting
them to $V$ actually converges to the identity as elements of ${\rm Aut}(V)$ equipped with its Lie group
structure, cf. Corollary~\ref{COR:KOBAYASHI_CONSEQUENCES}. Thus $\AUTOGROUP_V$ is a non-discrete subgroup
of ${\rm Aut}(V)$ and hence its closure must be a Lie group with strictly positive dimension.

It remains to check the second assertion. For this, recall
that $\AUTOGROUP$ preserves the real volume form associated with the
holomorphic volume form $\Omega$ given in~(\ref{EQN:VOLUME_FORM}). This
implies that the group of derivatives of elements of $G_p$ is a subgroup of
${\rm SL}(2,\mathbb{C})$. On the other hand, Corollary~\ref{COR:KOBAYASHI_CONSEQUENCES}, Part~(3),
informs us that $G_p$ must be compact. Since ${\rm SU}\, (2)$ is a maximal compact subgroup
\cite{BOREL} of ${\rm SL}(2,\mathbb{C})$, it follows from the classical
Bochner Linearization Theorem that the local action of $G_p$ around $p \in V$ is conjugate to the (local)
linear action of a closed subgroup of ${\rm SU}\, (2)$ on a neighborhood of the origin.  
\end{proof}


The possibility of having a point $p \in V$ whose stabilizer
$G_p$ is conjugate to all of $SU(2)$ is a challenge for us as it raises quite a few technical issues.
To avoid get involved in a much longer argument and keep us focused on the situations
of primary interest, we will work only with the following set of parameters:
\begin{align*}
\goodparams = \{(A,B,C,D) \in \mathbb{C}^4 \, : \,
\mbox{every fixed point of every element of $\AUTOGROUP_{A,B,C,D} \setminus \{{\rm
id}\}$ is in $\mathcal{J}_{A,B,C,D}$}\}.
\end{align*}
In Propositions~\ref{PROP:GOODPARAMS_IN_COMPLETMENT_HYPERSURFACES}
and~\ref{PROP:REAL_GOODPARAMS}, it will be seen that the set $\goodparams$ is ``quite large'' and
contains several parameters of interest. In particular, it will be shown that $\mathbb{C}^4 \setminus 
\goodparams$ is at worst a countable union of (proper) real-algebraic subsets of $\C^4$ and, in particular,
it has null Lebesgue measure.

First, however, we have a simple and well-known lemma.

\begin{lemma}\label{1-rigidity_kobayashimetric}
	Let $f$ be an automorphism of a Kobayashi hyperbolic domain $V$. Assume there is a point $p \in U$ that
	is fixed by $f$ and where the differential of $f$ coincides with the identity. Then $f$ is the identity
	on all of $V$.
\end{lemma}

\begin{proof}
Since $f$ is an isometry of the Kobayashi metric, the statement would be immediate if the Kobayashi
metric were a Riemannian metric which, however, is not always the case. To overcome this difficulty, we
locally replace the Kobayashi metric by the Bergman one as follows. For small \hbox{$r>0$}, let $B_r (p)$
denote the ball of radius~$r$ with respect to the Kobayashi metric. Since a holomorphic map
cannot increase the Kobayashi distance, it follows that $f (B_r (p)) \subset B_r (p)$. The analogous
argument applied to $f^{-1}$ allows us to conclude that $f$ induces an automorphism of
$B_r (p)$. Now, if $r >0$ is small enough, then $B_r (p)$ can be identified with a bounded domain
in some space $\C^n$ so that the Bergman metric is well defined. Thus $f$ induces an isometry of
the resulting Riemannian metric on $B_r (p)$ and hence coincides locally with the identity. The lemma
then follows.
\end{proof}

\begin{proposition}\label{PROP:GOODPARAMETERS_GOODSTABILIZERS}
Suppose that $(A,B,C,D) \in \goodparams$ and that $\AUTOGROUP_{A,B,C,D}$ is locally non-discrete on a
connected open $U \subset {\mathcal F}_{A,B,C,D}$. 
Let $V$ denote the Fatou component containing $U$. Then,
\begin{enumerate}
\item The closure $G= \overline{\AUTOGROUP_V}$ of $\AUTOGROUP_V$ in ${\rm Aut}(V)$ is a real Lie Group of dimension $\geq 1$.
\item The stabilizer $G_p$ of any $p \in V$ is trivial.
\end{enumerate}
\end{proposition}

\begin{proof}
Beyond the proof of Proposition~\ref{PROP_CLOSURE_GAMMA_V_IS_LIE}, it remains to show Claim~(2).
Suppose for contradiction there exists of some point $p \in V$ and some $g \in G \setminus \{{\rm id}\}$
satisfying $g(p)=p$. Note that $g$ may lie in $\overline{\AUTOGROUP_V} \setminus \AUTOGROUP_V$ so that the
statement does not follow immediately from the definition of the parameter set $\goodparams$.

Since $g \neq {\rm id}$, Lemma~\ref{1-rigidity_kobayashimetric} implies that $D g (p)$ is not the identity
either. On the other hand, $D g (p)$ is conjugate to a matrix in $SU(2)$ so that the preceding discussion
shows that $D g(p)  - {\rm id}$ is, in fact, invertible.

On the other hand, $g$ lies in $G = \overline{\AUTOGROUP_V}$ so that
there are elements $\gamma_n \in \AUTOGROUP$ with
$\gamma_n \rightarrow g$ locally uniformly on $V$. Moreover, since the functions are holomorphic,
this implies $C^\infty$ convergence on compact subsets of $V$.
Since $D g (p) - {\rm id}$ is
invertible, it follows from the implicit function theorem (for Banach spaces) that for sufficiently
large $n$ the mappings $\gamma_n$ have fixed points $p_n$ converging to $p$.
Therefore $D \gamma_n(p_n) \rightarrow Dg(p) \neq {\rm id}$, implying that $\gamma_n \neq {\rm id}$
for sufficiently large $n$.
Thus, we have found non-trivial $\gamma_n \in \AUTOGROUP$ having fixed
points in the Fatou component $V$, contradicting the choice of parameters
$(A,B,C,D) \in \goodparams$. We conclude that $G_p = \{{\rm id}\}$.
\end{proof}

%

Recall from Definition/Proposition~\ref{PROP:DESCRIPTION_HYPERBOLIC_ELEMENTS} that, bar the identity, all elements of $\AUTOGROUP$
are split in hyperbolic maps and parabolic maps. Furthermore
an element of $\AUTOGROUP$ is parabolic
if and only if it is conjugate to a non-trivial power of the generators $g_x, g_y$ or $g_z$.

\begin{lemma}\label{fixedpointparabolic_Juliaset}
	For any $(A,B,C,D)$, any fixed point of a
	parabolic element $\gamma \in \AUTOGROUP_{A,B,C,D}$ lies in the Julia set of the corresponding
	$\AUTOGROUP=\AUTOGROUP_{A,B,C,D}$-action. 
\end{lemma}

\begin{proof}
	Since parabolic maps are conjugate to a non-trivial power of one of the generators
	$g_x$, $g_y$, or $g_z$, it
	suffices to prove the statement for a non-trivial power of, say, $g_x$.  It
	follows from Proposition~\ref{PROP_ELLIPTIC_HYPERBOLIC} and
	Lemma~\ref{intersectinginfinityatdistinctpoints} that for all but finitely many
	values of $x_0 \in \mathbb{C} \setminus [-2,2]$ the action of $g_x$ on the
	fiber $S_{x=x_0}$ is loxodromic, with two distinct fixed points at infinity.
	Consider an iterate $g_x^\ell$ for some $\ell \neq 0$.  Any point on any such
	$S_{x=x_0}$ will have orbit under $g_x^\ell$ that tends to infinity.  These
	points form an open dense subset of $S_{A,B,C,D}$, implying that any  point having bounded
	orbit under $g_x^\ell$ (and hence any fixed
	point of $g_x^\ell$) must be in the Julia set.
\end{proof}

Now we will need a significantly more elaborate result.

\begin{lemma}\label{hyperbolicmaps_withhyperbolicfixedpoints}
	There is a countable union ${\mathcal H} \subset \mathbb{C}^4$ of real algebraic
	hypersurfaces such that if $(A,B,C,D) \in \mathbb{C}^4 \setminus {\mathcal H}$ 
	then $S_{A,B,C,D}$ is smooth and any hyperbolic $\gamma \in \AUTOGROUP_{A,B,C,D}$ has every fixed point consisting
	of a hyperbolic saddle point.
\end{lemma}

\begin{proof}
Let ${\rm NS} \subset \mathbb{C}^4$ be the set of parameters $(A,B,C,D)$ for
which $S_{A,B,C,D}$ is not smooth. It consists of finitely many complex
algebraic hypersurfaces and we immediately include {\rm NS} as part of~$\mathcal{H}$.

Fix a hyperbolic map $f_{A,B,C,D} \in \AUTOGROUP_{A,B,C,D}$ and recall that the
hyperbolic nature of $f_{A,B,C,D}$ depends only on its spelling in terms of the
generators $g_x$, $g_y$, and $g_z$. In particular, the notion of hyperbolic map
does not depend on the parameters $(A,B,C,D)$. Since $\AUTOGROUP$ is countable, we
can then fix the spelling of $f_{A,B,C,D}$ and reduce the proof to checking
that there are finitely many real-algebraic hypersurfaces $H \subset
\mathbb{C}^4$ away from which every fixed point of $f_{A,B,C,D}$ is a hyperbolic
saddle. The remainder of the proof consists of showing that this is, indeed,
the case.

Consider the set of of $7$-tuples $(A,B,C,D,x,y,z) \in \C^7$ and the subset
$\tilde{H} \subset \C^7$ consisting of points $(A,B,C,D,x,y,z)$ such that
\begin{enumerate}
\item $S_{A,B,C,D}$ is smooth, i.e. $(A,B,C,D) \in \C^4 \setminus {\rm NS}$,
\item $(x,y,z) \in S_{A,B,C,D}$,
\item $f_{A,B,C,D}(x,y,z) = (x,y,z),$  and
\item $Df_{A,B,C,D}(x,y,z)$ is not a hyperbolic saddle.  
\end{enumerate}
In our setting, Condition (4) is equivalent to requiring that
${\rm tr}(Df_{A,B,C,D}(x,y,z)) \in  [-2,2]$.
Therefore, $\tilde{H}$
is a semi-algebraic subset of $\mathbb{C}^7$, i.e.\ it is a set given by finitely many
polynomial equations and polynomial inequalities with real coefficients.

Notice that $H = {\rm pr}(\mathcal{H})$, where ${\rm pr}(A,B,C,D,x,y,z) =
(A,B,C,D)$. It follows from the Tarski-Seidenberg Theorem \cite[Theorem
2.2.1]{semi-algebraic} that $H$ is also semi-algebraic. By definition, the
dimension of a semi-algebraic set is the (real) dimension of the (real) Zariski
closure of the set. Therefore, if we prove that ${\rm dim}(H) \leq 7$ it will
follow that $H$ is contained in a finite union of real-algebraic hypersurfaces
of $\mathbb{C}^4$, which is sufficient for our purposes.  

The Cylindrical Algebraic Decomposition Theorem \cite[Theorem
2.3.6]{semi-algebraic} asserts that a semi-algebraic set can be decomposed into
finitely many sets, each of which is homeomorphic to $[0,1]^{d_i}$ for some
$d_i$.  Moreover, the dimension of the set (in the sense of the previous
paragraph) equals the maximum of the $d_i$, see \cite[Section
2.8]{semi-algebraic}. In particular, if ${\rm dim}(H) = 8$, then $H$ would
have non-empty interior. We will prove that this is not the case.

\vspace{0.1in}

First recall that the fixed points of $f_{A,B,C,D}$ are all isolated (see
Corollary~\ref{isolatedfixedpoints_hyperbolicmaps} or Lemma~16 in \cite{IU_ERGODIC}).
Let us first prove the following:

\vspace{0.1in}

\noindent {\it Claim}. There is an open $U \subset \mathbb{C}^4 \setminus (H \cup {\rm NS})$.

\vspace{0.1in}
\noindent

{\em Proof of the claim.} 
Consider a sequence of parameters $(A_n,B_n,C_n,Z_n) \in \C^4 \setminus {\rm
NS}$ converging to the Picard Parameters $(0,0,0,4)$. Suppose for
contradiction that for every $n$ the mapping $f_{A_n,B_n,C_n,D_n}$ has a
fixed point $p_n \in S_{A_n,B_n,C_n,D_n}$ that is not a hyperbolic saddle.  Then, since
$f_{A_n,B_n,C_n,D_n}$ preserves the volume form $\omega$, both eigenvalues of
$Df_{A_n,B_n,C_n,D_n}(p_n)$ have modulus equal to $1$.

For sufficiently large $n$ the fixed points $p_n$ remain away from some fixed
neighborhood of $\Delta_{\infty}$
(Proposition~\ref{fixedpoints_remainincompactparts}).  Therefore, we can
extract a subsequence so that $p_{n_k}$ converges to some point $p_\infty \in
S_{0,0,0,4}$.  Since $f_{A,B,C,D}$ is continuous and depends continuously on
the parameters we have that $p_\infty$ is a fixed point of $f_{0,0,0,4}$.  For
any $(A,B,C,D) \in \C^4$ let $F_{A,B,C}$ denote the extension of $f_{A,B,C,D}$
as an automorphism of $\C^3$.  The points $p_{n_k}$ and $p_\infty$ are also fixed points
for $F_{A_{n_k},B_{n_k},C_{n_k}}$ and $F_{0,0,0}$, respectively.
 For each $k$ all three of the eigenvalues of
$DF_{A_{n_k},B_{n_k},C_{n_k}}(p_{n_K})$ have modulus equal to $1$, with the
third eigenvalue corresponding to a direction transverse to
$S_{A_n,B_n,C_n,D_n}$.  Since the eigenvalues of a matrix depend
continuously on its entries and since the derivative of $D F_{A,B,C}(q)$ depends
continuously on the parameters $(A,B,C)$ and on the point $q$,
we conclude that each eigenvalue
of $DF_{0,0,0}(p_\infty)$ has modulus equal to $1$.
Since $p_\infty \in S_{0,0,0,4}$ this contradicts Corollary~\ref{COR:C3_EIGVALS}.

We conclude that there is some $n$ such that every fixed point of $f_{A_n,B_n,C_n,D_n}$
is a hyperbolic saddle.  Since we chose the parameters $(A_n,B_n,C_n,D_n) \in \C^4 \setminus \{\rm NS\}$ the surface $S_{A_n,B_n,C_n,D_n}$ is also smooth.  Both of these 
are open conditions and therefore they hold on some small neighborhood $U$ of $(A_n,B_n,C_n,D_n) \in \C^4$.
The claim follows at once.
\qed
\vspace{0.1in}

Consider now the set $\tilde{M} \subset \C^7$ consisting of $7$-tuples $(A,B,C,D,x,y,z) \in \C^7$ such that
\begin{enumerate}
\item $(x,y,z) \in S_{A,B,C,D},$
\item $f_{A,B,C,D}(x,y,z) = (x,y,z)$, and
\item $Df_{A,B,C,D}(p) - {\rm id}$ is singular.
\end{enumerate}
It is an complex algebraic subset of $\C^7$.  
The projection $M = {\rm pr}(\tilde{\rm M}) \subset \C^4$ onto the first
four coordinates is therefore {\it constructible},
see \cite{mumford}. Alternately, up to replacing the initial algebraic set by the corresponding projective
scheme, the so-called main theorem of elimination theory tells us that the resulting projection
on the coordinates $(A,B,C,D)$ yields an algebraic set $M$. In any case, the fundamental result to be used here
is the fact that the Zariski-closure of the constructible set $M$ must coincide with its closure for the
standard topology, see \cite{mumford}. Since it was shown that the complement of $M$
has non-empty interior in the standard topology, it follows that $M$ is contained in a {\it proper} Zariski-closed
subset of $\C^4$.

Consider the Zariski-open set $Z = \mathbb{C}^4 \setminus ({\rm NS} \cup \overline{M})$, where
$\overline{M}$ stands for the closure of the constructible set $M$. In particular, $Z$ is not empty.
Suppose that $(A_0, B_0,C_0,D_0) \in Z$ and 
that $p_{A_0, B_0,C_0,D_0}$ is a fixed point of $f_{A_0, B_0,C_0,D_0}$. Since $p_{A_0,
B_0,C_0,D_0}$ is simple ($Df_{A_0, B_0,C_0,D_0} (p_{A_0, B_0,C_0,D_0})
- {\rm id}$ is invertible), $p_{A_0, B_0,C_0,D_0}$ varies holomorphically with
  the parameters $(A,B,C,D)$.  In other words, we have a locally defined holomorphic
mapping $(A,B,C,D) \mapsto p_{A,B,C,D}$ for $(A,B,C,D)$ close enough to $(A_0,
B_0,C_0,D_0)$. Similarly, the differential $Df_{A,B,C,D} (p_{A,B,C,D})$ also
varies holomorphically with the parameters.

Moreover, the initial fixed point $p_{A_0, B_0,C_0,D_0}$ can actually be
(globally) continued along paths $c : [0,1] \rightarrow Z$. Indeed, as the
parameters vary in $Z$, two fixed points of $f_{A,B,C,D}$ cannot collide since
they are all simple. Furthermore, they cannot hit $\Delta_{\infty}$ either,
owing to Proposition~\ref{fixedpoints_remainincompactparts}.

\vspace{0.1in}

Suppose for contradiction that the set $H$ had non-empty interior.  
We can therefore choose some $(A_0,B_0,C_0,D_0)$ and some $\epsilon > 0$
such that $\mathbb{B}_\epsilon((A_0,B_0,C_0,D_0)) \subset H \cap Z$.
Since $f$ has finitely many fixed points we can reduce $\epsilon > 0$, if necessary,
so that there is some fixed point $p(A,B,C,D)$ varying holomorphically over
$\mathbb{B}_\epsilon((A_0,B_0,C_0,D_0))$ such that 
${\rm tr}(Df_{A,B,C,D}(p_{A,B,C,D})) \in [-2,2]$ for all $(A,B,C,D) \in \mathbb{B}_\epsilon((A_0,B_0,C_0,D_0))$.

Let $(A_1,B_1,C_1,D_1) \in U \cap Z$, where $U$ is the open set provided by the Claim above.
Consider a simple path $c : [0,1] \rightarrow Z$ with $c(0) =
(A_0,B_0,C_0,D_0)$ and $c(1) = (A_1,B_1,C_1,D_1)$.  Within
$Z$ there is a simply connected neighborhood $V$ of $c([0,1])$ on which
$p(A,B,C,D)$ and $Df_{A,B,C,D} (p_{A,B,C,D})$ vary holomorphically.
Since ${\rm tr}(Df_{A,B,C,D}(p_{A,B,C,D})) \in [-2,2]$ on an open neighborhood
of $c(0)$ the same holds on all of $V$.  
In particular, ${\rm tr}(Df_{A_1,B_1,C_1,D_1}(p_{A_1,B_1,C_1,D_1})) \in [-2,2]$,
contradicting that $(A_1,B_1,C_1,D_1) \in \mathbb{C}^4 \setminus H$.

We conclude that the (real) Zariski closure of $H$ has real dimension equal to $7$
and thus that $H$ is contained in finitely many real-algebraic hypersurfaces. The lemma is proved.
\end{proof}

We summarize these two lemmas with the following proposition.

\begin{proposition} \label{PROP:GOODPARAMS_IN_COMPLETMENT_HYPERSURFACES}
	There is a countable union of real-algebraic
	hypersurfaces ${\mathcal H} \subset \mathbb{C}^4$ 
	such that if $(A,B,C,D) \in \mathbb{C}^4 \setminus {\mathcal H}$ then every fixed point
	of every element of $\AUTOGROUP_{A,B,C,D} \setminus \{{\rm id}\}$ is in $\mathcal{J}_{A,B,C,D}$.\qed
\end{proposition}

Similarly, the argument used in Lemma~\ref{hyperbolicmaps_withhyperbolicfixedpoints} can be repeated
word-for-word to yield:

\begin{corollary}\label{family_ABCallequaltozero}
	For all but countably many $D \in \C$, every fixed point of every element in
	$\AUTOGROUP_{0,0,0,D}$, bar the identity, lies in $\mathcal{J}_{0,0,0,D}$.\qed
\end{corollary}

When $(A,B,C,D)$ are all real, there is a simple sufficient condition for
every fixed point of every element of $\AUTOGROUP_{A,B,C,D} \setminus \{{\rm id}\}$
to be in $\mathcal{J}_{A,B,C,D}$. More specifically, for real parameters $(A,B,C,D)$, the real
slice $S_{A,B,C,D}(\mathbb{R}) = S_{A,B,C,D} \cap \mathbb{R}^3$ is invariant by the action of
$\AUTOGROUP_{A,B,C,D}$ and the resulting dynamics in this real $2$-dimensional surface can be investigated
in further detail. In particular, Cantat proved in \cite[Theorem 5.1]{cantat-1}
that if the real slice $S_{A,B,C,D}(\mathbb{R}) = S_{A,B,C,D} \cap \mathbb{R}^3$ is connected, then for
any hyperbolic mapping $f$ the set of bounded orbits $\mathcal{K}_f$ of $f$ is contained in
$S_{A,B,C,D}(\mathbb{R})$ and
that $f$ is uniformly hyperbolic on ${\mathcal K}_f$. Moreover, according to Benedetto-Goldman \cite{benedettogoldman},
the real slice $S_{A,B,C,D}(\mathbb{R})$ is connected if and only if the product $ABCD < 0$ and
none of these (real) parameters lies in the interval $(-2,2)$. Taking into account 
Lemma~\ref{fixedpointparabolic_Juliaset}, the combination of Cantat's and Benedetto-Goldman's theorems
then yield:

\begin{proposition}\label{PROP:REAL_GOODPARAMS}
	If $(A,B,C,D)$ are real and $S_{A,B,C,D}(\mathbb{R})$ is connected, then then every fixed point
	of every element of $\AUTOGROUP_{A,B,C,D} \setminus \{{\rm id}\}$ is in $\mathcal{J}_{A,B,C,D}$.\qed
\end{proposition}

\section{Ruling out Fatou components: Proofs of Theorems H and K}

Let $V \subset S_{A,B,C,D}$ be a connected component of the Fatou set
${\mathcal F}_{A,B,C,D}$ and denote by $\AUTOGROUP_V \leq \AUTOGROUP$ its stabilizer.
The purpose of this section is to study the dynamics of $\AUTOGROUP_V$ on $V$ and,
in particular, to derive sufficient conditions to ensure that certain open sets
of $S_{A,B,C,D}$ must be contained in the Julia set.
We also remind the reader that $V$ contains only smooth points of
$S_{A,B,C,D}$, owing to Remark~\ref{REM:SINGULAR_POINTS_JULIA}. In our
discussion, we will have to distinguish between bounded and unbounded Fatou
component.

Let us first consider the case of unbounded Fatou components which relies heavily
on Lemma~\ref{LEM:ESCAPE_TO_INFINITY}. More generally, we resume the notation employed in
Section~\ref{dynamics_near_infinty}. Every vertex $v_i$, $i\in \{1,2,3\}$, of
$\Delta_{\infty}$ is contained in a neighborhood $\infinityneighborhood_\infty(i) \subset \overline{S}
= \overline{S}_{A,B,C,D}$ where ``standard coordinates'' are defined. The neighborhoods
$\infinityneighborhood_\infty(i)$ are assumed to be pairwise disjoint. The distance of a point in
$\infinityneighborhood_\infty(i)$ to $v_i$ is measured with the Euclidean metric arising from the ``standard
coordinates''. We then set
$\mathcal{V}_\infty = \{ v_1, v_2, v_3\}$ and 
$\infinityneighborhood_\infty = \infinityneighborhood_\infty(1) \cup \infinityneighborhood_\infty(2) \cup \infinityneighborhood_\infty(3)$.
Finally, the distance from $p \in \infinityneighborhood_\infty$ to $\mathcal{V}_\infty$ is equal to the
distance in $\infinityneighborhood_\infty(i)$ of $p$ to $v_i$ where $i \in \{1,2,3\}$ is chosen so that
$p \in \infinityneighborhood_\infty(i)$.

We are now ready to prove Theorems H and K.  We repeat the statements here for the convenience of the reader.

\begin{theoremH}
Suppose that for some parameters $A,B,C$ there is a point $p \in \mathbb{C}^3$
and $\epsilon > 0$ such that for any two  vertices $v_i \neq v_j \in
	\mathcal{V}_\infty$, $i \neq j$, there is a hyperbolic element $\gamma_{i,j} \in
	\AUTOGROUP_{A,B,C}$ satisfying:
	\begin{itemize}
		\item[(A)] ${\rm Ind}(\gamma_{i,j}) = v_i$ and ${\rm Attr}(\gamma_{i,j}) = v_j$, and
		\item[(B)] $ \sup_{z \in B_{\epsilon} (p)} \Vert \gamma_{i,j}(z) - z \Vert < K(\epsilon)$.
	\end{itemize}
	Then, for any $D$, we have that $B_{\epsilon/2}(p) \cap S_{A,B,C,D}$ is disjoint from any unbounded
	Fatou components of $\AUTOGROUP_{A,B,C,D}$.
	Here, $K(\epsilon) > 0$ denote the constant given in Proposition~\ref{PROP:CONV_TO_ID}.
\end{theoremH}

\begin{proof}
Let 
$S(0)$ be the set of all six elements $\gamma_{i,j} \in \AUTOGROUP$ satisfying the hypotheses of
Theorem~H.
As in Proposition~\ref{PROP:CONV_TO_ID} and Lemma~\ref{LEM:ESCAPE_TO_INFINITY},
for every natural number $n$ we will consider the inductively defined
sets of iterated commutators $S(n)$ where, for every $n$, the set $S(n+1)$ contains
every possible commutator of any two distinct elements of $S(n)$.

We assume aiming at a contradiction that for some parameter $D$
there is an unbounded Fatou component $V \subset S_{A,B,C,D}$ for $\AUTOGROUP_{A,B,C,D}$ such that
$V \cap B_{\epsilon/2}(p) \neq \emptyset$.
	
Let $p' \in B_{\epsilon/2}(p) \cap V$ and let $\delta > 0$ be sufficiently small
so that $B_{\delta}(p') \cap S_{A,B,C,D} \subset B_{\epsilon/2}(p) \cap V$.
Since the elements $\gamma_{i,j}$ satisfy Hypothesis~(B), it follows from Proposition~\ref{PROP:CONV_TO_ID}
that there is some integer $N \geq 0$ such that for any $\gamma \in S(N')$, with $N' \geq N$,
we have $\gamma(p') \in B_{\delta}(p')$ and hence that $\gamma(V) = V$.

For a fixed neighborhood $\infinityneighborhood_\infty$ of the vertices of
$\Delta_\infty$ as above, Lemma~\ref{LEM:ESCAPE_TO_INFINITY} gives us a
sequence of elements $\{\eta_n\}_{n=0}^\infty \in \AUTOGROUP$ satisfying
Assertions~(i) and~(ii) of the lemma in question.

We claim that $V$ intersects
$\infinityneighborhood_\infty$ non-trivially. Since $V$ is unbounded, there is a
sequence $\{q_k\}_{k=1}^\infty \subset V$ which accumulates to $\Delta_\infty$.
Passing to a subsequence, if necessary, we can suppose that it converges to
some $q_\infty \in \Delta_\infty$.  If $q_\infty \in \mathcal{V}_\infty$ then
the claim holds. Otherwise, we have that $\gamma^{(N)}_{2,1}(\Delta_\infty
\setminus \{v_2\}) = \{v_1\}$ so that $\gamma_{2,1}^{(N)}(q_\infty) = v_1$. Since
$\gamma_{2,1}^{(N)} \in S(N)$ it stabilizes $V$ and we obtain a sequence
$\{\gamma_{2,1}^{(N)}(q_k)\}_{k=1}^\infty \subset V$ that converges to $v_1$,
thus implying the claim.

Now consider some point $r \in V \cap \infinityneighborhood_\infty$.  
Because of Assertion~(i), Proposition~\ref{PROP:CONV_TO_ID} implies that 
$\{\eta_n\}_{n=0}^\infty$ converges uniformly to the identity on
the open set $B_{\delta}(p') \cap S_{A,B,C,D} \subset B_{\epsilon/2}(p) \cap V$.
In turn, since $V$ is a Fatou component, this implies
that $\{\eta_n\}_{n=0}^\infty$ actually converges uniformly to the identity on any
compact subset of $V$ (see the claim in the proof of Proposition~\ref{PROP_CLOSURE_GAMMA_V_IS_LIE}).
Applying this to the singleton set $\{r\}$, we find that
$\eta_n(r) \rightarrow r$. In contrast, Assertion~(ii) from Lemma~\ref{LEM:ESCAPE_TO_INFINITY}
implies that ${\rm dist}(\eta_n(r),\mathcal{V}_\infty) \rightarrow 0$, providing a contradiction.
	
We conclude that any Fatou component for $\AUTOGROUP_{A,B,C,D}$ that intersects $B_{\epsilon/2}(p) \cap S_{A,B,C,D}$
non-trivially must be bounded.
\end{proof}

\vspace{0.1in}

\begin{theoremK}\label{THM:CYCLIC_STABILIZERS_BOUNDED_COMPONENTS}
	Suppose that $(A,B,C,D) \in \goodparams$ and that $V$ is a bounded Fatou component
	for  $\AUTOGROUP_{A,B,C,D}$, then the stabilizer $\AUTOGROUP_V$ of $V$ is cyclic.
\end{theoremK}

It should be noticed that, in each of the examples from Section~\ref{SEC:LOCALLY_NONDISCRETE}
where we prove that $\AUTOGROUP_{A,B,C,D}$ is locally non-discrete on
some open $U \subset S_{A,B,C,D}$, the proof was carried out by producing non-trivial
elements of the sets $S(n)$ of iterated commutators that converge to the
identity on $U$ as $n$ tends to infinity.
The theorem above asserts that, if in addition we have
$(A,B,C,D) \in \goodparams$, then this set $U$ does not intersect any bounded
Fatou component of $\AUTOGROUP_{A,B,C,D}$.
Indeed, if $V$ were to be a Fatou component intersecting $U$ then for all sufficiently large
$n$ the elements of $S(n)$ would stabilize $V$ so that $\AUTOGROUP_V$ would not be Abelian.


\begin{lemma}\label{LEM:COMPACT_LIE_GROUP} Assume that $V$ is a bounded Fatou component of $\AUTOGROUP_{A,B,C,D}$.
	Then the closure $G = \overline{\AUTOGROUP_V}$ of $\AUTOGROUP_V$ in ${\rm Aut}(V)$ is a compact real Lie group. 
\end{lemma}

\begin{proof}
Since $G$ is a closed subgroup of the Lie group ${\rm Aut}(V)$ it is a real Lie group. Thus we only
have to show that $G$ is compact.

We first notice that every element of $G =  \overline{\AUTOGROUP_V}$ preserves the holomorphic volume
form $\Omega$ defined in~(\ref{EQN:VOLUME_FORM}). Indeed, by construction, $\Omega$ is invariant by
elements in $\AUTOGROUP_V$ and the condition of preserving $\Omega$ is clearly closed so that
it has to hold for the closure of $\AUTOGROUP_V$. Since $V$ is bounded, the total (real) volume of $V$
defined by means of $\Omega$ is finite. Hence, we can find a relatively compact
open set $K_0 \subset V$ such that ${\rm vol}_\Omega(K_0)
> \frac{1}{2} {\rm vol}_\Omega(V)$.  
This implies that for every $g \in G$ we have $g(K_0) \cap K_0 \neq \emptyset$.

Let $K = \overline{K_0}$. Since ${\rm Aut}(V)$ acts properly on $V$,
\begin{align*}
	\{\alpha \in {\rm Aut}(V) \, : \, \alpha(K) \cap K \neq \emptyset\}
\end{align*}
is a compact subset of ${\rm Aut}(V)$. It follows that the closed subset $G$ is compact as well.
\end{proof}

\begin{proof}[Proof of Theorem K]
As an abstract group, $\AUTOGROUP$ is isomorphic to the congruence group $\MATRIXGP_2$ that is defined
in~(\ref{definition_Gamma2}).
Since any Abelian subgroup of a non-elementary
Fuchsian group is cyclic, it suffices to prove that $\AUTOGROUP_V$ is Abelian.

We begin by pointing out that every element in $\AUTOGROUP_V$ is hyperbolic. Indeed,
parabolic elements are conjugate to one of the generators $g_x$, $g_y$, or $g_z$, 
and hence are such that every open set of $S_{A,B,C,D}$ contains points whose orbit under
(any) parabolic element is unbounded. Thus $\AUTOGROUP_V$ cannot possess parabolic elements. Finally,
it cannot possess elliptic elements either since $\AUTOGROUP$ contains no elliptic element;
see Definition/Proposition \ref{PROP:DESCRIPTION_HYPERBOLIC_ELEMENTS} and Remark \ref{REM:ELLIPTICAL_ELEMENTS}.

Now, suppose for contradiction that there are two non-commuting elements $\eta, \tau
\in \AUTOGROUP_V$. By using again the isomorphism between $\AUTOGROUP$ and $\MATRIXGP_2$, we see that
$\eta$ and $\tau$ correspond to hyperbolic elements in $\MATRIXGP_2$ which do not commute. In particular,
their iterates also do not commute since two hyperbolic elements
of $SL(2,\mathbb{Z})$ commute if and only if they have the same axes of translation in $\mathbb{H}^2$.

Owing to Lemma~\ref{LEM:COMPACT_LIE_GROUP}, the closure $G = \overline{\AUTOGROUP_V}$ is
a compact Lie Group. Since $\eta$ and $\tau$ are
hyperbolic elements, they have infinite order.  We can therefore find
subsequences of the iterates $\eta^{n_k}$ and $\tau^{n_\ell}$ converging to the
identity and such that the elements of this subsequence are pairwise different. Thus the subgroup of $G$
generated by $\eta$ and $\tau$ non-trivially accumulates on the identity which implies that
the dimension of $G$ itself as a real Lie group is strictly positive. Furthermore,
since $\eta^{n_k}$ and $\tau^{n_\ell}$ do not commute for any $k$ and $\ell$, we also have that
the identity component $G_0$ of $G$ is non-Abelian. On the other hand,
the only compact connected real Lie groups of dimension one or two being Abelian (tori), it
follows that, in fact, we have ${\rm dim}_{\mathbb{R}}(G) \geq 3$, where ${\rm dim}_{\mathbb{R}}(G)$
stands for the dimension of $G$ as real Lie group.

Conversely, the condition that the parameters $(A,B,C,D)$ are in the set $\goodparams$ implies
that for any point $p \in  V$ the orbit $G_0(p)$ of $p$ under $G_0$ is
diffeomorphic to $G_0$. In particular, ${\rm dim}_{\mathbb{R}}(G_0)~\leq~4$. However, if we had
${\rm dim}_{\mathbb{R}}(G_0) = 4$, then $G_0(p)$ would be four dimensional, implying that
$G_0(p)~=~V$. This is clearly impossible since $G_0$ is compact and $V$ is open. Summarizing,
we must have ${\rm dim}_{\mathbb{R}}(G) = 3$.

The action of $G_0$ on $V$ is smooth, proper, and free (since $(A,B,C,D)$ lies in $\goodparams$).
It follows that the quotient space $V / G_0$ can be given a
structure of a smooth manifold with $${\rm dim}_{\mathbb{R}}(V / G_0) = {\rm dim}_{\mathbb{R}}(V)
- {\rm dim}_{\mathbb{R}}(G_0) = 1$$ in such a way that the quotient map $\pi: V
\rightarrow V / G_0$ is a submersion. Thanks to the classical result of Ehresmann, this
gives $V$ the structure of a fiber bundle $V \rightarrow V / G_0$ where the
fibers are diffeomorphic to $G_0$, see, for example, \cite{LEE}. In particular, the base $V / G_0$
is of dimension~$1$.

As a one-dimensional smooth manifold, it follows that $V / G_0$ is either 
$\mathbb{S}^1$ or $\mathbb{R}$. The former case is impossible
because $V$ is non-compact, while the total space of a fiber bundle with compact base and compact fibers is compact.

We will now show that the possibility of having $V / G_0 =\mathbb{R}$ cannot occur either.
For this, note first that our
assumption on parameters implies that $S_{A,B,C,D}$ is
smooth and hence the closure $\overline{S}_{A,B,C,D}$ is smooth in
$\mathbb{CP}^3$.  It is therefore biholomorphic to the blow-up of
$\mathbb{CP}^2$ at six points, implying that $\overline{S}_{A,B,C,D}$ is simply
connected. If we choose some point $p_0 \in V$ then the orbit of $G_0$ through
$p_0$ gives an embedding of $G_0$ into ${S}_{A,B,C,D}$.  Since
$\overline{S}_{A,B,C,D}$ is simply connected, it follows that
$\overline{S}_{A,B,C,D} \setminus G_0(p_0)$ has two connected components $U_1$
and $U_2$, see, for example \cite[Proposition 7.1.1]{DAVERMAN_VENEMA}. Moreover, one of
these components, say $U_1$, contains the triangle at infinity $\Delta_\infty$ and
the other component $U_2$ is bounded in ${S}_{A,B,C,D}$.

By Theorem C, the Julia set $\mathcal{J}_{A,B,C,D}$ is connected. The Julia set is also unbounded since
it contains the fibers $S_{x_0}$ for $x_0 \in (-2,2)$ by virtue of Lemma~\ref{COR:GX_JULIA}. Therefore,
$\mathcal{J}_{A,B,C,D} \subset U_1$ and $U_2 \subset V$.

Recalling that $\pi : V \rightarrow \mathbb{R}$ stands for the bundle projection, we clearly can assume without
loss of generality that $\pi(p_0) = 0$.
The fiber bundle structure implies that $V \setminus G_0(p_0)$
has two connected components.  One of them corresponds to $U_1 \cap V$ and the other to $U_2 \subset V$.  Since $\pi$
is non-zero on each component we can suppose that $\pi(U_1 \cap V) = (0,\infty)$ and $\pi(U_2) = (-\infty,0)$.
However, notice that $U_2 \cup G_0(p_0)$ is closed and bounded (in $\mathbb{C}^3$) and hence it is compact.
This implies that $\pi$ attains a minimum value on $U_2 \cup G_0(p_0)$ which, in turn, contradicts
the fact that $\pi(U_2) = (-\infty,0)$.

We conclude that any two elements of the stabilizer $\AUTOGROUP_V$ must commute, and hence that $\AUTOGROUP_V$ is cyclic.
\end{proof}

\section{Coexistence: Proof of Theorems F and G}\label{SEC:PROOF_OF_THEOREM_G}

In this section we combine the previous results to prove Theorem F about the
coexistence of local discreteness and non-discreteness for an open
set of parameters and Theorem G about coexistence of Fatou set and Julia set
with non-empty interior for the same set of parameters, after removing
countably many real-algebraic hypersurfaces $\mathcal{H}$; see
Proposition~\ref{PROP:GOODPARAMS_IN_COMPLETMENT_HYPERSURFACES}.

\begin{proof}[Proof of Theorem F]
Lemma \ref{Firstexamples} and Proposition \ref{prop_stillDubrovinMazzocco} give
that there is an open neighborhood $\parameterneighborhood_1~\subset~\mathbb{C}^4$ that
contains $(0,0,0,0)$ and that contains each of the Dubrovin-Mazzocco parameters
$(A(a),B(a),C(a),D(a))$ for $a \in (-2,2)$ such that for each $(A,B,C,D) \in
\parameterneighborhood_1$ we have that $\AUTOGROUP_{A,B,C,D}$ is locally non-discrete on a
non-empty open set $U \subset S_{A,B,C,D}$.  Moreover, the proofs these results
are obtained by showing that for arbitrarily large $n$ there are non-trivial
elements of the sets $S(n)$ of iterated commutators of ``level $n$'' from
Proposition \ref{PROP:CONV_TO_ID}. Therefore, for each of these parameters
values we have non-commuting elements arbitrarily close to the identity on $U$.

On the other hand, Theorem E and the proof of Proposition~\ref{PROP:FATOU_C3} ensure
the existence of an open set $\parameterneighborhood_2~\subset~\mathbb{C}^4$ containing
$(0,0,0,0)$ along with each of the Dubrovin-Mazzocco parameters
$(A(a),B(a),C(a),D(a))$, with $a \in (-2,2)$, such that for each $(A,B,C,D) \in
\parameterneighborhood_2$ we have:
\begin{itemize}
\item[(i)] A point $p = (u,u,u) \in \mathbb{C}^3$ and an $\epsilon > 0$ with the
following property. For every point $q$ in the ball $\mathbb{B}_\epsilon(p)$ and
any non-trivial $\gamma \in \AUTOGROUP_{A,B,C,D}$ one of the coordinates of
$\gamma(q)$ has modulus greater than $|u| + \epsilon$.  In particular, any
non-trivial $\gamma \in \AUTOGROUP_{A,B,C,D}$ satisfies $\gamma(q) \not \in \mathbb{B}_\epsilon(p)$.
\item[(ii)] Some point $q_0 \in \mathbb{B}_\epsilon(p) \cap S_{A,B,C,D}$ which must therefore be in $\mathcal{F}_{A,B,C,D}$.
\end{itemize}

Let $V_{\rm BQ}$ be the Fatou component containing the point $q_0$ appearing in item~(ii) above.  
According to Proposition~\ref{PROP:FATOU_DICHOTOMY}, one of the following holds:
\begin{itemize}
\item[(a)]  there is a sequence of mappings $f_n \in \AUTOGROUP_{A,B,C,D} \setminus \{\rm id\}$ that converges
uniformly on compact subsets of $V_{\rm BQ}$ to the identity, or
\item[(b)] the action of $\AUTOGROUP_{A,B,C,D}$ on $V_{\rm BQ}$ is properly discontinuous.
\end{itemize}
However, Case (a) is impossible because it would follows that $f_{n}(q_0) \rightarrow q_0$ and, in particular, that $f_{n}(q_0) \in
\mathbb{B}_\epsilon(p)$ for sufficiently large~$n$.  This contradicts the preceding item (i).
We must therefore have that Case (b) holds on $V_{\rm BQ}$.

Therefore, the local non-discreteness of $\AUTOGROUP_{A,B,C,D}$ on $U$ and its local discreteness (and even the properly discontinuous action) on
$V_{\rm BQ}$ coexist in $S_{A,B,C,D}$ for every $(A,B,C,D)
\in \parameterneighborhood = \parameterneighborhood_1 \cap \parameterneighborhood_2$.  In other words, these
groups are locally non-discrete without being ``globally non-discrete''.
\end{proof}

\begin{proof}[Proof of Theorem G]
Let $\parameterneighborhood \subset \mathbb{C}^4$ be the open set of parameters
constructed in the proof of Theorem F.  For any $(A,B,C,D) \in \parameterneighborhood$
the group $\AUTOGROUP_{A,B,C,D}$ has a non-trivial Fatou component $V_{\rm BQ}
\subset S_{A,B,C,D}$, as proved in Theorem E.

Therefore, it remains to show that for any $(A,B,C,D) \in \parameterneighborhood \setminus
\mathcal{H}$ the open set $U \subset S_{A,B,C,D}$ on which
$\AUTOGROUP_{A,B,C,D}$ is locally non-discrete (from Theorem F) satisfies $U \subset
\mathcal{J}_{A,B,C,D}$.  Here, $\mathcal{H}$ is the countable union of
real-algebraic hypersurfaces provided by
Proposition~\ref{PROP:GOODPARAMS_IN_COMPLETMENT_HYPERSURFACES}.
Recall from Theorem~F that there are non-commuting pairs of elements of
$\AUTOGROUP_{A,B,C,D}$ arbitrarily close to the identity on $U$.  Therefore,
Theorem K gives that $U$ is disjoint from any bounded Fatou component of
$\AUTOGROUP_{A,B,C,D}$. In fact, to ensure that $U$ is disjoint from any bounded
Fatou component is the only place in the proof where the parameters in 
$\mathcal{H}$ need to be removed from the set
of parameters $\parameterneighborhood$.

We will now use Theorem H to show that $U$ is also disjoint from any unbounded
Fatou component, and this for every $(A,B,C,D) \in \parameterneighborhood$. Since this requires more specific
details, the discussion will be split into two cases in order to make the argument more clear.
Also, in the sequel, we are allowed to reduce the size of the open set $U$, if necessary.

\vspace{0.1in}
\noindent
{\bf Case 1}: When $(A,B,C,D)$ is sufficiently close to $(0,0,0,0)$.

\vspace{0.1in}
\noindent
We saw in the proof of Lemma \ref{Firstexamples} that 
that if $A$, $B$, and $C$ are all sufficiently close to $0$ and 
if $h_x = g_x^2, h_y=g_y^2,$ and $h_z=g_z^2$, then there is some $\epsilon > 0$ such that for any $h \in \{h_x,h_y,h_z\}$
we have
\begin{align*}
	\sup_{q \in \mathbb{B}_{\epsilon} ({\bm 0})} \Vert h(q) - q \Vert < K(\epsilon).
\end{align*}
Here, $K(\epsilon)$ is the constant from Proposition \ref{PROP:CONV_TO_ID}.
Therefore, if we let $S(0) = \{h_x,h_x^{-1},h_y,h_y^{-1},h_z,h_z^{-1}\}$ 
and define the sets $S(n)$ of iterated commutators of ``level $n$'' for each $n \geq 0$, it then 
follows from Proposition \ref{PROP:CONV_TO_ID} that for any $\gamma \in S(n)$ we have
\begin{align*}
	\sup_{q \in \mathbb{B}_{\epsilon/2} ({\bm 0})} \Vert\gamma (q) - q \Vert \leq \frac{K(\epsilon)}{2^n}.
\end{align*}
Since the relationship between $\epsilon$ and $K(\epsilon)$ given by (\ref{EQN:EPSILON_K_LINEAR_RELATION})  is linear, it follows that for each $\gamma \in S(1)$ we have
\begin{align}\label{CLOSE_TO_ID_EPS_OVER_2}
	\sup_{q \in \mathbb{B}_{\epsilon/2} ({\bm 0})} \Vert \gamma(q) - q \Vert < K(\epsilon/2).
\end{align}

Let
\begin{align*}
	\gamma_{1,2} = [h_x,h_z], \quad \gamma_{1,3} = [h_y,h_z], \quad \mbox{and} \quad \gamma_{2,3} = [h_y,h_x]
\end{align*}
and let $\gamma_{2,1} = \gamma_{1,2}^{-1}$, $\gamma_{3,1} = \gamma_{1,3}^{-1}$, and $\gamma_{3,2} = \gamma_{2,3}^{-1}$.
It is straightforward to check that these six mappings satisfy Hypothesis~(A) of Theorem~H. 
For example, one has
\begin{align*}
	\gamma_{1,2} = [h_x,h_z] = h_x^{-1} h_z^{-1} h_x h_z &= (s_y s_z s_y s_z) (s_x s_y s_x s_y) (s_z s_y s_z s_y) (s_y s_x s_y s_x) \\
	& = s_y s_z s_y s_z s_x s_y s_x s_y s_z s_y s_z s_x s_y s_x.
\end{align*}
Since the right-hand side represents a
cyclically reduced composition containing all three mappings $s_x, s_y,$ and $s_z$,
Definition/Proposition \ref{PROP:DESCRIPTION_HYPERBOLIC_ELEMENTS} implies that $\gamma_{1,2}$ is hyperbolic.
Moreover, since the first (right-most) mapping is $s_x$ we have ${\rm Ind}(\gamma_{1,2}) = v_1$
and since the last (left-most) mapping is $s_y$ we have ${\rm Attr}(\gamma_{1,2}) = v_2$.

Meanwhile, estimate (\ref{CLOSE_TO_ID_EPS_OVER_2}) implies that these six
mappings $\gamma_{i,j} \in S(1)$ satisfy Hypothesis (B) of Theorem H
on the ball of radius $\epsilon/2$.   Therefore, for all $(A,B,C,D)$ close enough to 
the origin in $\mathbb{C}^4$ the ball $\mathbb{B}_{\epsilon/2}(0) \cap S_{A,B,C,D} \subset U$ is disjoint from any unbounded Fatou component of $\AUTOGROUP_{A,B,C,D}$.

\vspace{0.1in}
\noindent
{\bf Case 2}: When $(A,B,C,D)$ is close to Dubrovin-Mazzocco parameters.

\vspace{0.1in}

\noindent
We saw in Section~\ref{SEC:LOCALLY_NONDISCRETE} that if we let
\begin{align*}
	A(a) = B(a) = C(a) = 2a + 4, \quad \mbox{and} \quad D(a) = -(a^2+8a+8)
\end{align*}
then for any $a \in (-2,2)$ the surface $S_a = S_{A(a),B(a),C(a),D(a)}$ has
three singular points $p_1(a),p_2(a),$ and $p_3(a)$ given in~(\ref{EQN:SING_PTS}). Each
of the singular points is a common fixed point of $s_x, s_y$, and~$s_z$.

Let us focus on the singular point $p_1(a)$ while pointing out that the entire discussion below applies
to $p_2(a)$ and $p_3(a)$ as well.
In the proof of Proposition~\ref{prop_stillDubrovinMazzocco} it was shown that
for any $a \in (-2,2)$ there is an $\epsilon > 0$, an open neighborhood
$W_0$ of $(A(a),B(a),C(a))$ in $\mathbb{C}^3$, and 
a sufficiently high iterate $k$ so that if we let
\begin{align*}
	f_x = g_x^k, \quad f_y = g_y^{-1} g_x^k \ g_y, \quad \mbox{and} \quad f_z = g_z^{-1} g_x^k \ g_z
\end{align*}
then for any $(A,B,C) \in W_0$ and any $f \in \{f_x,f_y,f_z\}$ we have
\begin{align*}
	\sup_{q \in \mathbb{B}_{\epsilon} (p_1(a))} \Vert f(q) - q \Vert < K(\epsilon).
\end{align*}
Let
\begin{align*}
	\gamma_{1,2} = [f_x,f_z], \quad \gamma_{1,3} = [f_y,f_z], \quad \mbox{and} \quad \gamma_{2,3} = [f_y,f_x]
\end{align*}
and let $\gamma_{2,1} = \gamma_{1,2}^{-1}$, $\gamma_{3,1} = \gamma_{1,3}^{-1}$, and $\gamma_{3,2} = \gamma_{2,3}^{-1}$.
One can then check that these six mappings satisfy Hypothesis (A) of Theorem H.
For example, one has that
\begin{align*}
	\gamma_{1,2} = [f_x,f_z] = f_x^{-1} f_z^{-1} f_x f_z &= (s_y s_z)^k \ (s_x s_y) (s_y s_z)^k (s_y s_x) \ (s_z s_y)^k \ (s_x s_y) (s_z s_y)^k (s_y s_x) \\
	&= (s_y s_z)^k \ (s_x s_z) (s_y s_z)^{k-1} (s_y s_x) \ (s_z s_y)^k \ (s_x s_y) (s_z s_y)^{k-1} (s_z s_x).
\end{align*}
This is a cyclically reduced composition containing all three mappings $s_x, s_y,$ and $s_z$ and therefore
it represents a hyperbolic mapping thanks to
Definition/Proposition~\ref{PROP:DESCRIPTION_HYPERBOLIC_ELEMENTS}.  Moreover, since the
first (right-most) mapping is $s_x$ we have ${\rm Ind}(\gamma_{1,2}) = v_1$
and since the last (left-most) mapping is $s_y$ we have ${\rm Attr}(\gamma_{1,2}) = v_2$.

As in the previous example, these six commutators may not be
$K(\epsilon)$ close to the identity on $\mathbb{B}_\epsilon(p_1(a))$.  However, we can
again use the linearity of the dependence of $K(\epsilon)$ on $\epsilon$ to see
that they satisfy Hypothesis (B) of Theorem H on $\mathbb{B}_{\epsilon/2}(p_1(a))$.  
Therefore, for all $(A,B,C,D)$ close enough to $(A(a),B(a),C(a),D(a))$
the ball $\mathbb{B}_{\epsilon/2}(p_1(a)) \cap S_{A,B,C,D} \subset U$ is disjoint from any unbounded Fatou component of $\AUTOGROUP_{A,B,C,D}$.

We have possibly reduced the size of the open set of parameters $\parameterneighborhood
\subset \mathbb{C}^4$, while still containing $(0,0,0,0)$ and still containing each of the
Dubrovin-Mazzocco parameters $(A(a),B(a),C(a),D(a))$ for $a \in (-2,2)$.  We have 
also possibly reduced the size of the open set $U$ on which $\AUTOGROUP_{A,B,C,D}$
is locally non-discrete in which a way that for any $(A,B,C,D) \in \parameterneighborhood
\setminus \mathcal{H}$ we have $U \subset \mathcal{J}_{A,B,C,D}$.

Therefore for all $(A,B,C,D) \in \parameterneighborhood \setminus \mathcal{H}$ the group
$\AUTOGROUP_{A,B,C,D}$ has a non-empty Fatou component $V_{\rm BQ} \subset
\mathcal{F}_{A,B,C,D}$ and a Julia set $\mathcal{J}_{A,B,C,D}$ with non-empty interior.
\end{proof}

\vspace{0.2in}

{\footnotesize
\noindent
Julio Rebelo \\
Institut de Math\'ematiques de Toulouse; UMR 5219, \\
Universit\'e de Toulouse, \\
 118 Route de Narbonne,\\
F-31062 Toulouse, France.\\
rebelo@math.univ-toulouse.fr.

\vspace{0.1in}

\noindent
Roland Roeder\\
 Department of Mathematical Sciences\\
Indiana University--Purdue University Indianapolis,\\
Indianapolis, IN, United States.\\
roederr@iupui.edu
}


\begin{thebibliography}{10}

\bibitem{VictorSeveralauthors}
S. Alvarez, D. Filimonov, V. Kleptsyn, D. Malicet, C. Meni\~no, A. Navas, and M. Triestino.
\newblock Groups with infinitely many ends acting analytically on the circle.
\newblock {\it Journal of Topology}, {\bf 12}, 4, 1315-1367, (2019).


\bibitem{BK}
Eric Bedford and Kyounghee Kim.
\newblock Dynamics of rational surface automorphisms: rotation domains.
\newblock {\em Amer. J. Math.}, 134(2):379--405, 2012.

\bibitem{BS}
Eric Bedford and John Smillie.
\newblock Polynomial diffeomorphisms of {${\bf C}^2$}. {II}. {S}table manifolds
  and recurrence.
\newblock {\em J. Amer. Math. Soc.}, 4(4):657--679, 1991.

\bibitem{benedettogoldman}
R. Benedetto and W. Goldman.
\newblock The topology of the relative character varieties of a
  quadruply-punctured sphere.
\newblock {\em Experiment. Math.}, 8(1), 85--103, (1999).




\bibitem{semi-algebraic}
J. Bochnak, M. Coste, and M.-F. Roy.
\newblock {\em  Real Algebraic Geometry.}
\newblock Springer, Berlin (1998).

\bibitem{BOREL}
A. Borel.
\newblock Sous-groupes compacts maximaux des groupes de {L}ie.
\newblock In {\em S\'{e}minaire {B}ourbaki, {V}ol. 1}, pages Exp. No. 33,
  271--279. Soc. Math. France, Paris, (1995).

\bibitem{bowditch}
B.H. Bowditch.
\newblock Markoff triples and quasi-Fuchsian groups.
\newblock {\em Proc. Lond. Math. Soc.}, 77, 697–736 (1998).





\bibitem{wall}
J. Bruce and C.T.C. Wall.
\newblock On the classification of cubic surfaces
\newblock {\em J. London Math. Soc.}, (2), 19, 245-256 (1979).








\bibitem{cantat-1}
S. Cantat.
\newblock Bers and {H}\'enon, {P}ainlev\'e and {S}chr\"odinger.
\newblock {\em Duke Math. J.}, 149(3), 411-460, (2009).

\bibitem{CANTAT_SURVEY}
Serge Cantat.
\newblock Dynamics of automorphisms of compact complex surfaces.
\newblock In {\em Frontiers in complex dynamics}, volume~51 of {\em Princeton
  Math. Ser.}, pages 463--514. Princeton Univ. Press, Princeton, NJ, 2014.


\bibitem{CANTAT_CREMONA}
Serge Cantat.
\newblock The {C}remona group.
\newblock In {\em Algebraic geometry: {S}alt {L}ake {C}ity 2015}, volume~97 of
  {\em Proc. Sympos. Pure Math.}, pages 101--142. Amer. Math. Soc., Providence,
  RI, 2018.


\bibitem{cantat-dujardin}
Serge Cantat and Romain Dujardin.
\newblock Random dynamics on real and complex projective surfaces.
\newblock {\em J. Reine Angew. Math.}, 802:1--76, 2023.


\bibitem{cantat-2}
S. Cantat and F. Loray.
\newblock Dynamics on character varieties and {M}algrange irreducibility of
  {P}ainlev\'e {VI} equation.
\newblock {\em Ann. Inst. Fourier (Grenoble)}, 59(7), 2927--2978, (2009).




\bibitem{DAVERMAN_VENEMA}
R. Daverman and G. Venema.
\newblock {\em Embeddings in manifolds}, volume 106 of {\em Graduate Studies in
  Mathematics}.
\newblock American Mathematical Society, Providence, RI, (2009).




\bibitem{deroin-1}
B. Deroin and R. Dujardin.
\newblock Random walks, Kleinian groups, and bifurcation currents.
\newblock {\em Invent. math.}, 190, 57-118, (2012).


\bibitem{deroin-2}
B. Deroin and R. Dujardin.
\newblock Lyapunov Exponents for Surface Group Representations.
\newblock {\em Commun. Math. Phys.}, 340, 433-469, (2015).






\bibitem{advances}
B. Deroin, V. Kleptsyn, and A. Navas.
\newblock Towards the solution of some fundamental questions
concerning group actions on the circle and codimension~one foliations.
\newblock {\it available from arXiv:1312.4133v2}.

\bibitem{CASDALGI}
Martin Casdagli. 
\newblock Symbolic dynamics for the renormalization map of a quasiperiodic Schrödinger equation. 
\newblock {\em Comm. Math. Phys.}, 107(2):295–318, 1986.


\bibitem{CONWAY}
John~B. Conway.
\newblock {\em Functions of one complex variable. {II}}, volume 159 of {\em
  Graduate Texts in Mathematics}.
\newblock Springer-Verlag, New York, 1995.

\bibitem{DGY}
David Damanik, Anton Gorodetski, and William Yessen.
\newblock The Fibonacci Hamiltonian.
\newblock Invent. math. (2016) 206:629–692.


\bibitem{DF}
J. Diller and C. Favre.
\newblock Dynamics of bimeromorphic maps of surfaces.
\newblock {\em American Journal of  Mathematics}, 123(6), 1135-1169, (2001).

\bibitem{dubrovinmazzocco}
B. Dubrovin and M. Mazzocco.
\newblock Monodromy of certain {P}ainlev\'{e}-{VI} transcendents and reflection
  groups.
\newblock {\em Invent. Math.}, 141(1), 55-147, (2000).

\bibitem{huti}
M.~H. El-Huiti.
\newblock Cubic surfaces of {M}arkov type.
\newblock {\em Mat. Sb. (N.S.)}, 93 (135), 331-346, 487, (1974).


\bibitem{anas}
A. Eskif and J. Rebelo.
\newblock Global rigidity of conjugations for locally non-discrete subgroups of ${\rm Diff}^{\omega} (S^1)$.
\newblock{\em Journal of Modern Dynamcis}, {\bf 15}, 41-93, (2019).

\bibitem{Favre}
C. Favre.
\newblock Classification of 2-dimensional contracting rigid germs and Kato surfaces, I.
\newblock {\em J. Math. Pures Appl. (9)}, 79, 475-514, (2000).



\bibitem{fischergrauert}
W. Fischer and H. Grauert.
\newblock Lokal-triviale {F}amilien kompakter komplexer {M}annigfaltigkeiten.
\newblock {\em Nachr. Akad. Wiss. G\"{o}ttingen Math.-Phys. Kl. II},
  1965, 89-94, (1965).


\bibitem{Sibony}
J. E. Fornaess and N. Sibony.
\newblock Complex dynamics in higher dimensions, II.
\newblock {\em in Modern methods in complex analysis (Princeton, NJ, 1992)}, Ann. of Math. Stud., 137,
Princeton University Press, Princeton, NJ, 135-182, (1995).

  


\bibitem{GHYS}
E. Ghys.
\newblock Sur les groupes engendr\'{e}s par des diff\'{e}omorphismes proches de
  l'identit\'{e}.
\newblock {\em Bol. Soc. Brasil. Mat. (N.S.)}, 24(2), 137-178, (1993).


\bibitem{goldman-2}
W. Goldman.
\newblock Ergodic theory on moduli spaces.
\newblock {\em Annals of Math.}, 146, 1-33, (1997).

\bibitem{goldman-1}
W. Goldman.
\newblock The modular group action on real ${\rm SL}\,(2)$-characters of a one-holed torus.
\newblock {\em Geom. Topol.}, 7, 443--486, (2003).

\bibitem{toledo}
W. Goldman and D. Toledo.
\newblock Affine cubic surfaces and relative ${\rm SL(2)}$-character varieties of compact surfaces.
\newblock {\it arXiv 1006.3838v2}, (2011).




\bibitem{gromov}
M. Gromov.
\newblock Rigid transformation group.
\newblock G\'eom\'etrie Diff\'erentielle (Travaux en cours, 33). Eds.
D. Bernard and Y. Choquet-Bruhat. Hermann, Paris, (1988).





\bibitem{horowitz}
R. Horowitz.
\newblock Induced automorphisms of Fricke characters of Free groups.
\newblock {\em Trans. AMS}, 208, 41-50, (1975).




\bibitem{Hu}
H. Hu, S.P. Tan, and Y. Zhang.
\newblock Polynomial automorphisms of $\mathbb{C}^n$ preserving the Markoff-Hurwitz polynomial.
\newblock {\em Geom. Dedicata}, 192, 207-243, (2018).




\bibitem{inaba}
M. Inaba, K. Iwasaki, and M. Saito.
\newblock Dynamics of the sixth {P}ainlev\'{e} equation.
\newblock In {\em Th\'{e}ories asymptotiques et \'{e}quations de
  {P}ainlev\'{e}}, volume~14 of {\em S\'{e}min. Congr.}, pages 103-167. Soc.
  Math. France, Paris, (2006).

\bibitem{iwasaki}
K. Iwasaki.
\newblock A modular group action on cubic surfaces and the monodromy of the
  {P}ainlev\'{e} {VI} equation.
\newblock {\em Proc. Japan Acad. Ser. A Math. Sci.}, 78(7), 131--135, (2002).

\bibitem{Iwasaki_CMP}
K. Iwasaki.
\newblock An Area-Preserving Action of the Modular Group on Cubic Surfaces and the
{P}ainlev\'{e} {VI} equation.
\newblock {\em Communications in Mathematical Physics}, {\bf 242}, 185--219, (2003).



\bibitem{gausstopainleve}
K. Iwasaki, H. Kimura, S. Shimomura, and M. Yoshida.
\newblock {\em From {G}auss to {P}ainlev\'{e}}.
\newblock Aspects of Mathematics, E16. Friedr. Vieweg \& Sohn, Braunschweig,
  (1991).

\bibitem{IU_ERGODIC}
K. Iwasaki and T. Uehara.
\newblock An ergodic study of {P}ainlev\'e {VI}.
\newblock {\em Math. Ann.}, 338(2), 295-345, (2007).

\bibitem{KOBAYASHI}
S. Kobayashi.
\newblock {\em Hyperbolic complex spaces}, volume 318 of {\em Grundlehren der
  Mathematischen Wissenschaften [Fundamental Principles of Mathematical
  Sciences]}.
\newblock Springer-Verlag, Berlin, (1998).



\bibitem{LEE_XU}
J. Lee and B.Xu
\newblock {\em Bowditch's Q-conditions and Minsky's primitive stability.}
Trans. Amer. Math. Soc. 373 (2020), no. 2, 1265–1305.



\bibitem{LEE}
J.M. Lee.
\newblock {\em Introduction to smooth manifolds}, volume 218 of {\it Graduate Texts in Mathematics}.
\newblock Springer, New York, second edition, (2013).



\bibitem{LT_ALGEBRAIC}
Oleg Lisovyy and Yuriy Tykhyy.
\newblock Algebraic solutions of the sixth painlevé equation.
\newblock {\em Journal of Geometry and Physics}, 85 124-163, (2014).


\bibitem{LORAY_REBELO}
F. Loray and J. Rebelo.
\newblock Minimal, rigid foliations by curves on {$\mathbb{C}\mathbb{P}^n$}.
\newblock {\em J. Eur. Math. Soc. (JEMS)}, 5(2), 147-201, (2003).


\bibitem{MPT}
Sara Maloni, Fr\'{e}d\'{e}ric Palesi, and Ser~Peow Tan.
\newblock On the character variety of the four-holed sphere.
\newblock {\em Groups Geom. Dyn.}, 9(3):737--782, 2015.


\bibitem{manin}
Yu. Manin.
\newblock{\em Cubic Forms: Algebra, Geometry, Arithmetics}.
\newblock{North-Holland, Elsevier, Amsterdam, Netherlands}.




\bibitem{mcmullen2}
C. McMullen.
\newblock Dynamics on K3 surfaces: Salem numbers and Siegel disks.
\newblock {\em J. reine angew. Math.} 545 (2002), 201—233

\bibitem{topologicaltrans-groups}
D. Montgomery and L. Zippin.
\newblock {\em Topological transformation groups}.
\newblock Dover edition, Mineola, New York (2018).

\bibitem{mumford}
D. Mumford.
\newblock {\em The Red Book of Varieties and Schemes}, volume 1358 of {\it Lecture Notes
in Mathematics}.
\newblock Springer-Verlag, Berlin, Heidelberg, (1999).






\bibitem{nakai}
I. Nakai.
\newblock Separatrix for Non Solvable Dynamics on
$(\C,0)$.
\newblock {\it Ann. Inst. Fourier}, 44(2), 569-599, (1994).



\bibitem{okamoto}
K. Okamoto.
\newblock Sur les feuilletages associ\'{e}s aux \'{e}quations du second ordre
  \`a points critiques fixes de {P}. {P}ainlev\'{e}.
\newblock {\em Japan. J. Math. (N.S.)}, 5(1), 1--79, (1979).

\bibitem{procesi}
C. Procesi.
\newblock The invariant theory of $n \times n$ matrices.
\newblock {\em Adv. in Math.}, 19(3), 306-381, (1976).


\bibitem{Ramis}
J.-P. Ramis
\newblock Dynamics on wild character varieties and Painlev\'e equations.
\newblock {\em Unpublished manuscript}, (2016).


\bibitem{REBELO_ENS}
J. Rebelo
\newblock Ergodicity and rigidity for certain subgroups of ${\rm Diff}^\omega(S^1)$.
\newblock {\em Ann. Sci. \'{E}cole Norm. Sup. }, 32(4) 433-453, (1999).


\bibitem{REBELO_REIS}
J. Rebelo and H. Reis.
\newblock Discrete orbits, recurrence and solvable subgroups of {${\rm
  Diff}(\mathbb{C}^2,0)$}.
\newblock {\em J. Geom. Anal.}, 27(1), 1-55, (2017).

\bibitem{RR2}
J. Rebelo and R. Roeder.
\newblock Questions about the dynamics on a natural family of affine cubic surfaces.
\newblock Preprint: \url{https://arxiv.org/pdf/2307.10962.pdf}.

\bibitem{ROBERTS}
John A. G. Roberts.
\newblock Escaping orbits in trace maps. 
\newblock {\em Phys. A}, 228(1-4):295–325, 1996.


\bibitem{segre}
B. Segre.
\newblock {\it The Non–singular Cubic Surfaces}.
\newblock Oxford University Press, (1942).



\bibitem{series}
C. Series.
\newblock The geometry of Markoff numbers.
\newblock {\em Math. Intell.}, 7, 20-29, (1985).

\bibitem{series2}
C. Series.
\newblock Primitive stability and the bowditch conditions revisited. \\
\newblock Preprint:  \url{https://arxiv.org/abs/2006.10403}.

\bibitem{STY}
C. Series, S.P. Tan, and Y. Yamashita.
\newblock The diagonal slice of Schottky space. (English summary)
\newblock Algebr. Geom. Topol. 17 (2017), no. 4, 2239–2282.





\bibitem{AASh}
A.A. Shcherbakov.
\newblock On the density of an orbit
of a pseudogroup of conformal mappings and a generalization of the
Hudai-Verenov theorem.
\newblock {\it Vestnik Movskovskogo Universiteta Mathematika}, 31(4), 10-15, (1982).

\bibitem{TWZ}
Ser~Peow Tan, Yan~Loi Wong, and Ying Zhang.
\newblock Generalized {M}arkoff maps and {M}c{S}hane's identity.
\newblock {\em Adv. Math.}, 217(2):761--813, 2008.



\bibitem{WIENHARD}
Anna Wienhard.
\newblock An invitation to higher {T}eichm\"{u}ller theory.
\newblock In {\em Proceedings of the {I}nternational {C}ongress of
  {M}athematicians---{R}io de {J}aneiro 2018. {V}ol. {II}. {I}nvited lectures},
  pages 1013--1039. World Sci. Publ., Hackensack, NJ, 2018.


\bibitem{YESSEN}
William Yessen
\newblock On the energy spectrum of 1D quantum Ising quasicrystal. 
\newblock {\em Ann. Henri Poincar\'e} 15, 419–467 (2014).


\bibitem{ZAK}
Saeed Zakeri.
\newblock {\em A course in complex analysis}.
\newblock Princeton University Press, Princeton, NJ, [2021] \copyright 2021.





\end{thebibliography}
\end{document}